\documentclass[11pt] {article}% use larger type; default would be 10ptpts
\usepackage{amsmath,amssymb,amsfonts}
\usepackage[utf8]{inputenc} % set input encoding (not needed with XeLaTeX)
\usepackage{amsthm}

\usepackage{tikz}
\usepackage{hyperref}
\usepackage{graphicx}
\usepackage{amsfonts}

\usepackage{color}
\usepackage[english]{babel}
\usepackage{ifthen}
\usepackage{graphicx}
\usepackage{tikz}
\usetikzlibrary{decorations.markings}
\usetikzlibrary{plotmarks}
\usetikzlibrary{patterns}
\usepackage{pgfplots}
\usepgfplotslibrary{fillbetween}
\usepackage{makeidx}
\usetikzlibrary{calc}

\usepackage{geometry} % to change the page dimensions
\usepackage{graphicx} 
\usepackage{array} % for better arrays (eg matrices) in maths
\usepackage{paralist} % very flexible & customisable lists (eg. enumerate/itemize, etc.)
\usepackage{verbatim} % adds environment for commenting out blocks of text & for better verbatim
\usepackage{subfig} % make it possible to include more than one captioned figure/table in a single float

\usepackage{textpos}

\newcommand{\R}{\mathbb{R}}

\newcommand{\C}{\mathbb{C}}

\newcommand{\D}{\mathbb{D}}

\newcommand{\n}{\vec{n}}
\newcommand{\s}{\mathbb{S}}
\newcommand{\Ar}{\mathring{A}}
\newcommand{\zb}{\bar{z}}

\newcommand{\Lr}{\vec{L}}
\tikzset { domaine/.style 2 args={domain=#1:#2} }

\newtheorem{theo}{Theorem}[section]
\newtheorem*{theo*}{Theorem}

\newtheorem{prop}[theo]{Proposition}
\newtheorem*{prop*}{Proposition}
\newtheorem{lem}{Lemma}[section]
\newtheorem{cor}{Corollary}[section]
\newtheorem*{cor*}{Corollary}

\newtheorem{remark}{Remark}[section]

\newcommand{\nocontentsline}[3]{}
\newcommand{\tocless}[2]{\bgroup\let\addcontentsline=\nocontentsline#1{#2}\egroup}

\title{Minimal bubbling for Willmore surfaces.}
\author{Nicolas Marque \thanks{Institut Mathématique de Jussieu, Paris VII, Bâtiment Sophie Germain, Case 7052, 75205 Paris Cedex 13, France. E-mail address : nicolas.marque@imj-prg.fr}}
\date{\today} % Activate to display a given date or no date (if empty),

\begin{document}

\maketitle

\abstract{ In this paper we  build an explicit example of a minimal bubble on a Willmore surface, showing there cannot be compactness for Willmore immersions of Willmore energy above $16 \pi$. Additionnally we prove an inequality on the second residue  for limits sequences of Willmore immersions with simple minimal bubbles. Doing so, we exclude some gluing configurations and prove compactness for immersed Willmore tori of energy below $12 \pi$.}

\tableofcontents
\hspace{0.5cm}
\section{Introduction}
The following is primarily concerned with the study of Willmore immersions in $\R^3$. Let $\Phi$ be an immersion from a closed Riemann surface $\Sigma$ into $\R^3$. We denote by $g:= \Phi^* \xi$ the pullback by $\Phi$ of the euclidean metric $\xi$ of $\R^3$, also called the first fundamental form of $\Phi$ or the induced metric. Let $d\mathrm{vol}_g$ be the volume form associated with $g$. The Gauss map $\n$  of $\Phi$ is the normal to the surface. In local coordinates $(x,y)$ : 
$$ \n := \frac{\Phi_x \times \Phi_y}{\left|\Phi_x \times \Phi_y \right|},$$
where $ \Phi_x = \partial_x \Phi$, $\Phi_y = \partial_y \Phi$ and $\times$ is the usual vectorial product in $\R^3$. 
Denoting  $\pi_{\n}$ the orthonormal projection on the normal (meaning $\pi_{\n}(v) = \langle \n , v \rangle \n$), the second fundamental form of $\Phi$ at the point $p \in \Sigma$ is defined as follows.
$$ \vec{A}_p (X,Y) := A_p(X,Y) \n :=\pi_{\n} \left( d^2 \Phi \left(X,Y \right) \right) \text{ for all } X,Y \in T_p\Sigma.$$
The mean curvature of the immersion at $p$ is then 
$$ \vec{H}(p)= H(p) \n= \frac{1}{2} Tr_g \left( A \right) \n,$$
while its tracefree second fundamental form is 
$$\Ar_p (X,Y) = A_p(X,Y)  - \frac{1}{2}H(p) g_p(X,Y).$$
The Willmore energy is  defined as $$W(\Phi) := \int_\Sigma H^2 d\mathrm{vol}_g.$$ Willmore immersions are critical points of this Willmore energy, and satisfy the Willmore equation : 
\begin{equation}
\label{lequationdeWillmoredanslintro}
\Delta_g H + \big| \Ar \big|^2 H= 0.
\end{equation}
The Willmore energy was already  under scrutiny in the XIXth century in the study of elastic plates, but to our knowledge W. Blaschke was the first to state (see \cite{MR0076373}) its invariance by conformal diffeomorphisms of $\R^3$  (which was later rediscovered by T. Willmore, see  \cite{bibwill}) and to study it in the context of conformal geometry. 

While the Willmore energy is the canonically studied Lagrangian, and serves as a natural measure of the complexity of a given immersion, its invariance is contextual. Indeed $W$ is not invariant by inversions whose center is on the surface, with the simplest example being the euclidean sphere which is sent to a plane once inverted at one of its points. The true \emph{pointwise} conformal invariant (as shown by T. Willmore, \cite{bibwill}) is in fact $\big| \Ar_p \big| d\mathrm{vol}_{g_p}$. The total curvature and tracefree curvature are then two relevant energies, respectively defined as follows : 
$$\begin{aligned}E( \Phi ) &:= \int_\Sigma \big| {A} \big|^2_g d\mathrm{vol}_g = \int_\Sigma \left|\nabla_g \n \right|^2 d\mathrm{vol}_g,\\
 \mathcal{E} ( \Phi ) &:= \int_\Sigma \big| {\Ar} \big|^2_g d\mathrm{vol}_g. \end{aligned}$$
Quick and straightforward computations (done for instance in appendix  A.1 of \cite{bibnmheps} in a conformal chart) ensure that both
\begin{equation} \label{courbtot} E ( \Phi )  = 4 W(\Phi) - 4 \pi \chi(\Sigma) \end{equation}
with $\chi(\Sigma)$ the Euler characteristic of $\Sigma$, and 
\begin{equation} \label{courbsanstrace} \mathcal{E}(\Phi) = 2W(\Phi) - 4\pi \chi(\Sigma).\end{equation}
The invariance of $W$ when the topology of the surface is not changed then follows from (\ref{courbsanstrace}). % the invariance of $\mathcal{E}$ as long as the topology of the surface is not changed.
  A Willmore surface is thus a critical point of $W$, $E$ and $\mathcal{E}$.

In the study of the moduli spaces of Willmore immersions, the compactness question has proven pivotal. E. Kuwert and R. Schätzle   (see \cite{bibkuwschat}) and later T. Rivière  (in arbitrary codimension see for instance theorem I.5 in \cite{bibanalysisaspects}) showed that Willmore immersions follow an $\varepsilon$-regularity result. These induce a now classical concentration of compactness dialectic, as originally developed by J. Sacks and K. Uhlenbeck, for Willmore surfaces with bounded total curvature (or alternatively, given (\ref{courbtot}), bounded Willmore energy and topology). In essence, sequences of Willmore surfaces converge smoothly away from concentration points, on which trees of Willmore spheres are blown (see \cite{bubbles} for an exploration of the bubble tree phenomenon in another simpler case). Y. Bernard and T. Rivière developed an energy quantization result for such sequences of Willmore immersions assuming their conformal class is in a compact of the Teichmuller space (see theorem I.2 in \cite{bibenergyquant}). P. Laurain and T. Rivière then showed one could replace  the bounded conformal class hypothesis by a weaker convergence of residues linked with the conservation laws. Since we will work with bounded conformal classes we here give abridged versions of theorems I.2 and I.3 of \cite{bibenergyquant}.% (see also  theorem 4.3 in \cite{bibpcmi}).
\begin{theo}
\label{energyquandberriv}
Let $\Phi_k$ be a sequence of Willmore immersions of a closed surface $\Sigma$. Assume that 
$$\limsup_{k\rightarrow \infty} W(\Phi_k) < \infty,$$
and that the conformal class of $\Phi^*_k \xi$ remains within a compact subdomain of the moduli space of $\Sigma$. Then modulo extraction of a subsequence, the following energy identity holds 
$$ \lim_{k\rightarrow \infty} W(\Phi_k) = W (\Phi_\infty) + \sum_{s=1}^p W(\eta_s) + \sum_{t=1}^q \left[ W( \zeta_t) - 4 \pi \theta_t \right],$$
where $\Phi_\infty$ (respectively $\eta_s$, $\zeta_t$) is a possibly branched smooth immersion of $\Sigma$  (respectively $\s^2$) and $\theta_t \in \mathbb{N}$.
%The maps $\eta_s$ and $\zeta_t$ are smooth, possibly branched, immersions of $\s^2$; and $\theta_t $ is the integer density of the current $\left( \zeta_t\right)_* \left[ \s^2\right]$ at some point $p_t = \zeta_t \left( \s^2 \right)$, namely 
%$$\theta_t := \lim_{r \rightarrow 0} \frac{ \mathcal{H}^2 \left( B_r^m(p_t) \cap \zeta_t \left( \s^2 \right) \right) }{ \pi r^2}.$$
Further there exists $a^1 \dots a^n \in \Sigma$ such that  $$ \Phi_k \rightarrow \Phi_\infty \text{ in } C^\infty_{\mathrm{loc}} \left( \Sigma \backslash \{a^1,\dots, a^n \} \right) $$ up to conformal diffeomorphisms of $\R^3\cup \{ \infty\}$.
Moreover there exists a sequence of radii $\rho^s_k$, points $x^s_k \in \C$ converging to one of the $a^i$
%, and $a^1_s \dots a^{n_s}_s \in \C$
 such that up to conformal diffeomorphisms of $\R^3$
$$\Phi_k\left( \rho^s_k y + x^s_k \right) \rightarrow \eta_s \circ \pi^{-1} (y) \text{ in } C^\infty_{\mathrm{loc}} \left( \C \backslash \{ \text{finite set} \} \right).$$
%in conformal charts around $a^i$ and 
 Finally
 there exists a sequence of radii $\rho^t_k$, points $x^t_k \in \C$ converging to one of the $a^i$
%, and $a^1_t \dots a^{n_t}_t \in \C$ 
such that up to conformal diffeomorphisms of $\R^3$
$$\Phi_k\left( \rho^t_k y + x^t_k \right) \rightarrow \iota_{p_t} \circ \zeta_t \circ \pi^{-1} (y) \text{ in } C^\infty_{\mathrm{loc}} \left( \C \backslash \{ \text{finite set} \} \right).$$ Here $\iota_{p_t}$ is an inversion at $p \in \zeta_t ( \s^2)$. The integer $\theta_t$ is the density of $\zeta_t$ at $p_t$.
\end{theo}
While theorem \ref{energyquandberriv} states an energy quantization for $W$, equality VIII.8 in \cite{bibenergyquant} offers in fact a stronger energy quantization for $E$ (and one for $\mathcal{E}$ follows). The $a^i$ are the aforementioned concentration points and the $\eta_s$ and $\iota_{p_t} \circ \zeta_t$ are the bubbles blown on those concentration points. More precisely, the $\eta_s$ are the compact bubbles, while the $\iota_{p_t} \circ \zeta_t$ are the non compact ones.  Non-compact bubbles stand out as a consequence of the conformal invariance of the problem (see \cite{MR2876249} to compare with the bubble tree extraction in the constant mean curvature framework).  One might notice that $W( \iota_{p_t} \circ \zeta_t)= W( \zeta_t) - 4 \pi \theta_t$, and deduce that if $W( \zeta_t) = 4 \pi \theta_t$, then the bubble $\iota_{p_t} \circ \zeta_t$ is minimal. This case, which we will refer to as \emph{minimal bubbling} will be of special interest to us in this article. Further if there is only one bubble at a given concentration point  we will call the bubbling \emph{simple}.  
%It has been shown in  \cite{michelatclassifi} that (branched) Willmore spheres are necessarily inversions of minimal immersions. Consequently Willmore bubbles are inversions of minimal spheres. 
Works  from Y. Li in \cite{MR3511481} (see also \cite{MR3843372}) ensure that compact simple bubbles cannot appear. These studies, furthered by P. Laurain and T. Rivière (see theorem 0.2 of \cite{MR3843372}, written just below) have yielded a compactness result for Willmore immersions of energy strictly below $12 \pi$.
\begin{theo}
\label{lacompacitesous12pistrict}
Let $\Sigma$ be a closed surface of genus $g\ge 1$ and $\Phi_k \, : \, \Sigma \rightarrow \R^3$ a sequence of Willmore immersions such that the induced metric remains in a compact set of the moduli space and 
$$\limsup_{k \rightarrow \infty } W \left( \Phi_k \right) < 12 \pi.$$
Then there exists a diffeomorphism $\psi_k$ of $\Sigma$ and a conformal transformation $\Theta_k $ of $\R^3 \cup \{ \infty \}$, such that $\Theta_k \circ \Phi_k \circ \psi_k$ converges up to a subsequence toward a smooth Willmore immersion $\Phi_\infty \, : \, \Sigma \rightarrow \R^3$ in $C^\infty \left( \Sigma \right)$.
\end{theo}
In the aforementioned paper P. Laurain and T. Rivière put forth a potential candidate for Willmore bubbling, consisting of an Enneper bubble glued on the branch point of an inverted Chen-Gackstatter torus, with an energy of exactly $12 \pi$ (see \cite{MR661204} for the definition of the Chen-Gackstatter torus). 

However before considering genus one sequences, a study of the spherical case offers interesting perspectives. Indeed in his seminal work \cite{bibdualitytheorem}, R. Bryant offered a classification of Willmore immersions of a sphere in $\R^3$, showed they were conformal transforms of minimal immersions (see theorem F), and thus that their Willmore energy was $4\pi$-quantized. Moreover while giving a complete description of the Willmore immersions of energy $16 \pi$ (part 5), R. Bryant remarked : 
\newline
{ \center "Surprisingly, this space [of Willmore immersions of energy $16 \pi$] is \emph{not} compact."}
 \vspace{5mm}
\newline
It is then interesting to consider whether one can degenerate a sequence of $16 \pi$ immersions into a bubble blown on a Willmore sphere. A quick study direct our search toward the most likely case :  a sequence degenerating into an Enneper immersion glued on the branch point of the inverse of a Lopez minimal surface. This will be our first result : 
\begin{theo}
\label{theoducontreexemle}
There exists $\Phi_k \, : \,  \s^2 \rightarrow \R^3$ a sequence of Willmore immersions such that $$W( \Phi_k )  = 16 \pi,$$ and $$\Phi_k \rightarrow \Phi_\infty,$$  smoothly on $ \s^2 \backslash \{ 0 \},$ where $\Phi_\infty$ is the inversion of a Lopez surface.
Further $$ \lim_{k \rightarrow \infty} E(\Phi_k) =   E( \Phi_\infty) + E( \Psi_\infty),$$
where $\Psi_\infty \, : \, \C \rightarrow \R^3$ is the immersion of an Enneper surface.
\end{theo}
Theorem \ref{theoducontreexemle} proves that minimal bubbles can appear and thus that Willmore immersions \emph{are not compact}. It might also indicate the possibility of gluing an Enneper bubble on an inverted Chen-Gackstatter torus. However R. Bryant's classification result proves that one cannot glue an Enneper bubble on an inverted Enneper surface  (the resulting surface would be of energy $12 \pi$, and thus limit  of Willmore immersions of equal energy, which R. Bryant showed did not exist). The local behavior of the limit surface around its branch point needs then to be constrained in order to forbid this case.  Since the Chen-Gackstatter torus and the Enneper surface are asymptotic near their branched end there is hope yet to eliminate this configuration.

%Non-compact bubbling thus remains the only simple bubbling to consider, with minimal simple bubbling being a prominent example and the main subject of the present paper. %We must remark we cannot a priori exclude non-compact non-minimal bubbling. Indeed one could imagine a minimal surface with a single branch point of order $n$ and $m>n$ simple ends without flux, inverted at the branch point to form a Willmore, but not minimal, bubble with a single end of order $n$. The existence of such bubbles must be considered a null-curve problema, and be treated with specific techniques.
The behavior of a Willmore surface around a branch point  can be fully described by an expansion, proven by Y. Bernard and T. Rivière in theorem 1.8 of \cite{bibpointremov}.
\begin{theo}
\label{theorempointremovdebernrivetvoila}
Let $\Phi \in C^\infty \left( \D \backslash \{ 0 \} \right) \cap \left( W^{2,2} \cap W^{1, \infty} \right) \left( \D \right)$ be a Willmore conformal branched immersion whose Gauss map $\n$ lies in $W^{1,2} \left( \D \right)$ and with a branch point at $0$ of multiplicity $\theta+1$. Let $\lambda$ be its conformal factor, $\vec{\gamma}_0 $ the \emph{first residue} defined as 
$$\vec{\gamma_0} := \frac{1}{4\pi} \int_{\partial \D} \vec{\nu}. \left( \nabla \vec{H} - 3 \pi_{\n} \left( \nabla \vec{H} \right) + \nabla^\perp \n \times \vec{H} \right).$$
Then there exists $\alpha \in \mathbb{Z}$ such that $\alpha \le \theta$  and locally around the origin, $\Phi$ has the following asymptotic expansion : 
$$\Phi(z) = \Re \left( \vec{A} z^{\theta+1} + \sum_{j= 1}^{\theta+1-\alpha} \vec{B}_j z^{\theta+1+j} + \vec{C}_\alpha |z|^{2\left(\theta+1\right)} z^{-\alpha} \right) - C \vec{\gamma}_0 \left( \ln |z|^{2\left( m+1\right)} - 4 \right) + \xi \left( z\right),$$
where $\vec{B}_j$, $\vec{C}_\alpha \in \C^3$ are constant vectors, $\vec{A} \in \C \backslash \{ 0 \}$, and $C \in \R$. Furthermore $\xi$ satisfies the estimates 
$$\begin{aligned}&\nabla ^j \xi (z) = O\left( |z|^{2\left( \theta+1 \right) - \alpha -j +1 - \upsilon} \right)  \text{ for all } \upsilon >0 \text{ and } j \le \theta+2- \alpha ,\\
&\left| z \right|^{-\theta} \nabla^{\theta- \alpha +3} \xi \in  L^p \text{ for all } p<  \infty.
\end{aligned}$$
In particular : 
$$ \vec{H} (z) = \Re \left( \vec{E}_\alpha \zb^{-\alpha} \right) - \vec{\gamma}_0 \ln |z| + \eta (z),$$
where $\vec{E}_\alpha \in \C^3 \backslash \{ 0 \} $. The function $\eta$ satisfies 
$$\begin{aligned} &\nabla^j \eta(z) = O \left( |z|^{1-j- \alpha - \upsilon } \right) \text{ for all } \upsilon >0 \text{ and } j \le \theta- \alpha,  \\
&|z|^\theta \nabla^{\theta+1 - \alpha} \eta \in L^p \text{ for all } p< \infty. \end{aligned}$$
\end{theo}
In the specific case of limits of Willmore immersions,  the punctured disk described in theorem \ref{theorempointremovdebernrivetvoila} is in fact the limit of simply connected disks, on which the first residue is null (see remark 1.1 in \cite{MR3843372}). Since away from the concentration point $\nabla \vec{H}^k - 3 \pi_{\n^k} \left( \nabla \vec{H}^k \right) + \nabla^\perp \n^k \times \vec{H}^k$ converges, the first residue around branch points of limit Willmore surfaces is always null. Such surfaces are called \emph{true} Willmore surfaces. The quantity $\alpha$, although called the \emph{second residue} (see definition 1.7 in \cite{bibpointremov}), is not actually a residue and is thus not necessarily null. It will then take center stage in the study of limit Willmore surfaces.
 With this tool we can refine our understanding of the behavior around minimal concentration points with the following theorem : 
%From this we deduce an immediate corollary of theorem \ref{lesecondresiduestmieuxintro} : 
\begin{theo}
\label{lecorintro}
Let $\Phi_k$ be a sequence of Willmore immersions of a closed surface $\Sigma$ satisfying the hypotheses of theorem \ref{energyquandberriv}. 
Then at each concentration point $p \in \Sigma$ of multiplicity $\theta_p+1$ on which a simple minimal bubble is blown, the second residue $\alpha_p$ of the limit immersion $\Phi_\infty$ satisfies $$\alpha_p \le \theta_p -1.$$ 
\end{theo}
%Further,  the proof of theorem \ref{lecorintro} will in fact achieve a uniform control of $H^\varepsilon$, which we will be able to exploit in an improvement of theorem 1.3 of \cite{bibnmheps}  : 
%\begin{theo}
%\label{laconvergenceamelioreethbjni}
%Let $\Phi^k \, : \,\Sigma  \rightarrow \R^3$ be a sequence of Willmore immersions satisfying the hypotheses of theorem  \ref{energyquandberriv}. Assume further that at each concentration point a simple minimal bubble is blown. Then $\Phi^k \rightarrow \Phi^0$ $C^{3, \eta}$ for all $\eta <1$.
%\end{theo}
One should be aware that minimal bubbling is not necessarily simple. Indeed one could for instance imagine an Enneper surface bubbling on the branch point  of a minimal surface of Enneper-Weierstrass data $(f, g) = (z^2, z)$, itself glued on a branch point of multiplicity $5$. However piling minimal spheres that way increases the total multiplicity. Minimal bubbling on branch points  of multiplicity $3$ is thus simple. Consequently theorem \ref{lecorintro} allows us to eliminate some surfaces as a support for minimal bubbling.
\begin{cor}
\label{lecaschengackstatter}
The convergence of Willmore immersions cannot lead to a minimal bubble and an inverted Chen-Gackstatter torus.
\end{cor}   
We can now extend theorem \ref{lacompacitesous12pistrict} :
\begin{theo}
\label{lacompacitesous12pilarge}
Let $\Sigma$ be a closed surface of genus $ 1$ and $\Phi_k \, : \, \Sigma \rightarrow \R^3$ a sequence of Willmore immersions such that the induced metric remains in a compact set of the moduli space and 
$$\limsup_{k \rightarrow \infty } W \left( \Phi_k \right) \le 12 \pi.$$
Then there exists a diffeomorphism $\psi_k$ of $\Sigma$ and a conformal transformation $\Theta_k $ of $\R^3 \cup \{ \infty \}$, such that $\Theta_k \circ \Phi_k \circ \psi_k$ converges up to a subsequence toward a smooth Willmore immersion $\Phi_\infty \, : \, \Sigma \rightarrow \R^3$ in $C^\infty \left( \Sigma \right)$.
\end{theo}
\begin{proof}
We only have to exclude the  case \begin{equation} \label{jqlsdkgnvslkgsrljkgvngvnwddlv}\displaystyle{\limsup_{k \rightarrow \infty } W \left( \Phi_k \right) = 12 \pi}.\end{equation} Consider then  $\Phi_k $ satisfying \eqref{jqlsdkgnvslkgsrljkgvngvnwddlv} and converging toward $\Phi_\infty$ away from a finite number of concentration points.  We consider a concentration point and reason on its multiplicity $\theta_0 +1$. If $\theta_0 \ge 1$,  using corollary \ref{uiompkikljioj} (see below), the bubble glued on its concentration point is branched, with the same multiplicity. Using proposition C.1 in \cite{MR3843372} ensures that the multiplicity is odd, and then  $\theta_0 \ge 2$. Given P. Li and S. Yau's inequality (see \cite{MR674407}) and \eqref{jqlsdkgnvslkgsrljkgvngvnwddlv}, $\Phi_\infty$ has a Willmore energy  of exactly $12 \pi$ meaning that the branch point is of multiplicity exactly $3$, that the bubbles have no Willmore energy (i.e. they are minimal and more accurately Enneper). Using formulas from \cite{biblammnguyen}, detailed in appendix (see propositions \ref{gaussbobonnet} and \ref{laregledesegalitesdvjno}), $ \Phi_\infty$ is the inverse of a minimal torus of total curvature $-8 \pi$. The main result of \cite{MR1058433} ensures that this minimal torus is a Chen-Gackstatter immersion. We are then in the case excluded by corollary \ref{lecaschengackstatter}.

If the concentration point is not branched, we refer the reader to the concluding remark of P. Laurain and T. Rivière's \cite{MR3843372} (found just before the appendix) which states that the energy is then at least $\beta_1 + 12 \pi$, where $\beta_1$ is the infimum of the Willmore energy of Willmore tori. We would then be above our $12 \pi$ ceiling, which concludes the proof.
% we consider the first bubble to be glued on it, and denote it $\Phi_1$. Let $\Psi_1$ be an inversion of $\Phi_1$ such that it is compact.  At each other end of  $\Phi_1$ of multiplicity $\theta +1$  a compact part of the bubble tree is glued, which, according to P. Li and S. Yau's inequality has at least $4 \pi \left( \theta +1 \right)$ of Willmore energy. Necessarily, given  \eqref{jqlsdkgnvslkgsrljkgvngvnwddlv},  there is at most one additional end. If there is no other end, and no branch point on $\Phi_1$, then $\Psi_1$ is the inversion of a minimal surface and cannot be a sphere. Then $W(\Psi_1)\ge 12 \pi$, which means that $W(\Phi_1) \ge 8 \pi$, which contradicts  \eqref{jqlsdkgnvslkgsrljkgvngvnwddlv}. If there is no other end but branch points, those are of order at least $3$ (for the same reasons as before), which implies  $W(\Psi_1)\ge 12 \pi$ and leads to a contradiction.
%If there is one other end, we see with the same reasonings that $\Phi_1$ has at least $4\pi$ of Willmore energy, the branch of the bubble tree glued on the end has at least $4 \pi$, and along with  $W( \Phi_\infty) > 4 \pi$, we find ourselves above our threshold.   Having excluded this last case concludes the proof.
\end{proof}
The compactness of Willmore tori could fail with an energy  strictly above $12 \pi$. One could imagine a sequence of non conformally minimal tori (similar to the ones described by U. Pinkall in \cite{MR799274}) of Willmore energy $12 \pi + \delta$ which degenerates into a branched  torus of same energy, with an Enneper bubble. 
To avoid contradicting theorem \ref{lecorintro}, the branch point would need a second residue $\alpha \le 1$. The existence of such a true Willmore torus is key in understanding compactness above $12 \pi$. Further one must notice that under the conclusion of theorem  \ref{lecorintro}, A. Michelat and T. Rivière, in \cite{michelatclassifi}, have proven that the Bryant's quartic, denoted $\mathcal{Q}$, of the limit surface is then holomorphic, meaning constant in the torus case.  Since, according to R. Bryant's \cite{bibdualitytheorem}, $\mathcal{Q} =0$ implies that the surface is the inversion of a minimal surface, one could hope to push A. Michelat and T. Rivière's reasoning  to the next order \emph{in the special case of a minimal bubble} and conclude that the limit surface is an inversion of a minimal surface. Its energy would then  be at least $16 \pi$, which would get us closer to the compactness strictly below $16 \pi$ to surfaces of genus lower than $1$. The possible counter example would be a $12 \pi$ branched minimal torus on which a non conformally minimal, non compact Willmore sphere is blown.

To extend the compactness below $12 \pi$ to immersions of a higher genus, it would be enough to show that the only minimal immersions of critical curvature are asymptotic to the Enneper surface near their branched end. This would be in agreement with the conjecture that the Chen-Gackstatter immersions are the only one with critical curvature for a given genus.

Section \ref{lasection1} will prove theorem \ref{theoducontreexemle} and build an exemple of minimal bubbling. Then section \ref{lasection2} will be devoted to translating theorem \ref{lecorintro} in local conformal charts and to the possible adjustments, in notation or with the conformal group, that can be done to simplify the problem. Section \ref{lasection3} will give the first expansions on the conformal factor. Section \ref{lasection4} will prove theorem \ref{lecorintro} and its corollaries.

{\bf Acknowledgments: } The author would like to thank his advisor Paul Laurain for his support and precious advices.
This work was partially supported by the ANR BLADE-JC.

% and section \ref{lasection5} will conclude with the study of Enneper bubbles and the proof of theorems \ref{theoremepresquefinalrin} and \ref{rgfqlhrfeointro}.
\section{Proof of theorem \ref{theoducontreexemle} }
\label{lasection1}
\begin{proof}
We will build a sequence of Willmore immersions  whose energy $E$ concentrates on a point where an Enneper bubble blows up.
Working from section 5 of R. Bryant's \cite{bibdualitytheorem}, we study a family of four ended minimal immersions  $\Psi_\mu \, : \, \C \backslash \{ a_1, a_2 , a_2 \} \rightarrow \R^3 $ : 
\begin{equation}
\label{ladefinition}
\begin{aligned}
\Psi_\mu &= 2 \Re \left( f_\mu \right) \\
f_\mu &= \frac{a_1}{z - \mu} + \frac{a_2}{z - \mu j} + \frac{a_3}{z- \mu j^2 } + a_4 z,
\end{aligned}
\end{equation}
with $a_1$, $a_2$, $a_3$, $a_4 \in \C^3$, $j^3 =1$, and $\mu$  a real parameter that will go toward $0$. As explained  in \cite{bibdualitytheorem}  the $(a_i)$ must be constrained for $\Psi_\mu$ to be a conformal immersion. Indeed : 
$$
\begin{aligned}
\left\langle \left(  \Psi_\mu \right)_z, \left(  \Psi_\mu \right)_z \right\rangle &= \left\langle \left(  f_\mu \right)_z, \left(  f_\mu \right)_z \right\rangle \\
&=\frac{\left\langle a_1, a_1 \right\rangle }{ \left( z- \mu \right)^4 } + \frac{\left\langle a_2, a_2 \right\rangle }{ \left( z- \mu j \right)^4 } +  \frac{\left\langle a_3, a_3 \right\rangle }{ \left( z- \mu j^2 \right)^4 } + \left\langle a_4, a_4 \right\rangle \\
&+\frac{ 2 \left\langle a_1, a_2 \right\rangle }{ \left( z- \mu \right)^2 \left( z - \mu j \right)^2 }+\frac{ 2 \left\langle a_1, a_3 \right\rangle }{ \left( z- \mu \right)^2 \left( z - \mu j^2 \right)^2 } +\frac{ 2 \left\langle a_2, a_3 \right\rangle }{ \left( z- \mu j \right)^2 \left( z - \mu j^2  \right)^2 } \\
&- \frac{2 \left\langle a_1, a_4 \right\rangle }{\left( z- \mu \right)^2}- \frac{2 \left\langle a_2, a_4 \right\rangle }{\left( z- \mu j \right)^2} - \frac{2 \left\langle a_3, a_4 \right\rangle }{\left( z- \mu j^2 \right)^2}.
\end{aligned}
$$
Further since given $u,v \in \C$ : 
$$
\begin{aligned} 
\frac{1}{\left( z-u \right)^2 \left( z- v \right)^2 } &= \frac{1}{ \left( u - v\right)^2 } \frac{1}{ (z-u)^2} + \frac{1}{ \left( u - v\right)^2 } \frac{1}{ (z-v)^2} - \frac{2}{(u-v)^3} \frac{1}{z-u} +\frac{2}{(u-v)^3 } \frac{1}{z-v},
\end{aligned}$$
we deduce that $\left\langle \left(  \Psi_\mu \right)_z, \left(  \Psi_\mu \right)_z \right\rangle = 0$ if and only if 
\begin{equation} \label{463546382636}\begin{aligned}
\left\langle a_1, a_1 \right\rangle &=\left\langle a_2, a_2 \right\rangle =\left\langle a_3, a_3 \right\rangle =\left\langle a_4, a_4 \right\rangle =0, \\
\left\langle a_1, a_2 \right\rangle &=\left\langle a_1, a_3 \right\rangle =\left\langle a_2, a_3 \right\rangle, \\
a_4 &= - \frac{1}{3 \mu^2} \left( a_1 + j a_2 + j^2 a_3 \right).
\end{aligned} \end{equation}
One can check that under the conditions \eqref{463546382636},  $(a_1,a_2,a_3)$ is a linearly independant family of $\C^3$ and thus that $\Psi_\mu$ is an immersion.

Here we take, with $b\in \C$ a parameter to be adjusted later,  $$ \begin{aligned} a_1&= \frac{1}{2 \mu^2} \begin{pmatrix} 1 \\ i \\ 0 \end{pmatrix}, \\
a_2 &=  \frac{j}{2 \mu^2} \begin{pmatrix} 1 \\ i \\ 0 \end{pmatrix} - \frac{ \mu^2 b^2}{2} \begin{pmatrix} 1 \\ -i \\ 0 \end{pmatrix}  + bj^2 \begin{pmatrix} 0 \\ 0 \\ 1 \end{pmatrix}, \\
a_3 &=  \frac{j^2}{2 \mu^2} \begin{pmatrix} 1 \\ i \\ 0 \end{pmatrix} - \frac{ \mu^2 b^2}{2} \begin{pmatrix} 1 \\ -i \\ 0 \end{pmatrix}  -bj \begin{pmatrix} 0 \\ 0 \\ 1 \end{pmatrix}. \end{aligned}$$
One can check that these $(a_i)$ satisfy \eqref{463546382636}.
Computing, we find :
$$f_\mu = \frac{3}{z^3 - \mu^3 } \frac{1}{2}\begin{pmatrix} 1 \\ i \\ 0 \end{pmatrix}  -b^2  \left( \mu^2 \frac{2z + \mu}{z^2 + \mu z + \mu^2} + \frac{ z }{3} \right) \frac{1}{2} \begin{pmatrix} 1 \\-i \\ 0 \end{pmatrix} + b j (j-1) \frac{ z + \mu }{ z^2 + \mu z + \mu^2 } \begin{pmatrix} 0 \\ 0 \\ 1 \end{pmatrix}.$$ To simplify this expression we set $b = \frac{3a}{2j (j-1)}$ with $a \in \C$ to be fixed at the end of the reasoning, and reach : 
\begin{equation}
\label{240520191849}
f_\mu = \frac{3}{z^3 - \mu^3 } \frac{1}{2}\begin{pmatrix} 1 \\ i \\ 0 \end{pmatrix}  +\frac{a^2}{4}  \left(3 \mu^2 \frac{2z + \mu}{z^2 + \mu z + \mu^2} + z  \right) \frac{1}{2} \begin{pmatrix} 1 \\-i \\ 0 \end{pmatrix} + \frac{3a}{2} \frac{ z + \mu }{ z^2 + \mu z + \mu^2 } \begin{pmatrix} 0 \\ 0 \\ 1 \end{pmatrix}.
\end{equation}
 Then $\Psi_\mu \, : \, \s^2 \rightarrow \R^3 $ is a sequence of minimal immersions with four simple planar ends. Applying propositions \ref{gaussbobonnet} and \ref{laregledesegalitesdvjno} we find : 
\begin{equation}
\int_{\s^2} K_{\Psi_\mu} d \mathrm{vol}_{g_{\Psi_\mu} } = -12 \pi,
\end{equation}
\begin{equation}
\label{78789868656565654652665264163253252326563265239625}
\int_{\s^2} \big| \Ar \big|^2_{\Psi_\mu} d \mathrm{vol}_{g_{\Psi_\mu} } = 24 \pi.
\end{equation}
Letting $\mu \rightarrow 0$ in \eqref{240520191849} we find that, away from $0$,
$$ f_\mu \rightarrow f_0 = \frac{3}{2 z^3} \begin{pmatrix} 1 \\ i \\ 0 \end{pmatrix} + \frac{a^2z}{8}  \begin{pmatrix} 1 \\ -i \\ 0 \end{pmatrix} + \frac{3a}{2z} \begin{pmatrix} 0 \\ 0 \\1 \end{pmatrix},$$
and deduce that $\Psi_\mu \rightarrow \Psi_0 := 2 \Re ( f_0)$ smoothly away from $0$, where $\Psi_0$ is a branched minimal immersion of the sphere with one simple planar end and one planar end of multiplicity $3$.  This immersion is in fact the Lopez minimal surface mentioned in   theorem \ref{theoducontreexemle}.  Then
\begin{equation}
\int_{\s^2} K_{\Psi_0} d \mathrm{vol}_{g_{\Psi_0} } = -8 \pi,
\end{equation}
\begin{equation}
\label{ehksdjnilkookmcjkqsl}
\int_{\s^2} \big| \Ar \big|^2_{\Psi_0} d \mathrm{vol}_{g_{\Psi_0} } = 16 \pi.
\end{equation}
Let $p $ be a point in $\R^3$ such that $d( p , \Psi_\mu) > 1$. We now introduced $\Phi_\mu := \iota_p \circ \Psi_\mu$ and $\Phi_0 := \iota_p \circ \Psi_0$, with $\iota (x) = \frac{x-p}{|x-p|^2}$ the inversion in $\R^3$ centered at $p$. Then $\Phi_\mu$ is a sequence of  closed Willmore conformal immersions of the sphere converging toward $ \Phi_0$ smoothly away from $0$, and $\Phi_0$ is a closed Willmore conformal branched immersion of the sphere with a single branch point of multiplicity $3$ at $0$.
Thus
\begin{equation}
\label{bqndsccd}
\int_{\s^2} K_{\Phi_\mu} d \mathrm{vol}_{g_{\Phi_\mu} } = 4 \pi,
\end{equation}
\begin{equation}
\label{2405202191829}
\int_{\s^2} K_{\Phi_0} d \mathrm{vol}_{g_{\Phi_0} } = 8 \pi.
\end{equation}
Since $\big| \Ar \big|^2 d \mathrm{vol}_g$ is a conformal invariant, we deduce from \eqref{78789868656565654652665264163253252326563265239625} and \eqref{ehksdjnilkookmcjkqsl} : 
\begin{equation}
\label{78789868656565654652665264163253252326563265239625bis}
\int_{\s^2} \big| \Ar \big|^2_{\Phi_\mu} d \mathrm{vol}_{g_{\Phi_\mu} } = 24 \pi,
\end{equation}
\begin{equation}
\label{ehksdjnilkookmcjkqslbis}
\int_{\s^2} \big| \Ar \big|^2_{\Phi_0} d \mathrm{vol}_{g_{\Phi_0} } = 16 \pi.
\end{equation}
With proposition  \ref{laregledesegalitesdvjno}  we conclude with \eqref{bqndsccd} and \eqref{78789868656565654652665264163253252326563265239625bis} :
\begin{equation}
\int_{\s^2} H^2_{\Phi_\mu} d \mathrm{vol}_{g_{\Phi_\mu}} = \frac{1}{2}\int_{\s^2} \big| \Ar \big|^2_{\Phi_\mu} d \mathrm{vol}_{g_{\Phi_\mu} }  + \int_{\s^2} K_{\Phi_\mu} d \mathrm{vol}_{g_{\Phi_\mu} } = 16 \pi, 
\end{equation}
and with \eqref{2405202191829} and \eqref{ehksdjnilkookmcjkqslbis} :
\begin{equation}
\label{240520191835}
\int_{\s^2} H^2_{\Phi_0} d \mathrm{vol}_{g_{\Phi_0}} = \frac{1}{2}\int_{\s^2} \big| \Ar \big|^2_{\Phi_0} d \mathrm{vol}_{g_{\Phi_0} }  + \int_{\s^2} K_{\Phi_0} d \mathrm{vol}_{g_{\Phi_0} } = 16 \pi.
\end{equation}
Comparing \eqref{78789868656565654652665264163253252326563265239625bis}-\eqref{240520191835} reveals that while :
$W(\Phi_\mu) \rightarrow W (\Phi_0)$, there is an energy gap of $8 \pi$ in $\mathcal{E}$ (or equivalently in $E$). From this, and the energy quantization theorem (theorem \ref{energyquandberriv}, written above), we deduce that a simple minimal bubble  of energy $E = 8 \pi$ is blown. The only possible bubble is then an Enneper surface (see for instance  \cite{MR852409}), given by :
\begin{equation}
\label{240520191839}
E(z) = 2 \Re \left( \frac{z}{2} \begin{pmatrix} 1 \\i \\ 0 \end{pmatrix} + \frac{z^2}{2} \begin{pmatrix} 0 \\ 0 \\ 1 \end{pmatrix} - \frac{z^3}{6} \begin{pmatrix} 1 \\ -i \\ 0 \end{pmatrix} \right).
\end{equation}
This is enough to ensure that the immersions $\Phi_\mu$ offer an exemple of an Enneper bubble appearing on a sequence of Willmore immersions, which proves theorem \ref{theoducontreexemle}. \end{proof}

We however wish to make the appearance of the Enneper bubble explicit in the computations. To do that we will perform a blow-up at the origin at scale $\mu^3$. This concentration scale has been determined the classical way (see the bubble tree extraction procedure in \cite{MR3843372} or \cite{bibenergyquant}) by computing  $\left\| \nabla \n_{\Psi_\mu} \right\|_{L^\infty \left(\s^2 \right)}$. Since these computations do not by themselves further the understanding of the bubbling phenomenons, they are omitted. Considering \eqref{240520191849} we find 
$$
\begin{aligned}
f_\mu \left( \mu^3 z \right) &= \frac{3}{ 2\mu^3 \left( \mu^6z^3 -1\right)} \begin{pmatrix} 1 \\ i \\ 0 \end{pmatrix} + \frac{a^2 \mu}{8} \left( 3 \frac{2 z \mu^2 +1}{1 + \mu^2 z + \mu^4z^2} + \mu^2 z \right) \begin{pmatrix} 1 \\ -i \\ 0 \end{pmatrix} \\& +  \frac{3a}{2 \mu} \frac{1 + \mu^2 z }{ 1 + \mu^2 z + \mu^4z^2 } \begin{pmatrix} 0 \\ 0 \\ 1 \end{pmatrix} \\
&= \left( - \frac{3}{2 \mu^3} - \frac{3 \mu^3 z^3}{2}  + O \left( \mu^9 \right) \right) \begin{pmatrix} 1 \\ i \\ 0 \end{pmatrix} +  \frac{a^2 }{2} \left( \frac{3\mu}{4}  +  z \mu^3 -  + O( \mu^5) \right) \begin{pmatrix} 1 \\- i \\ 0 \end{pmatrix} \\ &+a \left(\frac{3}{2\mu} - \frac{3}{2}\mu^3z^2  + O(  \mu^5) \right) \begin{pmatrix} 0 \\ 0 \\ 1 \end{pmatrix}.
\end{aligned}
$$
Which means that 
$$
\begin{aligned}
\Psi_\mu ( \mu^3 z ) &= \frac{1}{2}\left( \frac{3 \overline{a}^2\mu}{4} - \frac{3}{\mu^3}  + \mu^3 \left( \overline{a}^2 \overline{z} - 3 z^3 \right) + O ( \mu^5 ) \right) \begin{pmatrix} 1 \\ i \\0 \end{pmatrix}  \\
&+ \frac{1}{2} \left( \frac{3a^2 \mu}{4} - \frac{3}{\mu^3} + \mu^3 \left( a^2z - 3 \zb^3 \right) + O ( \mu^5 ) \right) \begin{pmatrix} 1 \\- i \\ 0 \end{pmatrix} \\
&+\left( \frac{3 \left( a + \overline{a} \right) }{2\mu} -\frac{3 \mu^3 }{2} \left(a z^2 + \overline{a} \zb^2 \right) + O (\mu^5 ) \right) \begin{pmatrix} 0 \\ 0 \\ 1 \end{pmatrix}.
\end{aligned}
$$
With $p =\frac{ p_1}{2} \begin{pmatrix} 1 \\ -i \\0 \end{pmatrix} + \frac{ \overline{ p_1} }{2 } \begin{pmatrix} 1 \\ i \\ 0 \end{pmatrix} + p_3 \begin{pmatrix} 0 \\ 0 \\1 \end{pmatrix}$ defined previously we conclude : 
\begin{equation}
\label{240520191937}
\begin{aligned}
\Psi_\mu ( \mu^3 z ) - p  &= \frac{1}{\mu^3} \left(  \frac{1}{2} \left( -3 - \mu^3 p_1 + \frac{3 \overline{a}^2 \mu^4}{4} + \mu^6 \left( \overline{a}^2 \overline{z} - 3 z^3 \right) + O ( \mu^8 ) \right) \begin{pmatrix} 1 \\ i \\0 \end{pmatrix} \right.  \\
& \left. +\frac{1}{2} \left( -3 - \mu^3\overline{ p_1} + \frac{3 a^2 \mu^4}{4} + \mu^6 \left( {a}^2 {z} - 3 \zb^3 \right) + O ( \mu^8 ) \right) \begin{pmatrix} 1 \\- i \\0 \end{pmatrix} \right.  \\ 
& \left. + \left( \frac{3 ( a + \overline{a} ) \mu^2}{2}  - \mu^3 p_3 - \frac{3 \mu^6}{2} (az^2 +  \overline{a}\zb^2 ) + O (\mu^8) \right) \begin{pmatrix} 0 \\ 0 \\ 1 \end{pmatrix} \right).
\end{aligned}
\end{equation}
Here the only relevant terms are the first non constant ones \emph{i.e.} those in $\mu^6$.
This yields :
\begin{equation}
\label{240520191937bis}
\begin{aligned}
\left| \Psi_\mu ( \mu^3 z ) - p  \right|^2 &= \frac{1}{ \mu^6} \left(  9  + 3 \mu^3 ( p_1 + \overline{p_1} )  +  \mu^4 \frac{9 (a^2 + \overline{a}^2 )}{4} +  \mu^4 \frac{9 (a + \overline{a} ) ^2 }{4}  \right.\\ &\left. - \frac{3\mu^5 (a + \overline{a} )}{2} ( p_3 + \overline{p_3} ) -3 \mu^6 ( a^2 z + \overline{a}^2 \zb^- 3 z^3 - 3 \zb^3 + |p_1|^2 + |p_3|^2 ) \right. \\& \left.+ O (\mu^7) \right).
\end{aligned}
\end{equation}
We can combine \eqref{240520191937} and \eqref{240520191937bis} :
$$
\begin{aligned}
\Phi_\mu \left( \mu^3 z \right) &= \frac{ \Psi_\mu - p }{ \left| \Psi_\mu - p \right|^2} \\&= \mu^3 \left( \frac{1}{2} \left( -\frac{1}{3} - \frac{\mu^3}{9} p_1 + \frac{ \overline{a}^2 \mu^4}{12} + \frac{\mu^6}{9} \left( \overline{a}^2 \overline{z} - 3 z^3 \right) + O ( \mu^7 ) \right) \begin{pmatrix} 1 \\ i \\0 \end{pmatrix} \right.  \\
& \left. +\frac{1}{2} \left( -\frac{1}{3} - \frac{\mu^3}{9}\overline{ p_1} + \frac{ a^2 \mu^4}{12} + \frac{\mu^6}{9} \left( {a}^2 {z} - 3 \zb^3 \right) + O ( \mu^7 ) \right) \begin{pmatrix} 1 \\ -i \\0 \end{pmatrix} \right.  \\ 
& \left. + \left( \frac{ a + \overline{a}  \mu^2}{6}  -\frac{ \mu^3}{9} p_3 - \frac{ \mu^6}{6} (a z^2 + \overline{a} \zb^2 ) + O (\mu^7) \right) \begin{pmatrix} 0 \\ 0 \\ 1 \end{pmatrix} \right)\left(  1 - \frac{1}{3} \mu^3 ( p_1 + \overline{p_1} )   \right. \\ & \left. -  \mu^4 \frac{ (a^2 + \overline{a}^2 )}{4}  -  \mu^4 \frac{ (a + \overline{a} ) ^2 }{4} + \frac{\mu^5 (a + \overline{a} )}{6} ( p_3 + \overline{p_3} ) \right. \\ & \left.+\frac{1}{3} \mu^6 ( a^2 z + \overline{a}^2 \zb- 3 z^3 - 3 \zb^3 + |p_1|^2 + |p_3|^2 + \frac{1}{3} \left( p_1 + \overline{p_1}\right)^2 )+ O (\mu^7) \right) \\
&=\Phi_\mu (0)  + \mu^9 \left( \frac{1}{9}\left(  \overline{a}^2 \zb - 3 z^3 -  a^2z - \overline{a}^2 \zb + 3 z^3 + 3 \zb^3 \right) \frac{1}{2 } \begin{pmatrix} 1\\ i \\ 0 \end{pmatrix}  \right. \\
& \left. +  \frac{1}{9} \left( a^2 z - 3 \zb^3  -  a^2z - \overline{a}^2 \zb + 3 z^3 + 3 \zb^3 \right) \frac{1}{2} \begin{pmatrix} 1 \\ -i \\ 0 \end{pmatrix}   -  \frac{a z^2 + \overline{a} \zb^2 }{6} \begin{pmatrix} 0 \\ 0 \\ 1 \end{pmatrix} \right)  \\&+ O ( \mu^{10} ) \\
&= \Phi_\mu (0) + \mu^9 \left( \left( \frac{\zb^3}{3} - \frac{a^2}{9}z \right) \frac{1}{2} \begin{pmatrix} 1 \\i \\ 0 \end{pmatrix}   + \left( \frac{z^3}{3} - \frac{a^2}{9} \zb \right) \frac{1}{2} \begin{pmatrix} 1 \\ -i \\ 0 \end{pmatrix}  \right. \\ &\left.- \left( \frac{a}{3} \frac{z^2}{2} + \frac{\overline{a}}{3} \frac{ \zb^2}{2} \right) \begin{pmatrix} 0 \\ 0 \\ 1 \end{pmatrix} \right) + O (\mu^{10} ).
\end{aligned}
$$
Taking $ a = 3$ we find exactly
\begin{equation}
\label{250520190830}
\Phi_\mu \left( \mu^3 z \right) = \Phi_\mu (0) - \mu^9 E( z)  + O\left( \mu^{10} \right).
\end{equation}
Hence we do have : 
$$\frac{\Phi_\mu ( \mu^3z) - \Phi_\mu (0) }{ -\mu^9 } \rightarrow E(z)$$ smoothly on every compact of $\C$, which does illustrate theorem \ref{theoducontreexemle}. 
\begin{remark}
Chosing another value for $a$ would have led to another Enneper surface, with Enneper-Weierstrass data $ (f,g) =(1, \frac{3}{a}z)$ instead of simply $ (f,g)=(1,z)$. 
\end{remark}
\begin{remark}
One must notice the fundamentally asymetric role of $\Phi_0$ (the surface) and $E$ (the bubble). Indeed while we have compactly glued $E$ on $\Phi_0$ we cannot compactly glue $\Psi_0 = \iota \circ \Phi_0$ on an inverted Enneper using the same construction, since $\Psi_0$ has an end  which is not on the concentration point. Doing so would require to glue a closed bubble tree on said planar end (and would necessarily add Willmore energy to the concentration point). Further theorem \ref{lecorintro} ensures that no construction will ever enable us to do so, given that the second residue of the inverted Enneper surface is $\alpha = 2$ (see \cite{bibpointremov}).
\end{remark}

\section{Setting for a proof of theorem \ref{lecorintro} in local conformal charts : }
\label{lasection2}
%We consider $\Phi^\varepsilon \, : \, \D \rightarrow \R^3$ a sequence of Willmore conformal immersions satisfying hypotheses \ref{hypoth1}-\ref{hypoth5}. 
While theorem \ref{lecorintro} is stated globally, its conclusion is  localized on the neighborhoods of concentration points on which a simple minimal bubble is blown. We will then work in conformal charts around such points.
The aim of this section is to draw a set of hypotheses that these maps may satisfy, sometimes up to slight adjustments that can be done without loss of generality.
 We will also delve into the first consequences of these hypotheses to clarify our framework.
\subsection{Convergence in local conformal charts : }
We here wish to show  the following : 
\begin{lem}
\label{lesecondresiduestmieuxintrononlem}
Let $\xi_k$ be a sequence of Willmore immersions of a closed surface $\Sigma$ satisfying the hypotheses of theorem \ref{energyquandberriv}. Then, in proper conformal charts around a concentration point on which a simple minimal bubble is blown, $\xi^k$ yields a sequence of Willmore conformal immersions 
$\Phi^\varepsilon \, : \, \D \rightarrow \R^3$, of conformal factor $\lambda^\varepsilon = \frac{1}{2} \ln \left( \frac{ \left|\nabla \Phi^\varepsilon \right|^2 }{2} \right)$, Gauss map $\n^\varepsilon$, mean curvature $H^\varepsilon$  and tracefree curvature $\Omega^\varepsilon:= 2 \left\langle \Phi^\varepsilon_{zz} , \n^\varepsilon \right\rangle$, satisfying the following set of hypotheses :
\begin{enumerate}
\item
\label{hypoth1}
There exists $C_0 >0$ such that
$$ \left\| \Phi^\varepsilon \right\|_{L^{\infty} \left( \D \right)} + \left\| \nabla \Phi^\varepsilon \right\|_{L^2 \left( \D \right)} + \left\| \nabla \lambda^\varepsilon \right\|_{L^{2,\infty} \left( \D \right)} + \left\| \nabla \n^\varepsilon \right\|_{L^2 \left( \D \right) } \le C_0.$$
\item 
\label{hypoth3}
$\Phi^\varepsilon \rightarrow \Phi^0$  $C_{\mathrm{loc}}^\infty \left( \D \backslash \{ 0 \} \right)$, where $\Phi^0$ is a true branched Willmore conformal  immersion, with a unique branch point of multiplicity $\theta_0 +1$ at $0$, meaning that 
\begin{equation}
\label{cequecaveutdiredavoirunbranchpoin}
\Phi_z^0 \sim_0 \vec{A} z^{\theta_0}.
\end{equation}
We denote $\lambda^0$ its conformal factor, $\n^0$ its Gauss map, $H^0$ its mean curvature and $\Omega^0$ its tracefree curvature.
\item
\label{hypoth4}
There exists a sequence of real numbers $C^\varepsilon>0$ such that 
$$\widetilde \Phi^\varepsilon := \frac{ \Phi^\varepsilon \left( \varepsilon. \right) - \Phi^\varepsilon (0) }{  C^\varepsilon } \rightarrow \Phi^1_z $$
$C^\infty_{\mathrm{loc}} \left( \C \right)$, where $\Phi^1$ is assumed to be a minimal conformal immersion of $\C$ with a branched end of multiplicity $\theta_1 +1$, meaning that :
$$\Phi^1_z \sim_\infty \widetilde A z^{\theta_1}.$$
We denote $\lambda^1$ its conformal factor, $\n^1$ its Gauss map, $H^1$ its mean curvature and $\Omega^1$ its tracefree curvature.
\item
\label{hypoth2}
$$\lim_{R \rightarrow \infty } \left( \lim_{\varepsilon \rightarrow 0 } \int_{\D_{\frac{1}{R}} \backslash \D_{\varepsilon R} } \left| \nabla \n^\varepsilon \right|^2 dz \right) = 0.$$
\item
\label{hypoth5}
$\left| \Omega^\varepsilon e^{-\lambda^\varepsilon} \right|$ reaches its maximum at $0$ and 
$$\left| \Omega^\varepsilon e^{-\lambda^\varepsilon} \right| (0)= \frac{2}{\varepsilon}. $$
\end{enumerate}
%Then $\theta_0=\theta_1= \theta$, and the second residue $\alpha$ of $\Phi^0$ satisfies  $\alpha \le \theta-1$.
\end{lem}
\begin{proof}
Such assumptions are natural if we consider $\xi^k$ satisfying the hypotheses of theorem \ref{energyquandberriv}.  Thanks to theorems I.2 and I.3 of \cite{bibenergyquant} such a sequence $\xi^k$ converges smoothly away from concentration points. % Let $p$ be a concentration point on which a simple minimal bubble is blown.
% sequences of Willmore immersions of a compact Riemann surface with uniformly bounded total curvature and such that the conformal class of the induced metric is in a compact of the moduli space. Indeed, thanks to theorems I.2 and I.3 of \cite{bibenergyquant} such a sequence $\xi^k$ converges smoothly away from concentration points. 
%(i.e. points $x \in \Sigma$ such that you cannot find disks of uniformly small $\nabla \n$ energy centered on $x$).
  In a conformal chart centered on a concentration point, $\xi^k$ yields a sequence of conformal, weak Willmore immersions $\Phi^k \, : \, \D \rightarrow \R^3$ converging smoothly away from the origin toward a true Willmore surface (i.e. hypothesis \ref{hypoth3}). Hypothesis \ref{hypoth1} stands if we choose proper conformal charts (see theorem 3.1 of P. Laurain and T. Rivière's   \cite{biboptimalestimates} for a detailed explanation). 
%Proposition III.1 in \cite{bibenergyquant} (recalled here as part of theorem \ref{energyquandberriv}) allows one to extract a Willmore bubble tree on the concentration point. 
Hypothesis \ref{hypoth4} then specifies that we consider the case where there is only one simple minimal bubble which concentrates on $0$ in the aforementioned chart. Hypothesis \ref{hypoth2} is just the energy quantization once the whole bubble tree is extracted and corresponds to
%Y. Bernard and T. Rivière's result of energy quantization on $\nabla \n$ (see 
inequality VIII.8 in \cite{bibenergyquant}. Further, by definition of a concentration point $$\left\| \nabla \n^{k} \right\|_{L^\infty \left( \D \right) } \rightarrow \infty.$$ On the other hand, the main result of 
\cite{bibnmheps} (namely inequality (96)) states that $$\left\| H^k \nabla \Phi^k \right\|_{L^\infty \left( \D \right) }  \le C(C_0).$$ Since $\left| \nabla \n^k \right|^2 = \left| H^k \nabla \Phi^k \right|^2 + \left| \Omega^k e^{-\lambda^k } \right|^2 $, necessarily 
$$\left\| \Omega^k e^{-\lambda^k} \right\|_{L^\infty \left( \D \right)} \rightarrow \infty.$$ We then define the concentration speed as $$\varepsilon_k = \frac{2}{\left\| \Omega^k e^{-\lambda^k} \right\|_{L^\infty \left( \D \right)}},$$ 
and we assume it is reached at the origin.  For simplicity's sake we reparametrize this sequence by the concentration speed which we denote $\varepsilon$. Hypothesis \ref{hypoth5} is then a consequence of this slight adjustment.%and hypothesis \ref{hypoth2} is just the energy quantization and corresponds to
%Y. Bernard and T. Rivière's result of energy quantization on $\nabla \n$ (see 
%inequality VIII.8 in \cite{bibenergyquant}. 
\end{proof}

An  immediate consequence of hypotheses \ref{hypoth3}-\ref{hypoth2} is the following energy quantization result
\begin{equation}
\label{energyquantization}
\int_\D \left| \nabla \n^\varepsilon \right|^2 dz \rightarrow \int_\D \left| \nabla \n^0 \right|^2 dz + \int_\C \left| \nabla \n^1 \right|^2 dz.
\end{equation}
while hypothesis \ref{hypoth1} ensures 
\begin{equation}
\label{labornesurlen1etlen0}
\left\| \nabla \n^0 \right\|_{L^2\left( \D \right) } + \left\| \nabla \n^1 \right\|_{L^2 \left( \C \right) } \le C(C_0).
\end{equation}

Further if we denote  $\widetilde{ \lambda}^\varepsilon$ the conformal factor of $\widetilde \Phi^\varepsilon$, $\widetilde H^\varepsilon$ its mean curvature, $\widetilde{ \Omega}^\varepsilon$ its tracefree curvature and $\widetilde \n^\varepsilon$ its Gauss map, we have 
\begin{equation}
\label{lescalinglambdatilde}
\widetilde \lambda^\varepsilon = \lambda^\varepsilon \left( \varepsilon . \right) - \ln \left( \frac{C^\varepsilon}{\varepsilon} \right),
\end{equation}
and
\begin{equation}\label{lescalingonomega} \widetilde \Omega^\varepsilon e^{- \widetilde \lambda^\varepsilon} = \varepsilon \left[ \Omega^\varepsilon e^{-\lambda^\varepsilon} \right] \left( \varepsilon . \right).\end{equation}
Using hypothesis \ref{hypoth5} one may conclude that \begin{equation} \label{tyrteeuziezf} \left| \widetilde \Omega^\varepsilon e^{- \widetilde \lambda^\varepsilon} \right|(0) = 2. \end{equation} Hypothesis \ref{hypoth4} then  yields  \begin{equation}\label{lemoduleomega1} \left| \Omega^1 e^{- \lambda^1} \right|(0) =2. \end{equation} 
Similarly we know, thanks to hypotheses \ref{hypoth4} and \ref{hypoth5}, \begin{equation} \label{rze8488} \left(\left| \widetilde \Omega^\varepsilon e^{- \widetilde \lambda^\varepsilon} \right| \right)_z(0) = 0, \end{equation} and  : 
\begin{equation} 
\label{lederiveemoduleomega1}
\left( \left| \Omega^1 e^{- \lambda^1} \right| \right)_z(0) =0.
\end{equation}

Finally hypothesis \ref{hypoth2} allows us to apply theorem \ref{theorempointremovdebernrivetvoila}
% exploit an expansion result for branched Willmore surfaces  (namely theorem I.8 in \cite{bibpointremov}by Y. Bernard and T. Rivière). Applied 
to $\Phi^0$ : there exists $\alpha \le \theta_0$, $\vec{A} \in \C^3\backslash \{ 0 \}$, $\left( \vec{B}_j \right)_{j=1.. \theta_0+1- \alpha} \in \C^3 $, $\vec{C}_\alpha \in \C^3 \backslash \{ 0 \}$ and $\xi \, : \, \D \rightarrow \R^3$ such that 
\begin{equation}
\label{decomposition0}
\Phi^0_z  = \vec{A}z^{\theta_0} + \sum_{j=1}^{\theta_0+1 -\alpha} \vec{B}_jz^{\theta_0+j} + \frac{\vec{C}_\alpha}{ \theta_0+1} z^{\theta_0- \alpha} \overline{z}^{\theta_0+1} +  \frac{\overline{\vec{C}_\alpha}}{ \theta_0+1- \alpha} z^{\theta_0} \overline{z}^{\theta_0+1 - \alpha}  + \xi_z ,
\end{equation}
where $\xi$ satisfies : 
$$ \nabla^j \xi = O\left( r^{2 \theta_0 +3- \alpha - j - \upsilon} \right),$$
for all $\upsilon >0$ and $j \le \theta_0+2 - \alpha$.
The second residue $\alpha$ is in fact defined as follows (see theorem I.8 in  \cite{bibpointremov}) : 
\begin{equation}
\label{lehetlealpha}
H^0 \sim C_\alpha |z|^{-\alpha}.
\end{equation}

Our proofs will  use the quantities $\Lr$, $S$ and $\vec{R}$, stemming from the Willmore conservation laws (see for instance  theorem I.4 in \cite{bibanalysisaspects}), which at the core, are a consequence of the conformal invariance of $W$ (see  \cite{bibnoetherwill}). More precisely $\Lr$, $S$ and $\vec{R}$ are defined as follows :
\begin{equation}\label{LRSintro}\begin{aligned}&  \nabla^\perp \Lr = \nabla \vec{H} -3 \pi_{\n} \left( \nabla \vec{H} \right) + \nabla^\perp \n \times \vec{H}, \\  &\nabla^\perp S = \langle \Lr, \nabla^\perp \Phi \rangle, \\
& \nabla^\perp \vec{R} = \Lr \times \nabla^\perp \Phi + 2H \nabla^\perp \Phi. \end{aligned}\end{equation}
Exploiting these was key in T. Rivière's proof of the $\varepsilon$-regularity for Willmore surfaces.

Under hypotheses \ref{hypoth1}-\ref{hypoth5}, the conclusion of \cite{bibnmheps} stands and yields (see (96)-(98) in the aforementioned paper) : 
\begin{equation} \label{lestimeefinaleenH} \left\| H^\varepsilon \nabla \Phi^\varepsilon \right\|_{L^\infty \left( \D \right)} \le   C (C_0),\end{equation}
% \left(  \| \nabla \n \|_{L^2 \left( \D \right)} + 1 \right)
\begin{equation} \label{lestimeefinaleenPhi} \left\| \nabla \Phi^\varepsilon \right\|_{W^{3,p} \left( \D \right)} \le      C(C_0), \end{equation}
while the second and third Willmore quantities satisfy
\begin{equation} \label{lestimeefinaleenSR}\| \nabla S^\varepsilon \|_{W^{1,p} \left( \D \right) } + \| \nabla \vec{R}^\varepsilon \|_{W^{1,p} \left( \D \right) }\le C(C_0)\end{equation} for all $p<\infty$. 
Up to an inconsequential translation one can further assume $\Phi^\varepsilon(0) = 0$.

\subsection{ Branch point-branched end correspondance : }
The goal of this subsection is to show  the equality of the multiplicity  of the end of the bubble  and the multiplicity of the branch point of the surface : 
\begin{theo}
\label{yqfhojlmiofjnkjdkfhskudilcwsppoficzwksd}
Let $\Phi^\varepsilon \, : \, \D \rightarrow \R^3$ satisfying \ref{hypoth1}-\ref{hypoth5}.
Then $\theta_0=\theta_1= \theta$.
\end{theo}
\begin{proof}
Since $\Phi^\varepsilon$ is conformal, the Liouville equation states 
\begin{equation}
\label{LiouvilleEquation}
\Delta \lambda^\varepsilon = K^\varepsilon e^{2\lambda^\varepsilon},
\end{equation}
where $K^\varepsilon$ is the Gauss curvature of $\Phi^\varepsilon$.
Then given $R \in \R_+$  \begin{equation} \label{hqebjsm,}\begin{aligned} \int_{\D_{\frac{1}{R}} \backslash \D_{\varepsilon R} } K^\varepsilon e^{2\lambda^\varepsilon} dz &= \int_{\D_{\frac{1}{R}} \backslash \D_{\varepsilon R} } \Delta \lambda^\varepsilon dz \\ 
&= \int_{\partial \D_{\frac{1}{R} }} \partial_r \lambda^\varepsilon d\sigma - \int_{\partial \D_{\varepsilon R} } \partial_r\lambda^\varepsilon d\sigma \\
&=\int_{\partial \D_{\frac{1}{R} }} \partial_r \lambda^\varepsilon d\sigma - \int_{\partial \D_{ R} } \varepsilon \partial_r\lambda^\varepsilon (\varepsilon . ) d\sigma\\
&= \int_{\partial \D_{\frac{1}{R} }} \partial_r \lambda^\varepsilon d\sigma - \int_{\partial \D_{ R} }  \partial_r \left[\lambda^\varepsilon (\varepsilon . )  \right]d\sigma \\
&= \int_{\partial \D_{\frac{1}{R} }} \partial_r \lambda^\varepsilon d\sigma- \int_{\partial \D_{ R} }  \partial_r\widetilde \lambda^\varepsilon  d\sigma. \end{aligned}\end{equation}
Besides, hypotheses \ref{hypoth3} and \ref{hypoth4} ensure that $\lambda^\varepsilon \rightarrow \lambda^0$ on $\partial \D_{\frac{1}{R} }$ and $\widetilde \lambda^\varepsilon \rightarrow \lambda^1 $ on $\partial \D_R$.
Further since $\Phi^0$ has a branch point of multiplicity $\theta_0+1$ at $0$, 
\begin{equation} \label{ontendsurlasurfacelambda} \lim_{R\rightarrow \infty} \int_{\D_{\frac{1}{R} } } \partial_r \lambda^0 d\sigma\rightarrow  2 \pi \theta_0.\end{equation}
Similarly $\Phi^1$ has an end of multiplicity $\theta_1+1$ at $\infty$, which implies : 
\begin{equation}
\label{ontendsurlabullelambda} \lim_{R\rightarrow \infty} \int_{\D_R} \partial \lambda^1 d\sigma \rightarrow  2 \pi \theta_1.\end{equation}
%Hence 
%\begin{equation}
%\label{ontendsurlasurfacelambda}
% \lim_{R\rightarrow \infty } \left(  \lim_{\varepsilon \rightarrow 0 }  \int_{\partial \D_{\frac{1}{R} }} \partial_r \lambda^\varepsilon d\sigma_{\partial %\D_{\frac{1}{R} }} \right) = \theta_0,
%\end{equation}
%and 
%%\begin{equation}
%\label{ontendsurlabullelambda}
%  \lim_{R\rightarrow \infty } \left(  \lim_{\varepsilon \rightarrow 0 }  \int_{\partial \D_{ R} }  \partial_r\widetilde \lambda^\varepsilon  d\sigma_{\partial \D_{ R }}  \right)=\theta_1.
%\end{equation}
Injecting \eqref{ontendsurlasurfacelambda} and \eqref{ontendsurlabullelambda}  in \eqref{hqebjsm,} yields
$$ \begin{aligned} 2 \pi \left| \theta_0 -\theta_1 \right| &\le \lim_{R\rightarrow \infty } \left(  \lim_{\varepsilon \rightarrow 0 } \left| \int_{\partial \D_{\frac{1}{R} }} \partial_r \lambda^\varepsilon d\sigma - \int_{\partial \D_{ R} }  \partial_r\widetilde \lambda^\varepsilon  d\sigma \right| \right) \\ 
&\le \lim_{R\rightarrow \infty } \left( \lim_{\varepsilon \rightarrow 0 }  \left|  \int_{\D_{\frac{1}{R}} \backslash \D_{\varepsilon R} } K^\varepsilon e^{2\lambda^\varepsilon} dz \right| \right) \le  \lim_{R\rightarrow \infty } \left( \lim_{\varepsilon \rightarrow 0 }  \int_{\D_{\frac{1}{R}} \backslash \D_{\varepsilon R} } \left|K^\varepsilon e^{2\lambda^\varepsilon} \right| dz  \right)  \\
&\le \lim_{R\rightarrow \infty } \left( \lim_{\varepsilon \rightarrow 0 }   \int_{\D_{\frac{1}{R}} \backslash \D_{\varepsilon R} } \left| \nabla \n^\varepsilon\right|^2dz  \right) =0,
\end{aligned}$$
using hypothesis \ref{hypoth1}. As a conclusion $\theta_0= \theta_1 = \theta$. 
\end{proof}

While we wrote the proof in the case specific to our paper, it remains valid for  any behavior of $\Phi^0$ and $\Phi^1$ (branched points or ends) and relies solely on the energy quantization result. In a broader frame this corresponds to the following : 
\begin{cor}
\label{uiompkikljioj}
A Willmore bubble with a branched end of multiplicity $\theta+1$  at infinity can only appear on a branch point of multiplicity $\theta+1$. \newline
A Willmore bubble with a branch point of multiplicity $\theta-1$  at infinity can only appear on a branched end of multiplicity $\theta-1$.
\end{cor}

\subsection{Macroscopic adjustments }
This section is dedicated to the proof of the following result : 
\begin{lem}
\label{theomacroscopicadjustment}
Let $\Phi^\varepsilon \, : \, \D \rightarrow \R^3$ be a sequence of Willmore conformal immersions satisfying hypotheses \ref{hypoth1}-\ref{hypoth5}. Then $\theta$ is even, and up to macroscopic adjustments  we can assume that  :
\begin{enumerate}
\setcounter{enumi}{5}
\item
\label{hypoth6}
$$ \Phi^1_z = \frac{1}{2} \begin{pmatrix} Q^2-P^2 \\ i \left( Q^2 + P^2 \right) \\ 2 PQ \end{pmatrix},$$
where $P, Q \in \C_{\frac{\theta}{2} } [X]$, $P \wedge Q = 1$, and 
$$\begin{aligned}
P(0) &= 0, \\
Q(0) &= P'(0) =1, \\ 
P''(0) &= 2 Q'(0).
\end{aligned}$$
The  end of multiplicity $\theta +1$ of $\Phi^1$ can be highlighted as follows : there exists $\widetilde A \in \C^3 \backslash \{ 0 \} $ and $V \in \C_{\theta -1} [X]$ such that  $$\Phi^1_z = \widetilde A z^\theta + V.$$
\end{enumerate}
\end{lem}
\begin{proof}
%Up to a translation one can assume $\Phi^\varepsilon (0) = 0$.
\underline{ \bf Adjusting with homothetic transformations of $\R^3$ :} \newline
%If we consider $M \in SO(3)$ and $\lambda \in \R^*_+$, 
%$$ \begin{aligned} &\lambda M \Phi^\varepsilon \rightarrow \lambda M \Phi^0 \quad C^\infty_{\mathrm{loc}} \left( \D \backslash \{ 0 \} \right),\\
%&\frac{ \lambda M \Phi^\varepsilon \left( \varepsilon . \right) - \lambda M \Phi^\varepsilon (0) }{C^\varepsilon} =  \lambda M \left( \frac{\Phi^\varepsilon \left( \varepsilon. \right) - \Phi^\varepsilon (0) }{ C^\varepsilon } \right) \rightarrow \lambda M \Phi^1 \quad C^\infty_{\mathrm{loc}} \left( \C \right).\end{aligned}$$
Since $\Phi^1$ has no branch point on $\C$, up to a fixed rotation and a dilation in $\R^3$ one can assume : 
\begin{equation} \label{laconditioninitsurlephi1z} \Phi^1_z (0) = \frac{1}{2} \begin{pmatrix} 1 \\ i \\ 0 \end{pmatrix}.\end{equation}
% that is $P(0)=0$, and $Q(0)=1$.
\underline{ \bf Adjusting the parametrization :} \newline
Taking $ M_{-\theta} = \begin{pmatrix} \cos \theta & -\sin \theta &0  \\ \sin \theta & \cos \theta &0 \\ 0 & 0 & 1 \end{pmatrix}$, we set 
$\Psi:= M_{-\theta} \Phi^1 \left( e^{i \theta}. \right)$. We denote respectively  $\lambda_\Psi$,  $H_\Psi$, $\Omega_\Psi$ and $\n_\Psi$ its conformal factor, its mean curvature, its tracefree curvature and its Gauss map. Then
$\Psi_z = e^{i \theta } M_{-\theta}  \Phi^1_z \left( e^{i \theta} . \right)$ which implies $ \n_\Psi = M_{-\theta} \n$ and $ e^{\lambda_\Psi} = e^{\lambda^1}$. Consequently using \eqref{laconditioninitsurlephi1z}
\begin{equation} \label{trezaqdre}\Psi_z (0) = \frac{e^{i \theta}}{2} M_\theta \begin{pmatrix} 1 \\ i \\ 0 \end{pmatrix} =\frac{e^{i \theta}}{2} \begin{pmatrix} e^{-i \theta} \\ i e^{-i \theta} \\ 0 \end{pmatrix}  =  \frac{1}{2} \begin{pmatrix} 1 \\ i \\ 0 \end{pmatrix}.  \end{equation}
Further we can compute $ \Psi_{zz} = e^{2i \theta } M_{-\theta} \Phi^1_{zz} \left( e^{i \theta} \right)$ and deduce 
$$\Omega_\Psi =  e^{2 i \theta} \Omega^1,$$
which in turn implies 
$$\left[ \Omega_\Psi e^{-\lambda_\Psi} \right] (0) = e^{2i \theta } \left[ \Omega^1 e^{-\lambda^1 } \right](0).$$ According to \eqref{lemoduleomega1},
%, $\left| \frac{2}{ \left[ \Omega^1 e^{-\lambda^1 } \right](0) } \right|= 1$,
one can choose $\theta$ such that $e^{2i \theta} = \frac{-2}{\left[ \Omega^1 e^{-\lambda^1 }\right](0)}$. In that case $ \Omega_\Psi(0) e^{-\lambda_\Psi(0)} = -2$, which yields thanks to \eqref{trezaqdre}
\begin{equation}
\label{leomegasi}
\Omega_\Psi (0) = -2.
\end{equation}
The sequence  $\Psi^\varepsilon = M_\theta \Phi^\varepsilon \left( e^{i \theta} . \right)$ satisfies hypotheses \ref{hypoth1}, \ref{hypoth2} and \ref{hypoth5}, while 
$$\begin{aligned} &\Psi^\varepsilon  \rightarrow M_\theta \Phi^0 \left( e^{i \theta} . \right) \quad C^\infty_{\mathrm{loc}} \left( \D \backslash \{ 0 \} \right), \\
& \widetilde \Psi^\varepsilon \rightarrow \Psi \quad C^\infty_{\mathrm{loc}} \left( \C \right). \end{aligned}$$
We will not change notations for simplicity's sake, and will merely assume, without loss of generality, that 
\begin{equation}
\label{leomega1}
\Omega^1(0) = -2.
\end{equation}
\underline{ \bf Summing up :} 
\begin{equation}
\label{leresumephi1macro}
\begin{aligned}
&\Phi^1_z = \frac{1}{2} \begin{pmatrix} 1 \\ i \\ 0 \end{pmatrix}, \\
&\Omega^1 (0) = -2, \\
&\left( \left| \Omega^1 \right|^2 e^{-2\lambda^1} \right)_z (0) = 0.
\end{aligned}
\end{equation}
\underline{ \bf Consequences on the Enneper-Weierstrass representation :} \newline
Since $ \Phi^1$ is a minimal immersion we can use the Enneper-Weierstrass representation : 
$$\Phi^1_z = \frac{f}{2} \begin{pmatrix} 1-g^2 \\ i\left( 1+ g^2 \right) \\ 2g \end{pmatrix}$$
where $f$ is a holomorphic function on $\C$ and $g$ a meromorphic one. 
Since (according to \eqref{labornesurlen1etlen0}) $\Phi^1$ has finite total curvature, $g$ is a meromorphic function of finite degree on $\C$. Thus there exists two polynomials $P, \, Q \in \C  \left[ X \right]$ such that $ P \wedge Q = 1$  and 
$g = \frac{P}{Q}$.
Since $\Phi^1$ has no end on $\C$, $f$ has a zero of order $2k$ at each pole of order $k$ of $g$. Consequently there exists a holomorphic function $\widetilde f$ such that 
$f = Q^2 \widetilde f$. Further $\Phi^1$ has no branch point on $\C$ and one finite end at $\infty$, thus $\widetilde f$ is a holomorphic function without zeros and of finite order at infinity, $i.e.$ a constant. We can then write 
\begin{equation}
\label{lapremiereformephi1}
\Phi^1_z = \frac{1}{2} \begin{pmatrix} Q^2 -P^2 \\ i \left( Q^2 + P^2 \right) \\ 2 PQ \end{pmatrix}.
\end{equation}
Further since $\Phi^1$ is assumed to have an end of multiplicity $\theta+1$ one can expand \eqref{lapremiereformephi1} as 
\begin{equation}
\label{leboutsurlephi1}
\Phi^1_z = \widetilde A z^{\theta} + O\left( z^{\theta-1} \right),
\end{equation}
where $\widetilde A \in \C^3\backslash \{ 0 \}$. Comparing \eqref{lapremiereformephi1} and \eqref{leboutsurlephi1} notably implies that $\theta$ is even and  $P, \, Q \in \C_{\frac{\theta}{2} } [X]$.
From \eqref{lapremiereformephi1} we then successively deduce 
\begin{equation}
\label{len1}
\n^1 = \frac{1}{|P|^2 + |Q|^2} \begin{pmatrix} P \overline{Q} + \overline{P}Q \\ i \left( \overline{P}Q - P \overline{Q} \right) \\|P|^2 - |Q|^2 \end{pmatrix},
\end{equation}
\begin{equation}
\label{lelambda1}
e^{2\lambda^1} = \left( |P|^2 + |Q|^2 \right)^2,
\end{equation}
\begin{equation}
\label{lephi1zz}
\Phi^1_{zz} = Q' \begin{pmatrix} Q \\ Q i \\P \end{pmatrix} - P' \begin{pmatrix} P \\ -i P \\ -Q \end{pmatrix},
\end{equation}
which implies 
\begin{equation}
\label{leOmega1exact}
\Omega^1 =2 \left( P Q' - P' Q \right),
\end{equation}
and in turn
\begin{equation}
\Omega^1 e^{- \lambda^1 } = 2 \frac{ PQ' -P'Q}{|P|^2+ |Q|^2 },
\end{equation}
\begin{equation}
\left| \Omega^1 e^{ - \lambda^1 } \right|^2 = 4 \frac{ |P|^2 |Q'|^2 + |P'|^2 |Q|^2 - P\overline{P'} Q' \overline{Q} - P'Q \overline{PQ'} }{(|P|^2+ |Q|^2) ^2} .
\end{equation}
Conditions \eqref{leresumephi1macro} then translate on $P$ and $Q$ as 
\begin{equation}
\label{conditionPQ1}
\begin{aligned}
P(0) &= 0, \\
Q(0) &= P'(0) = 1, \\
P''(0) &= 2Q'(0).
\end{aligned}
\end{equation}
This concludes the proof.
\end{proof}
\subsection{ Infinitesimal adjustments :}
This section will prove the infinitesimal counterpart of theorem \ref{theomacroscopicadjustment}.
\begin{lem}
\label{theomicroscopicadjustment}
Let $\Phi^\varepsilon \, : \, \D \rightarrow \R^3$ be a sequence of Willmore conformal immersions satisfying hypotheses \ref{hypoth1}-\ref{hypoth6}. \newline
Up to infinitesimal adjustments we can assume that : 
\begin{enumerate}
\setcounter{enumi}{6}
\item
\label{hypoth7}
$$\begin{aligned}
&\Phi^\varepsilon_z (0) = \frac{C^{\varepsilon}}{\varepsilon} \Phi_z^1 (0) =  \frac{C^\varepsilon}{2 \varepsilon} \begin{pmatrix} 1 \\ i \\ 0 \end{pmatrix},\\
&\left[ \Omega^\varepsilon e^{-\lambda^\varepsilon } \right] (0)  = \frac{1}{\varepsilon} \left[ \Omega^1 e^{-\lambda^1} \right](0) =- \frac{2}{\varepsilon}.
\end{aligned}$$
\end{enumerate}
\end{lem}
\begin{proof}
\underline{ \bf Using homothetic transformations of $\R^3$ :} \newline
By hypothesis \ref{hypoth4}, $\widetilde \Phi^\varepsilon_z (0)  \rightarrow \Phi^1_z (0)$, thus there exists a sequence of homothetic transformations $\sigma^\varepsilon \rightarrow Id$ such that $\sigma^\varepsilon \widetilde \Phi^\varepsilon_z (0) = \Phi^1_z (0)$. Since $\sigma^\varepsilon$ tends toward the identity, hypotheses \ref{hypoth1}-\ref{hypoth4} are still satisfied, and hypothesis \ref{hypoth5} still stands due to the conformal invariance properties of the tracefree curvature. We will then apply this sequence of transformations without changing the notations for simplicity's sake and assume $ \widetilde \Phi^\varepsilon_z (0) = \Phi^1_z (0)$. 
 Considering $$\widetilde \Phi^\varepsilon_z = \frac{\varepsilon}{C^\varepsilon} \Phi^\varepsilon_z,$$ we deduce with \eqref{leresumephi1macro},
\begin{equation} \Phi^\varepsilon_z(0) = \frac{C^\varepsilon}{2 \varepsilon} \begin{pmatrix} 1 \\ i \\0 \end{pmatrix}. \end{equation}
\underline{ \bf Adjusting the parametrization :} \newline
Using \eqref{lemoduleomega1} %we find
%$$\left| \frac{\widetilde \Omega^\varepsilon e^{- \widetilde \lambda^\varepsilon} }{ \Omega^1 e^{-\lambda^1} } \right|(0) = 1.$$
%We can then set, with
and \eqref{lescalingonomega}, we can set $$e^{2i\theta^\varepsilon} = \frac{ \left[ \widetilde \Omega^\varepsilon e^{- \widetilde \lambda^\varepsilon}\right](0) }{ \left[ \Omega^1 e^{-\lambda^1} \right](0) } = \varepsilon \frac{ \left[  \Omega^\varepsilon e^{-  \lambda^\varepsilon}\right](0) }{ \left[ \Omega^1 e^{-\lambda^1} \right](0) } \rightarrow 1.$$
We consider $\Psi^\varepsilon = M_{\theta^\varepsilon} \Phi^\varepsilon \left( e^{i \theta^\varepsilon} . \right)$. Since $e^{i \theta^\varepsilon}\rightarrow 1$, $\Psi^\varepsilon$ satisfies \ref{hypoth1}-\ref{hypoth6} (\ref{hypoth5} is still satisfied due to the invariance properties of the tracefree curvature). As detailed in the previous section we have $$\widetilde \Omega_\Psi^\varepsilon e^{-\widetilde \lambda^\varepsilon_\Psi } = -2,$$ 
which implies
$$ \left[ \Omega_\Psi^\varepsilon e^{- \lambda^\varepsilon_\Psi }\right](0) = -\frac{2}{\varepsilon}.$$ 
For simplicity's sake we will not change the notations and assume that $\Phi^\varepsilon$ satisfies
\begin{equation}
\label{laconditioninitialiesurleomegaepsilon}
\left[ \Omega^\varepsilon e^{-\lambda^\varepsilon} \right](0) = -\frac{2}{\varepsilon}.
\end{equation}
This gives us the desired result.
\end{proof}
%\begin{remark}
%The adjustments in lemma \ref{theomacroscopicadjustment} and \ref{theomicroscopicadjustment} were focused on the parametrization, translations, dilations and rotations. A whole family of conformal transforms revolving around the inversions remains. It will be the subject of the next subsection and we will need an additional condition  to avoid degeneracy when playing with inversions.
%\end{remark}

\section{Expanding the conformal factor}
\label{lasection3}
This section will prove the following expansion on the conformal factor, which will serve as a stepping point in the proof of theorem  \ref{lecorintro}.
\begin{theo}
\label{798915456462612115455}
Let $\Phi^\varepsilon$ be a sequence of Willmore conformal immersions satisfying  \ref{hypoth1}-\ref{hypoth7}. 
Then there exists $l^\varepsilon \in L^\infty \left( \D\right)$ such that  : 
$$\begin{aligned} 
&\lambda^\varepsilon = \ln \left( \varepsilon^\theta +r^\theta \right) + l^\varepsilon, \\
&\left\| l^\varepsilon \right\|_{L^\infty \left( \D \right) } \le C(C_0).
\end{aligned}$$
As a result if we denote $\chi = \sqrt{\varepsilon^2 + r^2 }$, the immersion satisfies the following Harnack inequality : 
$$\frac{\chi^\theta }{C(C_0)} \le e^{\lambda^\varepsilon} \le C(C_0) \chi^\theta.$$
\end{theo}
\begin{proof}
%\subsection{First expansion : }
\underline{ \bf Step 1 :  Controls on the neck area  }\newline
Given hypothesis \ref{hypoth2}, for any $\varepsilon_0>0$ arbitrarily small there exists $R$ big enough such that 
\begin{equation}
\label{sousleniveaudenergiedanslecou}
\lim_{\varepsilon\rightarrow 0} \left( \int_{\D_{\frac{1}{R} } \backslash \D_{\varepsilon R} } \left| \nabla \n^\varepsilon \right|^2 dz \right)\le \varepsilon_0.
\end{equation}
We first recall  lemma V.3 of \cite{bibenergyquant}.
\begin{lem}
\label{lelemauxdeberriv}
There exists a constant $\eta>0$ with the following property. Let $0<4r<R< \infty$. If $\Phi$ is any (weak) conformal immersion of $\Omega := \D_R \backslash \D_r$ into $\R^3$ with $L^2$-bounded second fundamental form and satisfying
$$\left\| \nabla \n \right\|_{L^{2,\infty} \left(\Omega \right)} < \sqrt{\eta},$$
then there exist $\frac{1}{2} < \alpha < 1$ and $A \in \R$ depending on $R$, $r$, $m$ and $\Phi$ such that 
\begin{equation}
\label{leharnackdanslecou}
\left\| \lambda (x) - d \ln |x| - A \right\|_{L^\infty \left( \D_{\alpha R} \backslash \D_{\frac{r}{\alpha}} \right)} \le C \left( \left\| \nabla \lambda \right\|_{L^{2,\infty} \left( \Omega \right)} + \int_\Omega \left| \nabla \n \right|^2 \right),
\end{equation} 
where $d$ satisfies 
\begin{equation}
\begin{aligned}
\left| 2 \pi d - \int_{\partial \D_r } \partial_r \lambda dl_{\partial \D_r} \right| \le& C \left[ \int_{\D_{2r} \backslash \D_r } \left| \nabla \n \right|^2dz \right.\\
&\left. + \frac{1}{\ln \frac{R}{r}} \left( \left\| \nabla \lambda \right\|_{L^{2,\infty} \left(\Omega \right)} + \int_{\Omega} \left| \nabla \n \right|^2 \right)\right].
\end{aligned}
\end{equation}
\end{lem}
Thus, according to \eqref{sousleniveaudenergiedanslecou}, there exists $R_0$ such that for all $R \ge  R_0$ and $\varepsilon$ small enough, we can apply lemma \ref{lelemauxdeberriv} on $\D_{\frac{1}{R}} \backslash \D_{\varepsilon R } $ and conclude that there exists $d^\varepsilon_R$ and $A^\varepsilon_R \in \R$ such that 
\begin{equation}
\label{leharnackdanslecouepsi}
\left\| \lambda^\varepsilon (x) - d^\varepsilon_R \ln r - A^\varepsilon_R \right\|_{L^\infty \left( \D_{\frac{1}{2R}} \backslash \D_{2 \varepsilon R} \right)} \le C_0,
\end{equation} 
\begin{equation}
\begin{aligned}
\left| d^\varepsilon_R - \frac{1}{2 \pi} \int_{\partial \D_{\varepsilon R} } \partial_r \lambda^\varepsilon dl_{\partial \D_{\varepsilon R}} \right| \le& C \left[ \int_{\D_{2\varepsilon R} \backslash \D_{\varepsilon R} } \left| \nabla \n \right|^2dz  + \frac{C_0}{-\ln \left( \varepsilon R^2\right)} \right].
\end{aligned}
\end{equation}
Here $C_0$ is the uniform bound given by hypothesis \ref{hypoth1} (up to a multiplicative uniform constant).
We saw in \eqref{ontendsurlabullelambda} that  $$ \lim_{R \rightarrow \infty} \left( \lim_{\varepsilon \rightarrow 0 } \frac{1}{2 \pi} \int_{\partial \D_{\varepsilon R} } \partial_r \lambda^\varepsilon dl_{\partial \D_{\varepsilon R}} \right) =\theta,$$ while hypothesis \ref{hypoth2} ensures that 
$$ \lim_{R \rightarrow \infty} \left( \lim_{\varepsilon \rightarrow 0 } \int_{\D_{2\varepsilon R} \backslash \D_{\varepsilon R} } \left| \nabla \n \right|^2 dz \right)=0.$$
Hence we can fix $R_1 >0$ such that for $\varepsilon $ small enough : 
\begin{equation}
\label{ledestpasloindum}
\left| d^\varepsilon - \theta \right| \le \frac{1}{10^3}.
\end{equation}
Since $R_1$ is fixed, we will get rid of the subscript on $d^\varepsilon$ and $A^\varepsilon$.
Then for any $\varepsilon $ small enough : 
\begin{equation}
\label{leharnackdanslecouepsiR1}
\left\| \lambda^\varepsilon (x) - d^\varepsilon \ln r - A^\varepsilon \right\|_{L^\infty \left( \D_{\frac{1}{2R_1}} \backslash \D_{2 \varepsilon R_1} \right)} \le C( C_0, R_1).
\end{equation} \newline
\underline{ \bf Step 2 : Estimates on the exterior boundary  :}\newline
%Since all the studied quantities are smooth,
%\begin{equation} \label{estimeebordexterieur} \left\| \lambda^\varepsilon (x) - d^\varepsilon \ln r - A^\varepsilon \right\|_{L^\infty \left( \partial \D_{\frac{1}{2R_1}}\right)} \le C_0.\end{equation}
%Furthermore, 
Hypothesis \ref{hypoth3} ensures that on $\partial \D_{\frac{1}{2R_1}}$, $\lambda^\varepsilon \rightarrow \lambda^0$ smoothly, and that $\lambda^0$ is a bounded function away from $0$, which implies 
\begin{equation}
\label{oncontroleletermelambdavarepsilonext}
\left\| \lambda^\varepsilon \right\|_{L^\infty \left( \partial \D_{\frac{1}{2R_1}}\right)}  \le C(  C_0, R_1).
\end{equation}
On the other hand \eqref{ledestpasloindum} ensures
\begin{equation}
\label{oncontroleletermedext}
\left| d^\varepsilon \ln R_1 \right| \le C(C_0, R_1).
\end{equation}
As a result, combining \eqref{leharnackdanslecouepsiR1} on $\partial \D_{\frac{1}{2R_1}}$, \eqref{oncontroleletermelambdavarepsilonext}, and \eqref{oncontroleletermedext} yields
$$\left| A^\varepsilon \right| \le C(C_0,R_1),$$
which we can inject in \eqref{leharnackdanslecouepsiR1} to obtain
\begin{equation}
\label{leharnackdanslecouepsiR1bis}
\left\| \lambda^\varepsilon (x) - d^\varepsilon \ln r \right\|_{L^\infty \left( \D_{\frac{1}{2R_1}} \backslash \D_{2 \varepsilon R_1} \right)} \le C(C_0, R_1).
\end{equation} \newline
\underline{ \bf Step 3 : Estimates on the interior boundary  :}\newline
Estimate \eqref{leharnackdanslecouepsiR1bis} implies 
\begin{equation} \label{onestbonsurlebordinterieurenlambda} \left\| \lambda^\varepsilon (x) - d^\varepsilon \ln r \right\|_{L^\infty \left( \partial \D_{2 \varepsilon R_1} \right)} \le C(C_0, R_1).\end{equation}
Further  \eqref{lescalinglambdatilde} yields
$$\left\| \lambda^\varepsilon (x) - d^\varepsilon \ln r \right\|_{L^\infty \left( \partial \D_{2 \varepsilon R_1} \right)}  = \left\| \widetilde \lambda^\varepsilon (x) - d^\varepsilon \ln r + \ln \left(\frac{C^\varepsilon}{\varepsilon} \right) - d^\varepsilon \ln \varepsilon \right\|_{L^\infty \left( \partial \D_{2  R_1} \right)}. $$
Hypothesis \ref{hypoth4} then ensures that 
\begin{equation} \label{oncontroleletermelambdatildeint} \left\| \widetilde \lambda^\varepsilon \right\|_{L^\infty \left( \D_{2R_1} \right) } \le C(C_0, R_1),\end{equation}
and \eqref{ledestpasloindum} that 
\begin{equation} \label{oncontroleletermedint}  \left\| d^\varepsilon \ln r \right\|_{L^\infty \left( \D_{2R_1} \right) } \le C(C_0, R_1). \end{equation}
Together \eqref{onestbonsurlebordinterieurenlambda}, \eqref{oncontroleletermelambdatildeint} and \eqref{oncontroleletermedint}  yield 
\begin{equation} \label{oncontroleleCepsilon} \left| \ln \left( \frac{ C^\varepsilon }{\varepsilon^{d^\varepsilon +1} } \right) \right| \le C(C_0, R_1). \end{equation}
A direct consequence of \eqref{ledestpasloindum} and \eqref{oncontroleleCepsilon} is the following estimate : 
\begin{equation} 
\label{leCepsiilonestmoralementepsilonm}
\frac{\varepsilon^{\theta +1 - 10^{-3} } }{C(R_1,C_0)} \le \frac{\varepsilon^{d^\varepsilon +1}}{C(C_0, R_1) } \le C^\varepsilon \le C(C_0, R_1) \varepsilon^{d^\varepsilon+1} \le C(C_0, R_1) \varepsilon^{\theta+1 + 10^{-3}}.
\end{equation}
\newline
\underline{ \bf Step 4 : Expanding the conformal factor on the whole disk  :}\newline
We forcefully write $\lambda^\varepsilon = \ln \left( \varepsilon^{d^\varepsilon} + r^{d^\varepsilon} \right) + l^\varepsilon$. We aim to show that $$\left\| l^\varepsilon \right\|_{L^\infty \left( \D\right) }\le C( C_0, R_1).$$
\newline
\underline{ On $ \D \backslash \D_{\frac{1}{4R_1} } $ :} \newline
Using hypothesis \ref{hypoth3},  \begin{equation} \label{lambdaalexterieurnya} \left\| \lambda^\varepsilon \right\|_{L^\infty \left( \D \backslash \D_{\frac{1}{4R_1}} \right) } \le C(C_0,R_1 ) .\end{equation} 
One might also notice  that, thanks to \eqref{ledestpasloindum},  on $\D \backslash \D_{\frac{1}{4R_1} } $ : 
\begin{equation} \label{lnepsilonrexterieur} \left| \ln \left( \varepsilon^{d^\varepsilon} + r^{d^\varepsilon} \right) \right| \le C(C_0, R_1) \left| \ln \left( \left( \frac{1}{2R_1} \right)^{m+ \frac{1}{100} } \right) \right|.\end{equation}
Then using \eqref{lambdaalexterieurnya} and \eqref{lnepsilonrexterieur} :
\begin{equation} \label{controlelalexterieur}  \begin{aligned} \left\| l^\varepsilon \right\|_{ L^\infty \left( \D \backslash \D_{\frac{1}{4R_1} } \right) } &\le  \left\| \lambda^\varepsilon \right\|_{ L^\infty \left( \D \backslash \D_{\frac{1}{4R_1} } \right) }  +  \left\| \ln \left( \varepsilon^{d^\varepsilon} + r^{d^\varepsilon} \right) \right\|_{ L^\infty \left( \D \backslash \D_{\frac{1}{4R_1} } \right) }  \\ 
&\le C(C_0, R_1).
\end{aligned}\end{equation}\newline
\underline{ On $ \D_{ 4 \varepsilon R_1 } $ :} \newline
Using hypothesis \ref{hypoth4} 
\begin{equation}
\label{oncontrolelefacalinterieur}
\left\| \widetilde  \lambda^\varepsilon \right\|_{L^\infty \left( \D_{4 R_1} \right) } \le C(C_0, R_1 ),
\end{equation}
while thanks to \eqref{oncontroleleCepsilon}
\begin{equation}
\label{oncontroleletermeenlogalinterieur}
\begin{aligned}
\left\| \ln \left( \varepsilon^{d^\varepsilon} + r^{d^\varepsilon} \right) - \ln \left( \frac{ C^\varepsilon}{\varepsilon} \right) \right\|_{L^\infty \left( \D_{4 \varepsilon R_1} \right) } &\le \left\| \ln \left( \varepsilon^{d^\varepsilon} +  \left( \varepsilon r\right)^{d^\varepsilon} \right) - \ln \left( \frac{ C^\varepsilon}{\varepsilon} \right) \right\|_{L^\infty \left( \D_{4 R_1} \right) } \\
&\le \left\| \ln \left( 1 + r^{d^\varepsilon} \right) \right\|_{L^\infty \left( \D_{4  R_1} \right) } + \left\|  \ln \left(  \frac{C^\varepsilon}{\varepsilon^{d^\varepsilon+1} }  \right) \right\|_{L^\infty \left( \D_{4 R_1} \right) } \\
&\le C(C_0,R_1).
\end{aligned}
\end{equation}
To conclude we deduce thanks to  \eqref{oncontrolelefacalinterieur} and \eqref{oncontroleletermeenlogalinterieur} : 
\begin{equation}
\label{oncontrolelepsilonalinterieur}
\begin{aligned}
\left\| l^\varepsilon \right\|_{L^\infty \left( \D_{4 \varepsilon R_1} \right) } &\le \left\| \lambda^\varepsilon - \ln \left( \varepsilon^{d^\varepsilon} + r^{d^\varepsilon} \right) + \ln \left( \frac{ C^\varepsilon}{\varepsilon} \right) - \ln \left( \frac{ C^\varepsilon}{\varepsilon} \right) \right\|_{L^\infty \left( \D_{4 \varepsilon R_1} \right) } \\
&\le  \left\|  \lambda^\varepsilon - \ln \left( \frac{ C^\varepsilon}{\varepsilon} \right) \right\|_{L^\infty \left( \D_{4 \varepsilon R_1} \right) } +  \left\| \ln \left( \varepsilon^{d^\varepsilon} +r^{d^\varepsilon} \right) - \ln \left( \frac{C^\varepsilon}{\varepsilon} \right) \right\|_{L^\infty \left( \D_{4 \varepsilon R_1} \right) } \\
&\le \left\| \widetilde \lambda^\varepsilon \right\|_{L^\infty \left( \D_{4  R_1} \right) } +  \left\| \ln \left( \varepsilon^{d^\varepsilon} +r^{d^\varepsilon} \right) - \ln \left( \frac{C^\varepsilon}{\varepsilon} \right) \right\|_{L^\infty \left( \D_{4 \varepsilon R_1} \right) } \\
& \le C(C_0,R_1).
\end{aligned}
\end{equation}
\newline
\underline{ On $ \D_{\frac{1}{2R_1} }\backslash \D_{ 2 \varepsilon R_1 } $ :} \newline
Thanks to \eqref{leharnackdanslecouepsiR1bis} : 
\begin{equation}
\label{oncontrolelepsiloncou}
\begin{aligned}
\left\| l^\varepsilon \right\|_{L^\infty \left( \D_{\frac{1}{2R_1} }\backslash \D_{ 2 \varepsilon R_1 } \right) } &\le \left\| \lambda^\varepsilon - d^\varepsilon \ln r \right\|_{L^\infty \left( \D_{\frac{1}{2R_1} }\backslash \D_{ 2 \varepsilon R_1 } \right) } + \left\| \ln \left( \frac{ r^{d^\varepsilon} }{ \varepsilon^{d^\varepsilon } +r^{d^\varepsilon} } \right) \right\|_{L^\infty \left( \D_{\frac{1}{2R_1} }\backslash \D_{ 2 \varepsilon R_1 } \right) } \\
&\le C(C_0, R_1).
\end{aligned}
\end{equation}

Combining \eqref{controlelalexterieur}, \eqref{oncontrolelepsilonalinterieur} and \eqref{oncontrolelepsiloncou}  yields 
\begin{equation}
\label{oncontrolelepsilon}
\left\| l^\varepsilon \right\|_{L^\infty \left( \D \right) } \le C(C_0,R_1),
\end{equation}
which is as desired. We now wish to refine this first expansion by showing that $d^\varepsilon$ converges toward $\theta$ fast enough to be replaced in \eqref{leharnackdanslecouepsiR1bis}.
\newline
\underline{ \bf Step 5 : Refinement  :}\newline
A consequence of estimate \eqref{oncontrolelepsilon} is the following Harnack inequality on the conformal factor : 
$$ \frac{ \varepsilon^{d^\varepsilon} + r^{d^\varepsilon} }{ C(C_0,R_1) } \le e^{\lambda^\varepsilon} \le C(C_0,R_1) \left( \varepsilon^{d^\varepsilon} + r^{d^\varepsilon} \right)$$
which, using the notation $\chi = \sqrt{ \varepsilon^2 + r^2 }$, we will rewrite in the more convenient form
\begin{equation}
\label{leharnacksurlefactconfepsilon}
 \frac{ \chi^{d^\varepsilon}  }{ C(C_0,R_1) } \le e^{\lambda^\varepsilon} \le C(C_0,R_1) \chi^{d^\varepsilon}.
\end{equation}
Injecting \eqref{ledestpasloindum} into \eqref{leharnacksurlefactconfepsilon} yields 
\begin{equation} 
\label{controlaudessusparmmoinsmille}
e^{\lambda^\varepsilon} \le C(C_0, R_1 ) \chi^{\theta- 10^{-3} }.
\end{equation}
%Applying the conclusion of \cite{bibnmheps} (precisely inequality (96)) ensures that 
%\begin{equation}
%\label{laconclusiondelarticleprecedent}
%\left\| H^\varepsilon \nabla \Phi^\varepsilon \right\|_{L^\infty \left( \D \right) } \le C( C_0).
%\end{equation}
Since $\Phi^\varepsilon$ is conformal, 
\begin{equation}
\label{ledputile1}\Delta \Phi^\varepsilon = 2 H^\varepsilon e^{2\lambda^\varepsilon} \n^\varepsilon = \chi^{\theta- 10^{-3} }  2H^\varepsilon e^{\lambda^\varepsilon} \frac{ e^{\lambda^\varepsilon} }{ \chi^{\theta-10^{-3}}}. \end{equation}
Noticing that \eqref{lestimeefinaleenH} and\eqref{controlaudessusparmmoinsmille} imply
\begin{equation}
\label{lestimeedutermeadroite1}
 \left\| 2H^\varepsilon e^{\lambda^\varepsilon} \frac{ e^{\lambda^\varepsilon} }{ \chi^{\theta-10^{-3}}} \right\|_{L^\infty \left( \D \right)} \le C(C_0, R_1),
\end{equation}
we can apply theorem \ref{theoBernardRivierebis} to equation \eqref{ledputile1} and find : 
\begin{equation}
\label{decompo1}
\Phi^\varepsilon_z = P^\varepsilon(z) + \varphi^\varepsilon_0,
\end{equation}
where $P^\varepsilon = \sum_{q=0}^\theta p^\varepsilon_q z^q \in \C_\theta \left[ X \right]$ with $\left| p^\varepsilon_q \right| \le C( C_0, R_1)$  for all $q \le \theta$ and $\varphi^\varepsilon_0 \, : \, \D \rightarrow \R^3$ satisfies 
\begin{equation}
\label{lecontrolesurlereste1}
\begin{aligned}
&\forall \upsilon >0 \quad \left\| \frac{ \varphi^\varepsilon_0 }{ \chi^{\theta+1- 10^{-3} - \upsilon}} \right\|_{L^\infty \left( \D \right) } \le C_\upsilon \left( C_0, R_1 \right), \\
&\forall p<\infty \quad  \left\| \frac{ \nabla \varphi^\varepsilon_0 }{ \chi^{\theta- 10^{-3} }} \right\|_{L^p\left( \D \right) } \le C_p \left( C_0, R_1 \right).
\end{aligned}
\end{equation}
By convergence of $\Phi^\varepsilon$ away from zero (hypothesis \ref{hypoth3}), $p^\varepsilon_q \rightarrow p_q \in \C$ as $\varepsilon$ goes to $0$.
Further \eqref{lecontrolesurlereste1} yields $\varphi^\varepsilon_0 \rightarrow \varphi_0$ $W^{1,p} \left( \D \right)$, with $\varphi_0$ satisfying 
\begin{equation}
\label{lecontrolesurlereste1limite}
\begin{aligned}
&\forall \upsilon >0 \quad \left\| \frac{ \varphi_0 }{ r^{\theta+1- 10^{-3} - \upsilon}} \right\|_{L^\infty \left( \D \right) } \le C_\upsilon \left( C_0, R_1 \right), \\
&\forall p<\infty \quad  \left\| \frac{ \nabla \varphi_0 }{ r^{\theta- 10^{-3} }} \right\|_{L^p\left( \D \right) } \le C_p \left( C_0, R_1 \right),
\end{aligned}
\end{equation}
since $\chi \rightarrow r$ as $\varepsilon \rightarrow 0$. Then \eqref{decompo1} ensures that 
\begin{equation}
\label{ladecomposition0limite}
\Phi_z^\varepsilon \rightarrow \sum_{q=0}^m p_q z^q + \varphi_0 .
\end{equation} 
Since we assumed that $\Phi^\varepsilon \rightarrow \Phi^0$ away from $0$, comparing \eqref{decomposition0} and \eqref{ladecomposition0limite} yields 
\begin{equation}
\label{lalimitedepepsilonq}
\begin{aligned}
&\forall q < \theta \quad p^\varepsilon_q \rightarrow 0 \\
& p^\varepsilon_\theta \rightarrow \vec{A} \neq 0.
\end{aligned}
\end{equation}
Further \eqref{decompo1} gives the following
\begin{equation}
\label{ladecompo1entilde}
\widetilde \Phi^\varepsilon_z = \sum_{q=0}^\theta p^\varepsilon_q \frac{\varepsilon^{q+1}}{C^\varepsilon} z^q + \frac{ \varepsilon \varphi_0^\varepsilon \left( \varepsilon . \right) }{C^\varepsilon}.
\end{equation}
One might also notice using \eqref{lecontrolesurlereste1}
\begin{equation}
\label{controlerlevarphi0danslabulle}
\begin{aligned}
\left| \varphi^\varepsilon_0 \right| \left( \varepsilon z \right) &\le  C(C_0, R_1) \chi^{\theta+ \frac{1}{2}} \left( \varepsilon z \right) \\
&\le \varepsilon^{\theta + \frac{1}{2}} C(C_0, R_1) \sqrt{1 + r^2 }^{\theta + \frac{1}{2}}. 
\end{aligned}
\end{equation}
This, along with  \eqref{oncontroleleCepsilon}, implies $$
 \begin{aligned}
\left|\frac{ \varepsilon \varphi_0^\varepsilon \left( \varepsilon . \right) }{C^\varepsilon} \right| &\le \left|  \frac{\varepsilon^{\theta+1+ \frac{1}{2}}}{C^\varepsilon} \right|C(C_0, R_1) \sqrt{1 + r^2 }^{\theta + \frac{1}{2}} \\
&\le C(C_0, R_1) \sqrt{1 + r^2 }^{\theta + \frac{1}{2}}  \varepsilon^{\frac{1}{2} }.
\end{aligned}
$$
Consequently 
\begin{equation}
\label{laconvergencedutermedereste}
 \frac{ \varepsilon \varphi_0^\varepsilon \left( \varepsilon z \right) }{C^\varepsilon} \rightarrow 0 \quad L^\infty_{\mathrm{loc} } \left( \C \right).
\end{equation}
Since $\widetilde \Phi^\varepsilon$ is assumed to converge smoothly towards $\Phi^1 $ on compacts of $\C$, we deduce from \eqref{leboutsurlephi1}, \eqref{ladecompo1entilde} and \eqref{laconvergencedutermedereste} 
\begin{equation}
\label{laconvergencedupolynome1}
 \sum_{q=0}^\theta p^\varepsilon_q \frac{\varepsilon^{q+1}}{C^\varepsilon} z^q  \rightarrow \Phi^1_z = \widetilde A z^\theta + O( z^{\theta-1}).
\end{equation}
Hence $$\frac{\varepsilon^{\theta+1}}{C^\varepsilon} p^\varepsilon_\theta \rightarrow \widetilde A \neq 0.$$ Further, given that $p^\varepsilon_\theta \rightarrow \vec{A} \neq 0$, there exists $C(C_0, R_1) >0$ such that
\begin{equation}
\label{oncontrolelerapportepsilonmplus1cepsilon}
\left| \ln \left( \frac{C^\varepsilon}{\varepsilon^{\theta+1}} \right) \right| \le C(C_0, R_1).
\end{equation}
Combining \eqref{oncontroleleCepsilon} and \eqref{oncontrolelerapportepsilonmplus1cepsilon} yields 
$$ \left| \ln \left( \frac{\varepsilon^{d^\varepsilon}}{\varepsilon^{\theta}} \right)  \right| \le C( C_0, R_1),$$
which ensures 
\begin{equation}
\label{ledepsilonestvraimentassezprochedum}
\left| \left(d^\varepsilon - \theta \right) \ln\varepsilon \right| \le C(C_0, R_1).
\end{equation}
Then, \eqref{leharnackdanslecouepsiR1bis} and \eqref{ledepsilonestvraimentassezprochedum} combine and yield
\begin{equation}
\label{labonneestimeesurlefacteurconformeavecm}
\begin{aligned}
\left\| \lambda^\varepsilon - \theta \ln r \right\|_{L^\infty \left( \D_{\frac{1}{2R_1}} \backslash \D_{2 \varepsilon R_1 } \right) } &\le  \left\| \lambda^\varepsilon - \theta \ln r \right\|_{L^\infty \left( \D_{\frac{1}{2R_1}} \backslash \D_{2 \varepsilon R_1 } \right) } + \left\| \left( \theta - d^\varepsilon \right) \ln r \right\|_{L^\infty \left( \D_{\frac{1}{2R_1}} \backslash \D_{2 \varepsilon R_1 } \right) } \\ 
&\le C(C_0, R_1).
\end{aligned}
\end{equation}
Since inequality \eqref{labonneestimeesurlefacteurconformeavecm} is analogous to \eqref{leharnackdanslecouepsiR1bis}, we can do all the reasonings from \eqref{leharnackdanslecouepsiR1bis} to \eqref{laconvergencedupolynome1}  with $d^\varepsilon = \theta$.
Then the conformal factor satisfies :
\begin{itemize}
\item
\begin{equation}
\label{labonneestimeeenmsurlefacteurconforme}
\lambda^\varepsilon = \ln \left( \varepsilon^\theta + r^\theta \right) + l^\varepsilon,
\end{equation}
with $l^\varepsilon$ such that $$\left\| l^\varepsilon \right\|_{L^\infty \left( \D \right) } \le C(C_0,R_1).$$
\item
\begin{equation}
\label{lebonharnackenmenfin}
\frac{\chi^\theta }{C(C_0, R_1) } \le e^{\lambda^\varepsilon} \le C(C_0, R_1) \chi^\theta.
\end{equation}
\end{itemize}
This concludes the proof of the desired result since $R_1$ is fixed. \qed

Further for simplicity's sake we can, up to an inconsequential (thanks to \eqref{oncontrolelerapportepsilonmplus1cepsilon})  adjustment, assume $C^\varepsilon = \varepsilon^{\theta+1}$. 
Then, exploiting \eqref{laconvergencedupolynome1} yields  : 
\begin{itemize}
\item \begin{equation} \label{leatildeestlememequeleA} \widetilde A = \vec{A}. \end{equation}
\item
$$ \sum_{q= 0}^\theta p^\varepsilon_q \varepsilon^{q-\theta} z^q  \rightarrow \Phi^1_z.$$
We can then decompose 
$$ P^\varepsilon = \varepsilon^\theta \Phi^1_z \left( \frac{z}{ \varepsilon} \right) + \varepsilon^\theta Q^\varepsilon \left( \frac{z}{\varepsilon} \right),$$
with $Q^\varepsilon \in \C_\theta[X]$ such that $Q^\varepsilon \rightarrow 0$.
\end{itemize}
$\Phi^\varepsilon$ then satisfies the following decomposition : 
\begin{equation}
\label{decomposition1labonnenem}
\Phi^\varepsilon_z =  \varepsilon^\theta \Phi^1_z\left( \frac{z}{\varepsilon} \right) + \varepsilon^\theta Q \left( \frac{z}{\varepsilon} \right) + \varphi^\varepsilon_0,
\end{equation}
 where 
\begin{equation}
\label{lecontrolesurlereste1labonne}
\begin{aligned}
&\forall \upsilon >0 \quad \left\| \frac{ \varphi^\varepsilon_0 }{ \chi^{\theta+1 - \upsilon}} \right\|_{L^\infty \left( \D \right) } \le C_\upsilon \left( C_0, R_1 \right), \\
&\forall p<\infty \quad  \left\| \frac{ \nabla \varphi^\varepsilon_0 }{ \chi^{\theta }} \right\|_{L^p\left( \D \right) } \le C_p \left( C_0, R_1 \right).
\end{aligned}
\end{equation}
\end{proof}
\begin{remark}
Equality \eqref{leatildeestlememequeleA} can be seen as prolonging theorem \ref{yqfhojlmiofjnkjdkfhskudilcwsppoficzwksd} : not only must the multiplicity of the end and the multiplicity of the branch point correspond, but so must the parametrization of the limit planes in both cases.
\end{remark}
\begin{remark} 
A. Michelat and T. Rivière have presented the author with another proof of the expansion which works in the more general framework of any simple bubble (in \cite{privatecommunicationsmichriv}).
\end{remark}
\section{Conditions on the limit surface :}
\label{lasection4}
The aim of this section is to prove both theorem \ref{lecorintro} and corollary \ref{lecaschengackstatter}. 
\subsection{ Proof of theorem \ref{lecorintro} : }
 As detailed in section \ref{lasection2} we can equivalently work in conformal parametrizations under hypotheses \ref{hypoth1}-\ref{hypoth7}. We will then instead prove : 
\begin{theo}
Let $\Phi^\varepsilon \, : \, \D \rightarrow \R^3$  be a sequence of Willmore conformal immersions satisfying hypotheses \ref{hypoth1}-\ref{hypoth7}. Then the second residue of $\Phi^0$ at $0$ satisfies 
$$\alpha \le \theta -1.$$
\end{theo}
\begin{proof}
\underline{ \bf Step 1 : Expansion of $\Phi^\varepsilon_{zz}$ :}\newline
We consider $\Phi^\varepsilon$ satisfying hypotheses \ref{hypoth1}-\ref{hypoth7}, and thus \eqref{labonneestimeeenmsurlefacteurconforme}-\eqref{lecontrolesurlereste1labonne}. 
The system (7) of \cite{bibnmheps} states 
\begin{equation}
\label{systemenRSPhiannexmieuxarticleprecedent}
\left\{
\begin{aligned}
\Delta S^\varepsilon &=  \left\langle H^\varepsilon \nabla \Phi^\varepsilon, \nabla^\perp \vec{R}^\varepsilon \right\rangle \\
\Delta \vec{R}^\varepsilon &= - H^\varepsilon \nabla \Phi^\varepsilon \times \nabla^\perp \vec{R}^\varepsilon - \nabla^\perp S^\varepsilon H^\varepsilon \nabla \Phi^\varepsilon\\
\Delta \Phi^\varepsilon &= \frac{1}{2} \left( \nabla^\perp S^\varepsilon. \nabla \Phi^\varepsilon + \nabla^\perp \vec{R}^\varepsilon \times \nabla \Phi^\varepsilon \right).
\end{aligned}
\right.
\end{equation}
Then \eqref{lestimeefinaleenH}, \eqref{lestimeefinaleenSR} and \eqref{systemenRSPhiannexmieuxarticleprecedent} yield : 
\begin{equation}
\label{premiercontroledeslaplaciensSR}
\left\| \Delta S^\varepsilon \right\|_{L^\infty \left( \D \right) } + \left\| \Delta \vec{R}^\varepsilon \right\|_{L^\infty \left( \D \right) } \le C(C_0).
\end{equation}
Applying theorem \ref{theoBernardRivierebis} gives the following decomposition on $S^\varepsilon$ and $\vec{R}^\varepsilon$ : 
\begin{equation}
\label{premieredecompoSR}
\begin{aligned}
S^\varepsilon_z &= S^\varepsilon_z(0) + \sigma^\varepsilon_0 \\
\vec{R}^\varepsilon_z &= \vec{R}^\varepsilon_z (0) + \rho^\varepsilon_0,
\end{aligned}
\end{equation}
with $\sigma^\varepsilon_0$ and $ \rho^\varepsilon_0$ satisfying 
\begin{equation}
\label{premiercontrolesigma0rho0}
\begin{aligned}
&\forall \upsilon>0 \quad \left| \sigma^\varepsilon_0 \right|(z) + \left| \rho^\varepsilon_0 \right| (z) \le C_{\upsilon} (C_0) \chi^{1-  \upsilon}, \\
&\forall p < \infty \quad \left\| \nabla \sigma^\varepsilon_0 \right\|_{L^p \left( \D \right) } + \left\| \nabla \rho^\varepsilon_0 \right\|_{L^p \left( \D \right) } \le C_p(C_0).
\end{aligned}
\end{equation}
Injecting \eqref{decomposition1labonnenem} and \eqref{premieredecompoSR} into the third equation of \eqref{systemenRSPhiannexmieuxarticleprecedent} yields 
\begin{equation}
\label{equationofthelaplacian}
\Delta \Phi^\varepsilon = 2 \Im \left( \vec{R}^\varepsilon_z(0) \times \overline{ \left[ \varepsilon^\theta \Phi^1_z \left(  \frac{z}{\varepsilon} \right) + \varepsilon^\theta Q^\varepsilon \left(  \frac{z}{\varepsilon} \right)\right] }+ S^\varepsilon_z (0) \overline{ \left[ \varepsilon^\theta \Phi^1_z \left(  \frac{z}{\varepsilon} \right) + \varepsilon^\theta Q^\varepsilon \left(  \frac{z}{\varepsilon} \right)\right] }  \right)  + \Psi^\varepsilon_0,
\end{equation} 
where $$\Psi^\varepsilon_0 := 2 \Im \left( \rho^\varepsilon_0 \times \overline{\Phi^\varepsilon_z} + \sigma^\varepsilon_0  \overline{\Phi^\varepsilon_z} + \vec{R}^\varepsilon_z \times \overline{\varphi^\varepsilon_0} + S^\varepsilon_z  \overline{\varphi^\varepsilon_0} \right)$$ satisfies 
\begin{equation}
\label{premiercontrolepsi0}
\begin{aligned}
&\forall \upsilon>0 \quad \left| \Psi^\varepsilon_0 \right|(z)\le C_{\upsilon} (C_0) \chi^{\theta+1 -  \upsilon}, \\
&\forall p < \infty \quad \left\| \frac{ \nabla \Psi^\varepsilon_0}{ \chi^{\theta}} \right\|_{L^p \left( \D \right) } \le C_p(C_0).
\end{aligned}
\end{equation}
One may notice that $$ \widetilde h^\varepsilon := 2 \Im \left( \vec{R}^\varepsilon_z(0) \times \overline{ \left[ \varepsilon^\theta \Phi^1_z \left(  \frac{z}{\varepsilon} \right) + \varepsilon^\theta Q^\varepsilon \left(  \frac{z}{\varepsilon} \right)\right] }+ S^\varepsilon_z (0) \overline{ \left[ \varepsilon^\theta \Phi^1_z \left(  \frac{z}{\varepsilon} \right) + \varepsilon^\theta Q^\varepsilon \left(  \frac{z}{\varepsilon} \right)\right] }  \right)$$ is the sum of a polynomial of degree $\theta$ in $z$ and a polynomial of degree $\theta$ in $\overline{z}$, whose coefficients are uniformly bounded by a constant depending on $C_0$.  Additionnally it is a  $O( \chi^\theta)$ thanks to theorem \ref{798915456462612115455}.
We can then find a polynomial $h^\varepsilon$ in $z$ and $\overline{z}$ of total degree $\theta+2$  such that 
$$\begin{aligned}
&h^\varepsilon (0) = h^\varepsilon_z (0) = h^\varepsilon_{\overline{z}}(0) = 0, \\
&\Delta h^\varepsilon = \widetilde h^\varepsilon, \\
&h^\varepsilon = O \left( \chi^{\theta+2} \right).
\end{aligned}$$
Then 
\begin{equation}
\label{lequationpourlepremierpassagedanslaboucle}
\Delta \left( \Phi^\varepsilon - h^\varepsilon \right) = \Psi^\varepsilon_0.
\end{equation}
Applying theorem \ref{lecordedegrearbitraire} to \eqref{lequationpourlepremierpassagedanslaboucle}, with $ a= \theta+1 - \upsilon$ for $\upsilon$ arbitrarily small yields 
\begin{equation}
\label{ladeuxiemedecompisitionduphivarepsilononvalefaire}
\Phi^\varepsilon_z = P^\varepsilon (z) + h^\varepsilon_z + \varphi^{\varepsilon}_1,
\end{equation}
where $P^\varepsilon$ is a polynomial of degree $\theta+1$ that we can split $ P^\varepsilon = P^\varepsilon_\theta + p^\varepsilon z^{\theta+1}$ with $P^\varepsilon_\theta \in \C_\theta[X]$, and $\varphi^\varepsilon_1$ satisfies : 
\begin{equation}
\label{lecontroledegree2surlevarphhi1}
\begin{aligned}
&\forall \upsilon >0 \quad \frac{ \left| \varphi^\varepsilon_1 \right|}{ \chi^{\theta+2 - \upsilon} } +\frac{ \left| \nabla  \varphi^\varepsilon_1 \right|}{ \chi^{\theta+1 - \upsilon} } \le C_\upsilon (C_0), \\
& \forall p< \infty \quad \left\| \frac{ \nabla^2 \varphi^\varepsilon_1}{ \chi^{\theta}} \right\|_{L^p \left( \D \right)} \le C_p(C_0).
\end{aligned}
\end{equation}
Comparing \eqref{decomposition1labonnenem} and \eqref{ladeuxiemedecompisitionduphivarepsilononvalefaire} as in the proof of lemma \ref{lelemmequimesertdheredite} yields : 
$$\begin{aligned}
P^\varepsilon_\theta &= \varepsilon^\theta \left[ \Phi^1 \left( \frac{z}{\varepsilon} \right) + Q^\varepsilon \left( \frac{z}{\varepsilon} \right) \right], \\
\varphi^\varepsilon_0 &= p^\varepsilon z^{\theta+1} + h^\varepsilon_z + \varphi^\varepsilon_1.
\end{aligned}
$$
Consequently $\varphi^\varepsilon_0$ satisfies : 
\begin{equation}
\label{lecontroledegree2surlevarphhi0}
\begin{aligned}
& \frac{ \left| \varphi^\varepsilon_0 \right|}{ \chi^{\theta+1} } +\frac{ \left| \nabla  \varphi^\varepsilon_0 \right|}{ \chi^{\theta} } \le C (C_0) \\
& \forall p< \infty \quad \left\| \frac{ \nabla^2 \varphi^\varepsilon_0}{ \chi^{\theta-1}} \right\|_{L^p \left( \D \right)} \le C_p(C_0).
\end{aligned}
\end{equation}
Estimates \eqref{lecontroledegree2surlevarphhi0} applied to \eqref{decomposition1labonnenem} allow for a pointwise expansion of $\Phi^\varepsilon_{zz}$ : 
\begin{equation}
\label{ladecompositionsurlephiepsilonzetzz}
\begin{aligned}
\Phi^\varepsilon_z &= \varepsilon^\theta \left[ \Phi^1_z \left( \frac{z}{\varepsilon} \right) + Q^\varepsilon \left( \frac{z}{\varepsilon} \right) \right] + \varphi^\varepsilon_0, \\
\Phi^\varepsilon_{zz} &=\varepsilon^{\theta-1} \left[ \Phi^1_{zz} \left( \frac{z}{\varepsilon} \right) + Q^\varepsilon_z \left( \frac{z}{\varepsilon} \right) \right] + \left(\varphi^\varepsilon_0 \right)_z .
\end{aligned}
\end{equation}
\underline{ \bf Step 2 : Initial conditions} \newline
The relations \eqref{ladecompositionsurlephiepsilonzetzz}  yield when evaluated at $0$
\begin{equation}
\label{ledeveloppementen0surphizetphizz}
\begin{aligned}
&\Phi^\varepsilon_z (0) = \varepsilon^\theta \Phi^1_z (0) + \varepsilon^m Q^\varepsilon (0) + O \left( \varepsilon^{m+1} \right), \\
& \Phi^\varepsilon_{zz}(0) = \varepsilon^{\theta-1} \Phi^1_{zz} (0) + \varepsilon^{\theta-1} Q^\varepsilon_z(0) + O \left( \varepsilon^m \right).
\end{aligned}
\end{equation}
There hypothesis \ref{hypoth7}  stands as :
$$\begin{aligned}
\Phi^\varepsilon (0) &= 0, \\
\Phi^\varepsilon_z (0) &= \varepsilon^\theta \Phi^1_z (0) = \frac{\varepsilon^\theta}{2} \begin{pmatrix} 1 \\ i \\ 0 \end{pmatrix}, \\
\left[ \Omega^\varepsilon e^{-\lambda^\varepsilon} \right] (0) &=\frac{1}{\varepsilon^\theta} \Omega^\varepsilon(0) = -\frac{2}{\varepsilon}.
\end{aligned}$$
This implies $\n^\varepsilon (0) = \begin{pmatrix} 0 \\ 0 \\-1 \end{pmatrix}$, and since $\frac{ \Omega^\varepsilon}{2} = \left\langle \n^\varepsilon, \Phi^\varepsilon_{zz} \right\rangle$,
\begin{equation}
\label{lesvaleursen0}
\begin{aligned}
&\Phi^\varepsilon_z (0) = \frac{\varepsilon^\theta}{2} \begin{pmatrix} 1 \\ i \\ 0 \end{pmatrix}, \\
&\left\langle \Phi^\varepsilon_{zz}(0) , \begin{pmatrix} 0 \\ 0 \\ 1 \end{pmatrix} \right\rangle = \varepsilon^{\theta-1}.
\end{aligned}
\end{equation}
Comparing \eqref{ledeveloppementen0surphizetphizz} and \eqref{lesvaleursen0} yields 
$$
\begin{aligned}
&Q^\varepsilon (0) = O( \varepsilon),\\
&\left\langle Q^\varepsilon_z (0) , \begin{pmatrix} 0 \\ 0 \\1 \end{pmatrix} \right\rangle = O( \varepsilon).
\end{aligned}
$$
However as we pointed out in remark \ref{laremarquequiaideunpeuquandmeme}, $Q^\varepsilon$ is loosely defined. We can then evacuate the coefficients of order $\varepsilon$ into $\varphi^\varepsilon_0$ (which we will do without changing the notations) to obtain : 
\begin{equation}
\label{Qen0enfinnya}
\begin{aligned}
&Q^\varepsilon (0) = 0, \\
&\left\langle Q^\varepsilon_z (0) , \begin{pmatrix} 0 \\ 0 \\1 \end{pmatrix} \right\rangle = 0.
\end{aligned}
\end{equation}
To conclude we write $Q^\varepsilon \in \C^3$ as 
$$Q^\varepsilon := A^\varepsilon \begin{pmatrix} 1 \\i \\ 0 \end{pmatrix} + B^\varepsilon \begin{pmatrix} 1 \\ -i \\ 0 \end{pmatrix} + C^\varepsilon \begin{pmatrix} 0 \\ 0 \\ 1 \end{pmatrix}, $$
then \eqref{Qen0enfinnya} yields 
\begin{equation}
\label{Qen0encoordonnees}
\begin{aligned}
&A^\varepsilon (0) = B^\varepsilon (0) = C^\varepsilon(0)=0, \\
& C^\varepsilon_z (0) = 0.
\end{aligned}
\end{equation}
When taken at $0$, \eqref{exprimerrzenfonctiondesz}, see in appendix, yields 
\begin{equation}
\label{exprimerrzenfonctiondesz0}
\vec{R}^\varepsilon_z (0)= \varepsilon^\theta \left( H^\varepsilon(0) + i V^\varepsilon (0) \right) \begin{pmatrix} 1 \\ i \\ 0 \end{pmatrix} + i S^\varepsilon_z (0) \begin{pmatrix} 0\\0\\1 \end{pmatrix}.
\end{equation}
Estimate \eqref{lestimeefinaleenSR} then ensures that  $\upsilon^\varepsilon :=\varepsilon^\theta \left( H^\varepsilon(0) + i V^\varepsilon (0) \right)$  is uniformly bounded : 
$$\left| \upsilon^\varepsilon \right| \le C( C_0).$$
\underline{ \bf Step 3 : $\Phi^\varepsilon$ is conformal}\newline
We will linearize the conformality condition : 
$$\left\langle \Phi^\varepsilon_z, \Phi^\varepsilon_z\right\rangle =0.$$
Injecting \eqref{decomposition1labonnenem} in the former yields
\begin{equation}
\label{laconditiondeconformalitecasgeneral}
Q^2 B^\varepsilon - P^2 A^\varepsilon + PQ C^\varepsilon + A^\varepsilon B^\varepsilon + \frac{\left(C^\varepsilon \right)^2 }{2} =0.
\end{equation}
Applying  hypothesis \ref{hypoth6} and \eqref{Qen0encoordonnees} then yields : 
\begin{equation}
\label{Bepsilonaunzerodordre2}
z^2 \text{ divides } B^\varepsilon.
\end{equation}
Conformality also implies
$$\left\langle \Delta \Phi^\varepsilon, \Phi^\varepsilon_z \right\rangle = 0.$$
Injecting \eqref{equationofthelaplacian} and \eqref{ladecompositionsurlephiepsilonzetzz} into the former then yields
\begin{equation}
\label{vcbnjd}
\left\langle \widetilde h^\varepsilon, \varepsilon^\theta \left[ \Phi^1_z  + Q^\varepsilon \right]\left(\frac{z}{\varepsilon} \right) \right\rangle = \left\langle \widetilde h^\varepsilon, \varphi^\varepsilon_0 \right\rangle + \left\langle \Delta \Phi^\varepsilon, \Psi^\varepsilon_0 \right\rangle =: \Psi^\varepsilon_1,
\end{equation}
with $\Psi^\varepsilon_1$ satisfying thanks to \eqref{premiercontrolepsi0} and \eqref{lecontroledegree2surlevarphhi0}
\begin{equation}
\label{lecontrolesurlevarpsi1int}
\forall \upsilon >0 \quad \left| \Psi^\varepsilon_1 \right| \le C_\upsilon \chi^{2\theta+1- \upsilon}.
\end{equation}
Considering that 
%\eqref{lequationnyaesdhfvn} ensures  
$\left\langle \widetilde h^\varepsilon, \varepsilon^\theta \left[ \Phi^1_z  + Q^\varepsilon \right]\left(\frac{z}{\varepsilon} \right) \right\rangle$ is a polynomial of degree at most $2\theta$ in $z$ and $\overline{z}$, we can state :
$$ \left\langle \widetilde h^\varepsilon, \varepsilon^\theta \left[ \Phi^1_z  + Q^\varepsilon \right]\left(\frac{z}{\varepsilon} \right) \right\rangle = \sum_{p+q=0}^{2\theta} h^\varepsilon_{pq} \varepsilon^{2\theta-p-q} z^p \overline{z}^q.$$
Together \eqref{vcbnjd} and \eqref{lecontrolesurlevarpsi1int} yield :
$$ \forall \upsilon >0 \quad \sum_{p+q=0}^{2\theta} h^\varepsilon_{pq} \varepsilon^{2\theta-p-q} z^p \overline{z}^q = O\left( \chi^{2\theta+1- \upsilon} \right).$$
Applying lemma \ref{lelemmeauxiliaireenpassantzzb}  then yields : 
\begin{equation}
\label{equatoonintermieairetberjni,o}
\forall p,q  \quad \forall \upsilon>0 \quad h^\varepsilon_{pq} = O \left(\varepsilon^{1- \upsilon} \right).
\end{equation}
\underline{ \bf Step 4 : Computing $\widetilde h^\varepsilon$ }\newline
We compute 
$$\begin{aligned}
\varepsilon^\theta \Phi^1_{\overline{z}} \left( \frac{z}{\varepsilon} \right)+ \varepsilon^\theta \overline{ Q^\varepsilon \left( \frac{z}{\varepsilon} \right)} &= \varepsilon^\theta \left( \overline{  - \frac{P^2 \left(\frac{z}{\varepsilon} \right)}{2} + B^\varepsilon \left( \frac{z}{\varepsilon} \right) }\right) \begin{pmatrix} 1 \\ i \\ 0 \end{pmatrix} + \varepsilon^\theta \left(\overline{ \frac{Q^2 \left(\frac{z}{\varepsilon} \right)}{2} + A^\varepsilon \left( \frac{z}{\varepsilon} \right)} \right) \begin{pmatrix} 1 \\- i \\0 \end{pmatrix} \\
&+ \varepsilon^\theta  \left( \overline{ P\left(\frac{z}{\varepsilon} \right) Q \left( \frac{z}{\varepsilon} \right) + C^\varepsilon \left( \frac{z}{\varepsilon} \right) } \right) \begin{pmatrix} 0 \\ 0 \\ 1\end{pmatrix}.
\end{aligned}
$$
Hence
$$\begin{aligned}
\vec{R}^\varepsilon_z (0) \times \overline{ \varepsilon^\theta \left[\Phi^1_z + Q^\varepsilon \right] \left( \frac{z}{\varepsilon} \right) } &= -2 i\upsilon^\varepsilon  \varepsilon^\theta \overline{ \left[\frac{Q^2}{2} + A^\varepsilon \right]\left( \frac{z}{\varepsilon} \right)}  \begin{pmatrix}0 \\ 0 \\1 \end{pmatrix} + i \upsilon^\varepsilon \varepsilon^\theta \overline{ \left[PQ + C^\varepsilon \right]\left( \frac{z}{\varepsilon} \right)} \begin{pmatrix} 1 \\ i \\ 0 \end{pmatrix}\\& +S^\varepsilon_z (0)  \varepsilon^\theta \overline{ \left[-\frac{P^2}{2} + B^\varepsilon \right]\left( \frac{z}{\varepsilon} \right)}\begin{pmatrix} 1 \\ i \\0 \end{pmatrix}  - S^\varepsilon_z(0)  \varepsilon^\theta \overline{ \left[\frac{Q^2}{2} + A^\varepsilon \right]\left( \frac{z}{\varepsilon} \right)} \begin{pmatrix} 1 \\ -i \\ 0 \end{pmatrix},
\end{aligned}$$
and 
$$\begin{aligned}
S^\varepsilon_z (0)  \overline{ \varepsilon^\theta \left[\Phi^1_z + Q^\varepsilon \right] \left( \frac{z}{\varepsilon} \right) } &= S^\varepsilon_z (0)  \varepsilon^\theta \overline{ \left[-\frac{P^2}{2} + B^\varepsilon \right]\left( \frac{z}{\varepsilon} \right)}\begin{pmatrix} 1 \\ i \\0 \end{pmatrix}  + S^\varepsilon_z(0)  \varepsilon^\theta \overline{ \left[\frac{Q^2}{2} + A^\varepsilon \right]\left( \frac{z}{\varepsilon} \right)} \begin{pmatrix} 1 \\ -i \\ 0   \end{pmatrix} \\
&+S^\varepsilon_z(0)  \varepsilon^\theta \overline{ \left[PQ + C^\varepsilon \right]\left( \frac{z}{\varepsilon} \right)} \begin{pmatrix} 0 \\ 0 \\ 1   \end{pmatrix}. 
\end{aligned}$$
Then 
$$\begin{aligned}
\widetilde h^\varepsilon &= 2 \Im \left(  \left( S^\varepsilon_z (0)  \varepsilon^\theta \overline{ \left[-P^2 + 2B^\varepsilon \right]\left( \frac{z}{\varepsilon} \right)} +i \upsilon^\varepsilon \varepsilon^\theta \overline{ \left[PQ + C^\varepsilon \right]\left( \frac{z}{\varepsilon} \right)}  \right) \begin{pmatrix} 1 \\ i \\0\end{pmatrix}  \right. \\ & \left.+  \left( S^\varepsilon_z(0)  \varepsilon^\theta \overline{ \left[PQ + C^\varepsilon \right]\left( \frac{z}{\varepsilon} \right)}  - i\upsilon^\varepsilon  \varepsilon^\theta \overline{ \left[Q^2 +2 A^\varepsilon \right]\left( \frac{z}{\varepsilon} \right)}\right) \begin{pmatrix}0 \\ 0 \\ 1 \end{pmatrix} \right).
\end{aligned}$$
From this we deduce 
\begin{equation}
\label{lequationnyaesdhfvn}\begin{aligned}
\left\langle \widetilde h^\varepsilon, \varepsilon^\theta \left[ \Phi^1_z  + Q^\varepsilon \right]\left(\frac{z}{\varepsilon} \right) \right\rangle &= \left( S^\varepsilon_z (0)  \varepsilon^\theta \overline{ \left[-P^2 + 2B^\varepsilon \right]\left( \frac{z}{\varepsilon} \right)} +i \upsilon^\varepsilon \varepsilon^\theta\theta \overline{ \left[PQ + C^\varepsilon \right]\left( \frac{z}{\varepsilon} \right)}  \right) \varepsilon^\theta \left[- P^2 + 2 B^\varepsilon \right] \left( \frac{z}{\varepsilon} \right) \\
&+  \left( S^\varepsilon_z(0)  \varepsilon^\theta \overline{ \left[PQ + C^\varepsilon \right]\left( \frac{z}{\varepsilon} \right)}  - i\upsilon^\varepsilon  \varepsilon^\theta \overline{ \left[Q^2 +2 A^\varepsilon \right]\left( \frac{z}{\varepsilon} \right)}\right) \varepsilon^\theta \left[PQ + C^\varepsilon \right] \left( \frac{z}{\varepsilon} \right) \\
&= S^\varepsilon_z (0) \varepsilon^{2\theta} \left[ |P|^2 \left( |P|^2 + |Q|^2 \right) +2 \Re \left(PQ \overline{C^\varepsilon } -2 P^2 \overline{B^\varepsilon} \right) + 4 \left| B^\varepsilon \right|^2  + \left| C^\varepsilon \right|^2 \right]\left(\frac{z}{\varepsilon} \right) \\
&+ i\upsilon^\varepsilon \varepsilon^{2\theta} \left[-P \overline{Q} \left( |P|^2 + |Q|^2 \right) - P^2 \overline{ C^\varepsilon} +2 B^\varepsilon \overline{PQ} -2 PQ \overline{A^\varepsilon} \right. \\ 
&\left. -C^\varepsilon \overline{Q^2} +2 B^\varepsilon \overline{ C^\varepsilon} -2 C^\varepsilon \overline{A^\varepsilon} \right] \left(\frac{z}{\varepsilon} \right).
%&=\varepsilon^{2m} \left[ P\left( |P|^2 + |Q|^2 \right) \left( S^\varepsilon_z (0)\overline{P}  -i \upsilon^\varepsilon \overline{Q} \right) \right]
\end{aligned}
\end{equation}
Studying \eqref{lequationnyaesdhfvn} with \eqref{Qen0encoordonnees}, \eqref{Bepsilonaunzerodordre2}, \eqref{vcbnjd}, \eqref{lecontrolesurlevarpsi1int} and hypothesis \ref{hypoth6} in mind, we can write 
\begin{equation}
\label{yuiujhkohoijopijhko}
\left\langle \widetilde h^\varepsilon, \varepsilon^\theta \left[ \Phi^1_z  + Q^\varepsilon \right]\left(\frac{z}{\varepsilon} \right) \right\rangle =-i \upsilon^\varepsilon \varepsilon^{2\theta-1} z  + O(r^2),
\end{equation}
which implies $h^\varepsilon_{1,0} = -i \upsilon^\varepsilon$, and in turn thanks to \eqref{equatoonintermieairetberjni,o} : 
\begin{equation}
\label{leupsillonestbiencontrolecava}
\forall s >0 \quad \upsilon^\varepsilon = O \left( \varepsilon^{1-s} \right).
\end{equation}
Then, \eqref{vcbnjd}, \eqref{lecontrolesurlevarpsi1int} and  \eqref{lequationnyaesdhfvn} give us : 
\begin{equation}
\label{lequationnyaesdhfvnbis}\begin{aligned}
\left\langle \widetilde h^\varepsilon, \varepsilon^\theta \left[ \Phi^1_z  + Q^\varepsilon \right]\left(\frac{z}{\varepsilon} \right) \right\rangle &= S^\varepsilon_z (0) \varepsilon^{2\theta} \left[ |P|^2 \left( |P|^2 + |Q|^2 \right) +2 \Re \left(PQ \overline{C^\varepsilon } -2 P^2 \overline{B^\varepsilon} \right) + 4 \left| B^\varepsilon \right|^2  + \left| C^\varepsilon \right|^2 \right]\left(\frac{z}{\varepsilon} \right) \\ &+ O \left( \chi^{2\theta+1-\upsilon} \right).
%&=\varepsilon^{2m} \left[ P\left( |P|^2 + |Q|^2 \right) \left( S^\varepsilon_z (0)\overline{P}  -i \upsilon^\varepsilon \overline{Q} \right) \right]
\end{aligned}
\end{equation}
A similar process on the remaining polynomial allows us to state 
\begin{equation}
\label{lesepsilonestviencontrole}
\forall \upsilon>0 \quad S^\varepsilon_z (0) = O \left( \varepsilon^{1- \upsilon} \right).
\end{equation}
\underline{ \bf Step 5 : Conclusion} \newline
From \eqref{premieredecompoSR}, \eqref{leupsillonestbiencontrolecava} and \eqref{lequationnyaesdhfvnbis} we deduce : 
\begin{equation}
\label{aupremierordrecasaute}
\begin{aligned}
&\forall \upsilon >0 \quad \left|\vec{R}^\varepsilon_z \right| + \left| S^\varepsilon_z \right| \le C_\upsilon \chi^{1-\upsilon},\\
&\forall p< \infty \quad \left\| \nabla \vec{R}^\varepsilon_z \right\|_{L^p \left( \D \right) } + \left\| \nabla S^\varepsilon_z \right\|_{L^p \left( \D \right) } \le C_p.
\end{aligned}
\end{equation}
Inequality (119) from \cite{bibnmheps} then yields :
\begin{equation}
\label{aupremierordrechasaute}
\forall \upsilon >0 \quad \left|H^\varepsilon e^{\lambda^\varepsilon} \right| \le C_\upsilon \chi^{1-\upsilon}.
\end{equation}
Letting \eqref{aupremierordrechasaute} converge away from $0$ gives, thanks to hypothesis \ref{hypoth3}, the following : 
$$\forall \upsilon >0 \quad \left|H^0 e^{\lambda^0} \right| \le C_\upsilon r^{1-\upsilon}.$$
However since $\Phi^0$ is assumed to have a branch point of order $\theta+1$ at $0$, by definition,  $e^{\lambda^0} \sim C r^\theta$, which means
\begin{equation}\label{yyyyyyyyyyyyyyyya} \forall \upsilon>0 \quad \left| H^0 \right| \le C_\upsilon r^{1-\theta - \upsilon}.\end{equation}
By definition of $\alpha$ (see \eqref{lehetlealpha}), $H^0 \simeq r^{-\alpha}$. Since $\alpha \in \mathbb{Z}$, \eqref{yyyyyyyyyyyyyyyya} ensures : 
\begin{equation}
\label{lapremiereconclusionsuralpha}
\alpha \le \theta-1.
\end{equation}
This concludes the proof of the desired result. 
 \end{proof}
In the continuity of the previous proof we can improve on a convergence result obtained in \cite{bibnmheps} :
\begin{theo}
\label{laconvergenceamelioreethbjni}
Let $\Phi^k \, : \,\Sigma  \rightarrow \R^3$ be a sequence of Willmore immersions satisfying the hypotheses of theorem  \ref{energyquandberriv}. Assume further that at each concentration point a simple minimal bubble is blown. Then $\Phi^k \rightarrow \Phi^0$ $C^{3, \eta}$ for all $\eta <1$.
\end{theo}
\begin{proof}
As before we can reason locally, under hypotheses \ref{hypoth1}-\ref{hypoth7}, and will continue from \ref{lapremiereconclusionsuralpha}.
Injecting \eqref{aupremierordrecasaute} into \eqref{equationofthelaplacian} ensures :
\begin{equation}
\label{lelaplacienestbiencontrole}
\begin{aligned}
&\forall \upsilon>0 \quad \left| \Delta \Phi^\varepsilon \right| \le C_\upsilon \chi^{\theta+1 - \upsilon} \\
&\forall p< \infty \quad \left\| \frac{\nabla \left(\Delta \Phi^\varepsilon \right)}{ \chi^\theta} \right\|_{L^p \left( \D \right) } \le C_p.
\end{aligned}
\end{equation} 
We can then compute 
\begin{equation}
\label{ondecomposeleHphiaunmeilleurordre}
\begin{aligned}
&H^\varepsilon \Phi^\varepsilon_z = \frac{\Phi^\varepsilon_{z \zb} \times \Phi^\varepsilon_z }{i \left| \Phi^\varepsilon_z \right|^2}, \\
&\nabla \left(H^\varepsilon \Phi^\varepsilon_z \right) =  \frac{\nabla \left(\Phi^\varepsilon_{z \zb} \right) \times \Phi^\varepsilon_z }{i \left| \Phi^\varepsilon_z \right|^2} +  \frac{\Phi^\varepsilon_{z \zb} \times \nabla \left( \Phi^\varepsilon_z \right) }{i \left| \Phi^\varepsilon_z \right|^2} - \left( \left\langle  \nabla \Phi^\varepsilon_z, \Phi_{\zb} \right\rangle + \left\langle \Phi_z, \nabla \Phi_{\zb} \right\rangle \right) \frac{\Phi^\varepsilon_{z \zb} \times \Phi^\varepsilon_z }{i \left| \Phi^\varepsilon_z \right|^4}.
\end{aligned}
\end{equation}
Combining \eqref{lebonharnackenmenfin}, \eqref{ladecompositionsurlephiepsilonzetzz} and \eqref{lelaplacienestbiencontrole} yields : 
\begin{equation} 
\label{lehestderivableaunordredeplus1}
\begin{aligned}
&\forall \upsilon >0 \quad  \left\| \frac{H^\varepsilon \Phi^\varepsilon_z}{\chi^{1-\upsilon} } \right\|_{L^\infty \left( \D \right)} \le C_\upsilon, \\
&\forall p< \infty \quad \left\| \nabla \left(H^\varepsilon \Phi^\varepsilon_z \right) \right\|_{L^p \left(\D\right)} \le C_p.
\end{aligned}
\end{equation}
Consequently, injecting \eqref{lehestderivableaunordredeplus1} into \eqref{systemenRSPhiannexmieuxarticleprecedent} and applying Calderon-Zygmund yields 
\begin{equation}
\label{gkieuhjfj}
\forall p < \infty  \quad \left\| \nabla S^\varepsilon \right\|_{W^{2,p} \left( \D \right)} + \left\| \nabla \vec{R}^\varepsilon \right\|_{W^{2,p} \left( \D \right)}  +  \left\| \nabla \Phi^\varepsilon \right\|_{W^{3,p} \left( \D \right)} \le C(C_0).
\end{equation}
Which proves theorem \ref{laconvergenceamelioreethbjni} thanks to classical embeddings.
\end{proof}
\begin{remark}
We can further our expansions to the next order. Indeed  injecting \eqref{lehestderivableaunordredeplus1} and \eqref{aupremierordrecasaute} into \eqref{systemenRSPhiannexmieuxarticleprecedent} yields 
$$\begin{aligned}
&\forall \upsilon >0 \quad \left\|\frac{ \Delta S^\varepsilon }{\chi^{2-\upsilon }} \right\|_{L^\infty \left( \D \right) } + \left\|\frac{ \Delta \vec{R}^\varepsilon }{\chi^{2-\upsilon }} \right\|_{L^\infty \left( \D \right) } \le C_\upsilon, \\
& \forall p< \infty \quad \left\|\frac{ \Delta \nabla S^\varepsilon }{\chi} \right\|_{L^p \left( \D \right) } + \left\|\frac{ \Delta \nabla \vec{R}^\varepsilon }{\chi} \right\|_{L^p \left( \D \right) } \le C_p.
\end{aligned}$$
Applying corollary \ref{lecordedegrearbitraire} then yields 
\begin{equation}
\label{ladecompositionenrsordre23}
\begin{aligned}
S^\varepsilon_z &= S^\varepsilon_z (0) + s^\varepsilon_1 z + s^\varepsilon_2 z^2 + \sigma^\varepsilon_1, \\
\vec{R}^\varepsilon_z &= \vec{R}^\varepsilon_z (0) + \vec{r^\varepsilon_1} z + \vec{r^\varepsilon_2} z^2 + \vec{\rho^\varepsilon_1},
\end{aligned}
\end{equation}
where the $s^\varepsilon_j$ and the $\vec{r}^\varepsilon_j$ are uniformly bounded constants and $\sigma^\varepsilon_1$, $\rho^\varepsilon_1$ satisfy :
\begin{equation}
\begin{aligned}
&\forall \upsilon> 0 \quad \left|\frac{ \sigma^\varepsilon_1}{\chi^{3- \upsilon} } \right| +\left|\frac{ \nabla \sigma^\varepsilon_1}{\chi^{2- \upsilon} } \right|  + \left|\frac{ \vec{\rho^\varepsilon_1}}{\chi^{3- \upsilon} } \right| +\left|\frac{ \nabla \vec{\rho^\varepsilon_1}}{\chi^{2- \upsilon} } \right| \le C_\upsilon, \\
&\forall p< \infty \quad \left\| \frac{ \nabla^2 \sigma^\varepsilon_1 }{ \chi} \right\|_{L^p \left( \D \right) }+ \left\| \frac{ \nabla^2 \vec{\rho^\varepsilon_1} }{ \chi} \right\|_{L^p \left( \D \right) } \le C_p.
\end{aligned}
\end{equation}
Setting $\sigma^\varepsilon_0 = s^\varepsilon_1z + s^\varepsilon_2 z^2 + \sigma^\varepsilon_1$ and $ \vec{\rho}^\varepsilon_0 =  \vec{r^\varepsilon_1} z + \vec{r^\varepsilon_2} z^2 + \rho^\varepsilon_1$ yields
\begin{equation}
\begin{aligned}
&  \left|\frac{ \sigma^\varepsilon_0}{\chi}  \right| +\left|{ \nabla \sigma^\varepsilon_0}  \right|  + \left|\frac{ \vec{\rho}^\varepsilon_0}{\chi}  \right| +\left| \nabla \vec{\rho}^\varepsilon_0 \right| \le C, \\
&\forall p< \infty \quad \left\| \ \nabla^2 \sigma^\varepsilon_0 \right\|_{L^p \left( \D \right) }+ \left\|\nabla^2 \vec{\rho}^\varepsilon_0  \right\|_{L^p \left( \D \right) } \le C_p.
\end{aligned}
\end{equation}
We can then do all the reasonings from \eqref{equationofthelaplacian} to \eqref{lehestderivableaunordredeplus1} for better controls : 
\begin{equation} 
\label{lehestderivableaunordredeplus1sansleupsilon}
\begin{aligned}
  \left| \frac{H^\varepsilon \Phi^\varepsilon_z}{\chi}  \right| + \left| \frac{S^\varepsilon_z}{\chi}  \right|+ \left| \frac{\vec{R}^\varepsilon_z}{\chi}  \right| & \le C.%\\
%\forall p< \infty \quad \left\| \nabla \left(H^\varepsilon \Phi^\varepsilon_z \right) \right\|_{L^p \left(\D\right)} + \left\| \nabla S^\varepsilon_z \right\|_{L^p \left(\D\right)}+  \left\| \nabla \vec{R}^\varepsilon_z \right\|_{L^p \left(\D\right)} &\le C_p.
\end{aligned}
\end{equation}
Injecting this added regularity into the third equation of \eqref{systemenRSPhiannexmieuxarticleprecedent} ensures : 
\begin{equation}
\label{oncontrolelelaplacienphietladeriveedulaplacien}
\begin{aligned}
&\left| \frac{\Delta \Phi^\varepsilon}{ \chi^3} \right| + \left| \frac{ \Delta \nabla \Phi}{ \chi^2 } \right| \le C \\
& \forall p < \infty \quad \left\| \frac{\Delta \nabla^2 \Phi }{ \chi} \right\|_{L^p \left( \D \right)} \le C_p.
\end{aligned}
\end{equation}
With another application of corollary \ref{lecordedegrearbitraire} we can expand $\Phi^\varepsilon_z$ in the following manner :
\begin{equation}
\label{ladecompositionsurlephiepsilonzetzzordre3}
\begin{aligned}
\Phi^\varepsilon_z &= \varepsilon^\theta \left[ \Phi^1_z \left( \frac{z}{\varepsilon} \right) + Q^\varepsilon \left( \frac{z}{\varepsilon} \right) \right] + \varphi^\varepsilon_0, 
\end{aligned}
\end{equation}
with 
\begin{equation}
\label{lecontroledegree3surlevarphhi0}
\begin{aligned}
& \frac{ \left| \varphi^\varepsilon_0 \right|}{ \chi^{\theta+1} } +\frac{ \left| \nabla  \varphi^\varepsilon_0 \right|}{ \chi^{\theta} } +\frac{ \left| \nabla^2  \varphi^\varepsilon_0 \right|}{ \chi^{\theta-1} }\le C (C_0) \\
& \forall p< \infty \quad \left\| \frac{ \nabla^3 \varphi^\varepsilon_0}{ \chi^{\theta-2}} \right\|_{L^p \left( \D \right)} \le C_p(C_0).
\end{aligned}
\end{equation}
\end{remark}
\subsection{Proof of corollary \ref{lecaschengackstatter} :  }
\begin{proof}
We only need to show that the inverse of a Chen-Gackstatter torus has a branch point of second residue $\alpha = 2$. Let $\Psi \, : \,  (\C \backslash \mathbb{Z}^2) / \mathbb{Z}^2 \rightarrow \R^3$ be a parametrization of the Chen-Gackstatter torus, $p \in \R^3$ such that $d( p, \Psi) >1$, and $\Phi = \iota \circ ( \Psi - p)$, the studied inverse. Let $b$ denote its branch point. 

It is interesting to notice that the second residue of $\Phi$ can be read on its conformal Gauss map $Y$ defined as : 
$$Y_\Phi = H_\Phi \begin{pmatrix} \Phi \\ \frac{ |\Phi|^2-1}{2} \\ \frac{ |\Phi|^2+1}{2} \end{pmatrix} + \begin{pmatrix} \n_\Phi \\ \left\langle \n_\Phi , \Phi \right\rangle \\ \left\langle \n_\Phi , \Phi \right\rangle \end{pmatrix}.$$ Indeed $\Phi$ and $\n_\Phi$ are bounded around the branch point, and thus necessarily $$Y_\Phi \sim_b C_\Phi z^{-\alpha}.$$ Further since $\Phi$ and $\Psi$ differ by a conformal transform,  it has been shown in \cite{bibnmconfgaussmap}, that there exists a \emph{fixed} matrix $M \in SO(4,1)$ such that $Y_\Psi = M Y_\Phi$. This yields that necessarily $Y_\Psi \sim_b C_\Psi z^{-\alpha}$.
Hence, considering that $\Psi$ is minimal, we deduce that 
$$ Y_\Psi = \begin{pmatrix} \n_\Psi \\ \left\langle \n_\Psi, \Psi \right\rangle \\ \left\langle \n_\Psi, \Psi\right\rangle \end{pmatrix}.$$
Since $\n_\Psi$ is bounded, 
\begin{equation}
\label{250520191321}
\left\langle \n_\Phi, \Phi \right\rangle \sim_b C z^{-\alpha}.
\end{equation}
We will now use the Enneper-Weierstrass parametrization  of $\Psi$ and \eqref{250520191321} to compute the second residue of $\Psi$ at its branch point.
Chen-Gackstatter is a minimal surface of genus $1$ and of  Enneper-Weierstrass data centered on the branch point :  $ (f,g) =\left(2 \varphi(z) , A\frac{\varphi_z}{\varphi}(z) \right)$ (see \cite{MR661204}) where  $\varphi$ is the Weierstrass elliptic function, of elliptic invariants
$$\begin{aligned}
g_2 &= 60 \sum_{m,n = -\infty}^{\infty} \frac{1}{\left( m+ni \right)^4 } >0, \\
g_3 &= 0,
\end{aligned}$$
and $$A= \sqrt{ \frac{3 \pi}{2 g_2}} \in \R_+.$$
Then, $\varphi$ has the following expansion around $0$  (see \cite{MR1027834}) :
$$\begin{aligned}\varphi(z) &= \frac{1}{z^2} + O(z^2) \\
\varphi_z (z) &= \frac{-2}{z^3} + O (z).
\end{aligned}$$ 
Hence we can state that \begin{equation} \label{250520191413} \begin{aligned} \Phi &= 2 \Re \left( \int \frac{1}{z^2} \begin{pmatrix}  1 \\ i \\ 0 \end{pmatrix} - \frac{4 A^2}{z^4} \begin{pmatrix} 1 \\ -i \\ 0 \end{pmatrix} - \frac{4A}{z^3} \begin{pmatrix} 0 \\ 0 \\ 1 \end{pmatrix} + O( 1) dz  \right) \\ 
&= 2 \Re \left( \frac{4 A^2}{3z^3} \begin{pmatrix} 1 \\ -i \\ 0 \end{pmatrix} - \frac{1}{z} \begin{pmatrix} 1 \\ i \\ 0 \end{pmatrix} +\frac{2A}{z^2} \begin{pmatrix} 0 \\ 0 \\ 1 \end{pmatrix} + O(z)  \right) \\
&=  \left( \frac{4 {A}^2 }{3 \zb^3} - \frac{1}{z} \right) \begin{pmatrix} 1\\ i \\0 \end{pmatrix} + \left( \frac{4 A^2}{3z^3} - \frac{1}{\zb} \right) \begin{pmatrix} 1 \\ -i \\ 0 \end{pmatrix} + 2A \left( \frac{1}{z^2} + \frac{1 }{ \zb^2}  \right) \begin{pmatrix} 0 \\ 0 \\ 1 \end{pmatrix} + O(r).\end{aligned} \end{equation}
Similarly : 
$$\begin{aligned}
\Phi_z \times \Phi_{\zb} &= \left(  \frac{1}{z^2} \begin{pmatrix}  1 \\ i \\ 0 \end{pmatrix} - \frac{4 A^2}{z^4} \begin{pmatrix} 1 \\ -i \\ 0 \end{pmatrix} - \frac{4A}{z^3} \begin{pmatrix} 0 \\ 0 \\ 1 \end{pmatrix} + O( 1) \right) \times \left(  \frac{1}{\zb^2} \begin{pmatrix}  1 \\ -i \\ 0 \end{pmatrix} - \frac{4{A}^2}{\zb^4} \begin{pmatrix} 1 \\ i \\ 0 \end{pmatrix} - \frac{4{A}}{\zb^3} \begin{pmatrix} 0 \\ 0 \\ 1 \end{pmatrix} + O( 1) \right) \\
&=- \frac{4 {A} i }{z^2 \zb^3} \begin{pmatrix} 1 \\ i \\ 0 \end{pmatrix} + \frac{32 i A^4}{r^8} \begin{pmatrix} 0 \\ 0 \\1 \end{pmatrix} - \frac{16 A^3 i }{z^4 \zb^3} \begin{pmatrix} 1 \\ -i \\0 \end{pmatrix} - \frac{4Ai}{z^3 \zb^2} \begin{pmatrix} 1 \\ -i \\ 0 \end{pmatrix} - \frac{16 A^3 i}{z^3 \zb^4} \begin{pmatrix} 1 \\ i \\ 0 \end{pmatrix} + O \left( \frac{1}{r^4} \right)\\
&= \frac{32 i A^4}{r^8} \left( - \left( \frac{ z }{2A} + \frac{z^2 \zb }{ 8 A^3 } \right) \begin{pmatrix} 1 \\ i \\ 0 \end{pmatrix}  - \left( \frac{\zb }{2 A} + \frac{ \zb^2 z }{ 8 A^3 } \right) \begin{pmatrix}1 \\ -i \\ 0 \end{pmatrix} + \begin{pmatrix} 0 \\ 0 \\ 1 \end{pmatrix} + O ( r^4 ) \right),
\end{aligned}$$
and 
$$ \begin{aligned}
\left| \Phi_z \right|^2 &= \frac{32 A^4}{r^8}  + \frac{16 A^2}{r^6} + O( r^{-4} ) \\
&= \frac{32 A^4}{r^8}  \left( 1 + \frac{r^2}{2 A^2} + O( r^4) \right),
\end{aligned}$$
which yields
\begin{equation}
\label{250520191408}
\n_\Phi = \left( 1 - \frac{r^2}{2 A^2} \right) \begin{pmatrix} 0 \\ 0 \\1 \end{pmatrix}  - \left( \frac{z}{2A} - \frac{z^2 \zb}{8 A^3 }  \right) \begin{pmatrix} 1 \\i \\ 0 \end{pmatrix} - \left( \frac{\zb}{2 A } - \frac{z \zb^2 }{8A^3} \right) \begin{pmatrix} 1 \\ - i \\ 0 \end{pmatrix} + O( r^4).
\end{equation}
Combining \eqref{250520191413} and  \eqref{250520191408} ensures : 
\begin{equation}
\label{250520191414}
\left\langle \n_\Phi , \Phi \right\rangle = 2A \left( \frac{1}{z^2} + \frac{1}{\zb^2} \right) - \frac{4 A}{3 z^2} - \frac{4 A}{3 \zb^2}  + O \left( \frac{1}{r} \right) = \frac{2A}{3z^2} + \frac{2A}{3\zb^2}  + O \left( \frac{1}{r} \right).
\end{equation}
Considering \eqref{250520191414} in light of \eqref{250520191321} yields $\alpha = 2$. Applying theorem \ref{lecorintro}  concludes the proof.
\end{proof}
\subsection{An exploration of the consequences of theorem \ref{lecorintro}}
We consider $\Phi \, : \, \D \rightarrow \R^3$ a true Willmore conformal branched immersion with a single branch point at $0$, of multiplicity $\theta +1$, and of second residue $\alpha \le \theta-1$.  Then applying theorem \ref{theorempointremovdebernrivetvoila} we can expand $\Phi$ around $0$ in the following way : 
\begin{equation} \label{310520191442}\Phi (z)= 2 \Re \left( \frac{1}{2 (\theta +1 ) }  \begin{pmatrix} 1 \\ -i \\ 0 \end{pmatrix} z^{\theta+1} + \sum_{j= 1}^{\theta+1 - \alpha} \frac{\vec{A}_j}{\theta +1+j} z^{\theta+1  +j} + \left( \frac{ C z^{ \theta +1  -\alpha} \zb^{ \theta +1 } + \overline{C} z^{\theta +1   } \zb^{\theta +1 - \alpha} }{(\theta +1- \alpha)({\theta +1})} \right) \begin{pmatrix}0 \\ 0  \\1 \end{pmatrix} \right) +  \xi,\end{equation}
where $\xi$ satisfies 
$$ \nabla^j \xi  = O \left( |z|^{2 (\theta +1} - \alpha -j +1 - \nu \right) \text{ for all } \upsilon>0 \text{ and } j\le \theta +2 - \alpha.$$
Further if we do the conformal change of variables $Z^{\theta +1} = z^{\theta +1} + A z^{\theta+2}$, \eqref{310520191442} becomes
$$ \begin{aligned} 
\Phi (Z) &= 2 \Re \left( \frac{1}{2 (\theta +1 ) }  \begin{pmatrix} 1 \\ -i \\ 0 \end{pmatrix} Z^{\theta+1} + \left(\frac{\vec{A_1}}{\theta +2} +A \begin{pmatrix} 1  \\ -i \\ 0 \end{pmatrix}  \right)Z^{\theta +2}\right. \\ & \left.+ \left( \frac{ C Z^{ \theta +1  -\alpha} \overline{Z}^{ \theta +1 } + \overline{C} Z^{\theta +1   } \overline{Z}^{\theta +1 - \alpha} }{(\theta +1- \alpha)({\theta +1})} \right) \begin{pmatrix}0 \\ 0  \\1 \end{pmatrix} \right) + O\left(|Z|^{\theta+3} \right) .\end{aligned}$$
Thus up to doing a conformal change of charts we can assume that $\vec{A_1}$ has no component along $\begin{pmatrix} 1 \\ -i \\ 0 \end{pmatrix}$, meaning :
\begin{equation}
\label{310520191450}
\vec{A_1} = \frac{U}{2}\begin{pmatrix} 1 \\ i \\ 0 \end{pmatrix} +V \begin{pmatrix} 0 \\ 0 \\ 1 \end{pmatrix}.
\end{equation}
Then
\begin{equation} \label{310520191518}\Phi_z = \frac{1}{2} \begin{pmatrix} 1 \\ -i \\ 0 \end{pmatrix} z^\theta + \sum_{j= 1}^{\theta+1 - \alpha} \vec{A}_j z^{\theta  +j} + \left( \frac{ C}{\theta +1} z^{ \theta -\alpha} \zb^{ \theta +1 } + \frac{\overline{C}}{\theta +1- \alpha} z^{\theta  } \zb^{\theta +1 - \alpha} \right) \begin{pmatrix}0 \\ 0  \\1 \end{pmatrix} +  \xi_z,\end{equation}
Using $\Phi$ conformal, we can expand $\left\langle \Phi_z , \Phi_z \right\rangle $ to the order $z^{2\theta +1}$ and conclude that 
$$U = \left\langle \vec{A_1} , \begin{pmatrix} 1 \\ -i \\0 \end{pmatrix} \right\rangle = 0.$$
Then, wishing to expand the Gauss map $\n$ we compute : 
$$\left| \Phi_z \right|^2 = \frac{r^{2\theta} }{2}  + O\left( r^{2 \theta +2} \right),$$
$$\begin{aligned} \Phi_z \times \Phi_{\zb} &=  \left( \frac{1}{2} \begin{pmatrix} 1 \\ -i \\ 0 \end{pmatrix} +  Uz^{\theta +1} \begin{pmatrix} 0 \\ 0 \\1 \end{pmatrix} + O( r^{\theta +2} ) \right) \times  \left( \frac{1}{2} \begin{pmatrix} 1 \\ i \\ 0 \end{pmatrix} + \overline{ U} \zb^{\theta +1} \begin{pmatrix} 0 \\ 0 \\1 \end{pmatrix} + O( r^{\theta +2} ) \right) \\ 
&= \frac{ir^{2 \theta} }{2} \begin{pmatrix} 0 \\ 0 \\ 1 \end{pmatrix} - \frac{iUz^{\theta +1} \zb^\theta}{2} \begin{pmatrix} 1 \\ i \\ 0 \end{pmatrix} - \frac{i\overline{U}z^{\theta} \zb^{\theta+1}}{2} \begin{pmatrix} 1 \\ -i \\ 0 \end{pmatrix}  + O(r^{2\theta +2 } ).
\end{aligned}$$
Hence we can write 
\begin{equation}\label{310520191509} 
\n = \begin{pmatrix} 0 \\ 0 \\ 1 \end{pmatrix} - {U z }\begin{pmatrix} 1 \\ i \\ 0 \end{pmatrix} -  \overline{ U} \zb \begin{pmatrix} 1 \\ -i \\0 \end{pmatrix} + O (r^2).
\end{equation}
One can differentiate \eqref{310520191518}, and obtain 
$$ \Phi_{zz} = \frac{ \theta}{2} \begin{pmatrix} 1 \\ -i \\ 0 \end{pmatrix}z^{\theta-1} + \sum_{j=1}^{\theta +1- \alpha} \left( \theta +j \right) \vec{A}_j z^{\theta - 1 + j} + \left( \frac{ ( \theta - \alpha)C}{\theta +1} z^{\theta-1 - \alpha} \zb^{\theta+1} + \frac{ \theta \overline{C}}{\theta +1 - \alpha } z^{\theta-1} \zb^{\theta +1 - \alpha} \right) \begin{pmatrix} 0 \\ 0 \\1 \end{pmatrix} + \xi_{zz}.$$
Taking the scalar product with \eqref{310520191509} we find : 
$$ \frac{ \Omega}{2} = ( \theta +1) U z^\theta - \theta U z^\theta + O \left( r^{\theta +1} \right) = U z^\theta+ O ( r^{\theta +1}).$$
In theorem 4.11 of \cite{michelatclassifi}, A. Michelat and T. Rivière have shown that \begin{equation} \label{310520191539}\left\langle \vec{A_1}, C \begin{pmatrix} 0 \\ 0 \\ 1 \end{pmatrix} \right\rangle  = U C = 0.\end{equation}
This precise equality is found at (4.53), (4.54) and (4.69) (depending on whether $\theta \ge 4$, $\theta =3$ or $\theta =2$) of the aforementioned article and stems from the conformal relation $\left\langle \Phi_z , \Phi_z \right\rangle = 0$, expanded to the order $3 \theta +2 - \alpha$ (which requires furthering expansion \eqref{310520191518} in the way detailed in \cite{michelatclassifi}) in order to consider the first terms in $\zb$.
Equality \eqref{310520191539} implies notably that either $C =0$, and thus that $\alpha \le \theta-2$, of $U=0$, and thus that $\left|\Omega \right| e^{-\lambda} = O(r)$, meaning that  $0$ is an umbilic point. Further theorem  4.11 of \cite{michelatclassifi} states that Bryant's quartic (see in the introduction, or in \cite{bibnmconfgaussmap}) is then holomorphic across the branched point.

Theorem \ref{lecorintro} has the following corollary :
\begin{cor}
Let $\Phi_k$ be a sequence of Willmore immersions of a closed surface $\Sigma$ satisfying the hypotheses of theorem \ref{energyquandberriv}. 
Then at each concentration point $p \in \Sigma$ of multiplicity $\theta_p+1$ on which a simple minimal bubble is blown, either the second residue of $\Phi_\infty$ at $p$ satisfies $$\alpha_p \le \theta_p -2,$$ 
or $p$ is an umbilic point for $\Phi_\infty (\Sigma)$.

In both cases Bryant's quartic $\mathcal{Q}$ is holomorphic across those branch points.
\end{cor}

\renewcommand{\thesection}{\Alph{section}}
\setcounter{section}{0} 
\section{Appendix}
\tocless\subsection{Weighted Calderon-Zygmund}
Theorems \ref{theoBernardRiviere} and \ref{theoBernardRivierebis} are taken from Y. Bernard and T. Rivière's  ~\cite{bibpointremov} (Proposition C.2 and C.3). 
\begin{theo}
\label{theoBernardRiviere}
Let $u \in C^2 \left( \D \backslash \{ 0 \} \right)$ solve
$$ \Delta u(z) = \mu(z) f(z) \text{ in } \D ,$$
with $f \in L^p \left( \D \right)$ for $2 < p \le \infty$ and the weight $\mu$ satisfying for some $a \in \mathbb{N}$
$$|\mu ( z) | = O \left( |z|^{a} \right). $$
Then $$ u_z (z) = P(z) + |z|^a T(z) $$ with $P \in \C_a \left[ X \right] $  and $  T = O ( |z|^{1- \frac{2}{p}-\upsilon} )$ for all $\upsilon >0$. More precisely one has 
$$ \left\| \frac{T}{|z|^{1- \frac{2}{p}-\upsilon}} \right\|_{L^\infty( \D)} \le C_{\upsilon} \left( \left\| \frac{ \mu}{|z|^a} \right\|_{L^\infty( \D)} \| f\|_{L^p( \D)} + \| u \|_{C^1 \left( \partial \D \right) } \right).$$
\vspace{5mm}
Additionally if $\mu \in C^1 \left( \D \backslash \{ 0 \} \right)$, $a \neq 0$ and $$ \nabla \mu (z) = O \left( |z|^{a-1} \right)$$
Then : $$u_{zz}(z) = P_z + |z|^a Q$$ with $Q \in L^{p'} \left( \D \right)$ for all $p'<p$ and 
$$ \| Q \|_{L^{p'}( \D)} \le C_{p'} \left(  \left( \left\| \frac{ \mu}{|z|^a} \right\|_{L^\infty( \D)} + \left\| \frac{ \nabla \mu}{|z|^{a-1}} \right\|_{L^\infty( \D)} \right) \left\| f \right\|_{L^p( \D)} + \left\| u \right\|_{C^1 \left( \partial \D \right) }  \right).$$
In fact $Q = \frac{ \left( |z|^a T(z) \right)_z }{ |z|^a}$.
\end{theo}
\begin{remark}
Theorem  \ref{theoBernardRiviere} works with $a=0$, it is the classic Calderon-Zygmund theorem.
\end{remark}
\begin{proof}
 We will write the proof for $p=\infty$ to paint a picture of the involved reasonings and refer the reader to the original results for the general case ($p<\infty$).
Such an estimate can be written freely away from $0$. One can then assume $|z| \le \frac{1}{2}$.
Using Green's formula for the Laplacian and denoting $\nu$ the outer normal unit vector to $\partial \D$, one writes explicitely $u$ : 
\begin{equation}\label{lepetitgreen} \begin{aligned}
u_z (z) &= \frac{1}{2 \pi }\left( \int_{\partial \D}  \left( \frac{ \zb-\bar{x}}{|z-x|^2} \partial_\nu u(x) - u(x) \partial \frac{\zb-\bar{x}}{|z-x|^2} \right) d\sigma(x) - \int_{\D} \frac{\zb-\bar{x}}{|z-x|^2} \mu(x) f(x) dx \right) \\
&=: J_0(z) + J_1(z).
\end{aligned} \end{equation}
We first point out that  for $|x| > |z|$ one term can be expanded : 
$$ \frac{ \bar{x}- \zb }{ |x - z |^2} =  \sum_{m \ge 0 } z^m {x}^{- \left( m+1 \right) }.$$
Then we find : 
$$ \begin{aligned}
J_0(z) &=  \frac{1}{2 \pi} \sum_{m\ge 0 } \int_{\partial \D} \left(  z^m {x}^{- \left( m+1 \right) } \partial_\nu u(x) - u(x) \partial_\nu \left(  z^m {x}^{- \left( m+1 \right) } \right) \right) d\sigma(x) \\
&= \frac{1}{2\pi}  \sum_{m\ge 0 } z^m \int_0^{2 \pi} \left( (m+1) u(e^{i \theta} ) -  \left( \partial_\nu u \right) \left( e^{i\theta} \right) \right) e^{i (m+1) \theta } d \theta \\
&= \sum_{m \ge 0} C_m z^m
\end{aligned}$$ 
where the $C_m$ are complex valued constants depending only on the $C^1$ norm of $u$ along $\partial \D$. Since $u$ is by hypothesis bounded $C^1$ on the boundary of the unit disk by hypothesis,  $ \int_0^{2 \pi}  u(e^{i \theta}) e^{i (m+1) \theta } d \theta$ and $ \int_0 ^{2\pi} \partial_\nu u  \left( e^{i\theta} \right) e^{i (m+1) \theta } d \theta$ are bounded by the $C^1$ norm of  $u$ and thus the $C_m$ are growing at most linearly. Thus there exists a $ \delta >0$ such that for $ |z| \le \delta $,  and a $C >0$
$$ \begin{aligned}
&J_0 (z) = \sum_{m=0}^{a } C_m z^m + \sum_{m =  a  +1}^\infty C_m z^m \\
&\left| \sum_{m = a+1}^\infty C_m z^m \right| \le C |z|^{ a  +1}.
\end{aligned}$$
Then one writes \begin{equation} \label{expressiontype} \begin{aligned}  J_0(z) &=  \sum_{m=0}^{  a } C_m z^m  + |z|^{a +1} T_0(z) \text{ in } \D_\delta , \\
\text{ with }|z|^{ a +1} T_0 &=\sum_{m =  a +1}^\infty C_m z^m , \\ |T_0| &  \le C \|u\|_{C^1\left(\partial \D \right)} < \infty \text{ in } \D_\delta.  \end{aligned} \end{equation}
Now we notice $J_0$ is uniformly bounded, with bounds depending only on  $\|u \|_{C^1 \left( \partial \D )\right)}$,  on $\D_{\frac{1}{2}}$. We can then extend (\ref{expressiontype}) to the whole of $\D_{\frac{1}{2}}$ up to a constant adjustment.

One must now control $J_1$. We start by writing : 
$$\begin{aligned} J_1(z) &= \frac{1}{2 \pi} \int_{ \D_{2 |z|} } \frac{\zb - \bar{x} }{ | z-x|^2} \mu(x) f(x) dx + \frac{1}{2 \pi} \int_{\D \backslash \D_{2 |z|} } \frac{\zb - \bar{x} }{ | z-x|^2} \mu(x) f(x) dx \\
&=  \frac{1}{2 \pi} \int_{ \D_{2 |z|} } \frac{\zb - \bar{x} }{ | z-x|^2} \mu(x) f(x) dx + \frac{1}{2 \pi} \int_{\D \backslash \D_{2 |z|} }  \sum_{m=0}^\infty z^m {x}^{- \left( m+1 \right) } \mu(x) f(x) dx.
\end{aligned}$$
Now since on $\D \backslash \D_{2 |z|}$ , $$ \sum_{m=0}^\infty \left(\frac{|z|}{|x|} \right)^m \le \sum_{m=0}^\infty \frac{1}{2^m} < \infty,$$ we deduce 
$$\begin{aligned} J_1(z) &=  \frac{1}{2 \pi} \int_{\D_{2 |z|}} \frac{\zb - \bar{x} }{ | z-x|^2} \mu(x) f(x) dx +  \sum_{m=0}^\infty  \frac{1}{2 \pi} \int_{\D \backslash \D_{2 |z|} }z^m {x}^{- \left( m+1 \right) } \mu(x) f(x) dx.
\end{aligned}$$
We then introduce the following decomposition :
\begin{equation} \label{decompositionJ1} J_1(z) = I_1(z) + \sum_{m=0}^{  a  } I_1^m(z) + I_2^m(z) - \sum_{m=0}^{ a } I_1^m(z) + \sum_{m = a +1}^\infty I_2^m(z), \end{equation}
where :
$$\begin{aligned}
I_1(z) &:= \frac{1}{2 \pi} \int_{\D_{2 |z|}} \frac{\zb - \bar{x} }{ | z-x|^2} \mu(x) f(x) dx, \\
I_1^m(z) &:= \frac{1}{2 \pi} \int_{\D_{2 |z|}} z^m {x}^{- \left( m+1 \right) } \mu(x) f(x) dx ,\\
I_2^m(z) &:= \frac{1}{2 \pi} \int_{\D \backslash  \D_{2 |z|} } z^m {x}^{- \left( m+1 \right) } \mu(x) f(x) dx.
\end{aligned}$$
We notice $$ \sum_{m=0}^{ a } I_1^m(z) + I_2^m(z) =  \sum_{m=0}^{a } \frac{z^m}{2 \pi} \int_{\D } {x}^{- \left( m+1 \right) } \mu(x) f(x) dx, $$ and for $m\le  a $
$$ \begin{aligned}  \left| \int_{\D } {x}^{- \left( m+1 \right) } \mu(x) f(x) dx \right| &\le   \left\| \frac{ \mu}{|z|^a} \right\|_{L^\infty( \D)}  \|f\|_{L^\infty \left( \D \right) } \int_{\D  } |x|^{a-m-1}dx \\
&\le C  \left\| \frac{ \mu}{|z|^a} \right\|_{L^\infty( \D)}  \|f\|_{L^\infty \left( \D \right) }. \end{aligned} $$
which yields \begin{equation}\label{lepetitpolynom} \sum_{m=0}^a I_1^m(z) + I_2^m(z) =\sum_{m=0}^a A_m z^m\end{equation} with $A_m = \frac{1}{2 \pi}  \int_{\D } {x}^{- \left( m+1 \right) } \mu(x) f(x) dx$ a sequence of finite coefficients.
Besides : 
\begin{equation} \label{leI1} \begin{aligned}
| I_1(z) | &\le \frac{1}{2 \pi} \int_{\D_{2 |z|}} \frac{1 }{ | z-x|} |\mu(x)| |f(x) |dx \le \frac{1}{2 \pi} \left\|f \right \|_{L^\infty \left( \D \right) } \left\| \frac{ \mu}{|x|^a} \right\|_{L^\infty \left( \D \right) }   \int_{\D_{2 |z|}}  \frac{|x|^a}{ | z-x|}dx \\ 
&\le C_a |z|^a \|f \|_{L^\infty \left( \D \right) } \left\| \frac{ \mu}{|x|^a} \right\|_{L^\infty \left( \D \right) }  \int_{\D_{2 |z|} } \frac{1 }{ | z-x|} dx  \\
%&\le C_a |z|^a \|f \|_{L^\infty \left( \D \right) } \left\| \frac{ \mu}{|x|^a} \right\|_{L^\infty \left( \D \right) } \int_{B(z, 3 |z|)} \frac{1 }{ | z-x|} dx  \\
&\le  C |z|^{a+1} \|f \|_{L^\infty \left( \D \right) }  \left\| \frac{ \mu}{|x|^a} \right\|_{L^\infty \left( \D \right) }, \end{aligned}\end{equation}
and for $m \le a$,
\begin{equation} \label{leI1m} \begin{aligned}
| I_1^m(z) | & \le \frac{1}{2 \pi} \int_{\D_{2 |z|} } |z|^m |x|^{- \left( m+1 \right) } |\mu(x)|| f(x)| dx \\
&\le C \| f \|_{L^\infty \left( \D \right) }\left\| \frac{ \mu}{|x|^a} \right\|_{L^\infty \left( \D \right) } |z|^m \int_{\D_{2 |z|} }  |x|^{a- \left( m+1 \right) }dx \le   C_a |z|^{a+1} \|f \|_{L^\infty \left( \D \right) } \left\| \frac{ \mu}{|x|^a} \right\|_{L^\infty \left( \D \right) }. \end{aligned}\end{equation}
Finally for $a+2 \le m$ we write
\begin{equation}\label{leI2m} \begin{aligned}
|I_2^m(z) | &\le C |z|^m \| f \|_{L^\infty (\D) } \int_{\D \backslash \D_{2|z|} }  |x|^{ - \left(m +1 \right) } |\mu(x)| dx \\
&\le C|z|^m \left\| \frac{ \mu}{|x|^a} \right\|_{L^\infty \left( \D \right) } \| f \|_{L^\infty (\D) }\int_{\D \backslash \D_{2|z|} }  |x|^{a - \left(m +1 \right) }dx \\
& \le C|z|^{m}\left\| \frac{ \mu}{|x|^a} \right\|_{L^\infty \left( \D \right) } \| f \|_{L^\infty (\D) }\int_{2|z|}^1 r^{a-m}dr \\
&\le C|z|^m \left\| \frac{ \mu}{|x|^a} \right\|_{L^\infty \left( \D \right) } \| f \|_{L^\infty (\D) } \frac{ |2z|^{a+1-m} -1}{m-a-1} \\
&\le C\frac{1}{2^{m-a-1} \left( m-a-1\right)}|z|^{a+1} \left\| \frac{ \mu}{|x|^a} \right\|_{L^\infty \left( \D \right) } \| f \|_{L^\infty (\D) },
\end{aligned}\end{equation}
while $I_2^{a+1}$ is controlled in the following way
\begin{equation} \label{leI2aplus1} \begin{aligned}
|I_2^{a+1}(z) | &\le C |z|^{a+1} \| f \|_{L^\infty (\D) } \int_{\D \backslash \D_{2|z|} }  |x|^{ - \left(a +2 \right) } |\mu(x)| dx \\
&\le C|z|^{a+1} \left\| \frac{ \mu}{|x|^a} \right\|_{L^\infty \left( \D \right) } \| f \|_{L^\infty (\D) }\int_{\D \backslash \D_{2|z|} }  |x|^{-2 }dx \\
&\le C |z|^{a+1} \ln |z| \left\| \frac{ \mu}{|x|^a} \right\|_{L^\infty \left( \D \right) } \| f \|_{L^\infty (\D) }\\
&\le C_\upsilon |z|^{a+1-\upsilon} \left\| \frac{ \mu}{|x|^a} \right\|_{L^\infty \left( \D \right) } \| f \|_{L^\infty (\D) } \quad \forall \upsilon>0.
\end{aligned}
\end{equation}
Consequently \begin{equation} \label{lasum} \begin{aligned} \left| \sum_{m= a+1}^\infty I_2^m(z) \right| &\le  C_\upsilon |z|^{a+1-\upsilon} \left\| \frac{ \mu}{|x|^a} \right\|_{L^\infty \left( \D \right) } \| f \|_{L^\infty (\D) } \sum \frac{1}{2^{m-a-1}}\\ &\le    C_\upsilon |z|^{a+1-\upsilon} \left\| \frac{ \mu}{|x|^a} \right\|_{L^\infty \left( \D \right) } \| f \|_{L^\infty (\D) }. \end{aligned} \end{equation}
Injecting (\ref{lepetitpolynom})-(\ref{lasum}) into (\ref{decompositionJ1}) shows $J_1$ satisfies 
 \begin{equation} \label{expressiontype2} \begin{aligned}  J_1(z) &=  \sum_{m=0}^a A_m z^m  + |z|^{a+1 - \upsilon} T_1(z) \text{ in } \D 
\\ |T_1| &  \le C_\upsilon \left\| \frac{ \mu}{|x|^a} \right\|_{L^\infty \left( \D \right) } \| f \|_{L^\infty (\D) }.  \end{aligned} \end{equation}
To conclude, (\ref{expressiontype}) and (\ref{expressiontype2}) yield the desired result on $u_z$ when applied to (\ref{lepetitgreen}).
\vspace{3mm}
To prove the next part of the theorem one needs only notice that necessarily 
\begin{equation}
\label{resultatintermediaire0}
\begin{aligned}
|z|^a Q(z)& = \left( |z|^a T(z) \right)_z \\
&= \left( \sum_{m \ge a+1} C_m z^m \right)_z + I_{1 \, z} (z) + \sum_{m\ge a+1} I_{2 \, z}^m(z) - \sum_{0 \le m \le a} I_{1\, z }^m (z)
\end{aligned} 
\end{equation}
Now since we have shown that the $C_m$ have a mere linear growth,  $ \left( \sum_{m \ge a+1} C_m z^m \right)_z =   \sum_{m \ge a+1} mC_m z^{m-1}$ has the same strictly positive convergence radius. The same argument as before applies and yields the wanted control on the first term of (\ref{resultatintermediaire0}).
The other terms are estimated as before.
Indeed : 
\begin{equation}
\label{resultatintermediaire1}
\begin{aligned}
| I_{1\, z }^m (z)| &\le \left| \frac{1}{2\pi} \int_{\D \cap \D_{2 |x|} } mz^{m-1} x^{-m-1} \mu(x) f(x) dx + \frac{1}{2\pi} \frac{ \zb}{|z|} \int_{\partial \D_{2 |z|}} z^m x^{-m-1} \mu(x) f(x) dx \right|\\
&\le C_m |z|^{m-1}  \left\| \frac{ \mu}{|x|^a} \right\|_{L^\infty \left( \D \right) } \| f \|_{L^\infty (\D) } \int_{ \D_{2 |x|} }|x|^{a-m-1}dx + C|z|^a \left\| \frac{ \mu}{|x|^a} \right\|_{L^\infty \left( \D \right) } \| f\|_{L^\infty (\D) } \\
& \le C_a|z|^a \| f \|_{L^\infty (\D) },
\end{aligned}
\end{equation}
as long as $ m \le a$.
Similarly : 
\begin{equation}
\label{resultatintermediaire2}
\begin{aligned}
| I_{2\, z }^m (z)| &\le \left| \frac{1}{2\pi} \int_{\D \backslash \D_{2 |x|} } mz^{m-1} x^{-m-1} \mu(x) f(x) dx + \frac{1}{2\pi} \frac{ \zb}{|z|} \int_{\partial  \left( \D \backslash \D_{2|z|} \right)} z^m x^{-m-1} \mu(x) f(x) dx \right|\\
&\le C m  |z|^{m-1} \left\| \frac{ \mu}{|x|^a} \right\|_{L^\infty \left( \D \right) } \| f \|_{L^\infty (\D) } \int_{ \D \backslash \D_{2|z|} }|x|^{a-m-1}dx + \frac{C}{2^m}|z|^a  \left\| \frac{ \mu}{|x|^a} \right\|_{L^\infty \left( \D \right) } \| f\|_{L^\infty (\D) } \\
& \le \frac{C}{2^m}|z|^a \left\| \frac{ \mu}{|x|^a} \right\|_{L^\infty \left( \D \right) } \| f \|_{L^\infty (\D) } 
\end{aligned}
\end{equation}
for $m \ge a+2$; while 
\begin{equation}
\label{resultatintermediaire2bis}
\begin{aligned}
| I_{2\, z }^{a+1} (z)| 
&\le C_a   |z|^{a} \left\| \frac{ \mu}{|x|^a} \right\|_{L^\infty \left( \D \right) } \| f \|_{L^\infty (\D) } \left( \int_{ \D \backslash \D_{2|z|} }|x|^{-2}dx + 1 \right)\\
&\le C_a   |z|^{a} \left\| \frac{ \mu}{|x|^a} \right\|_{L^\infty \left( \D \right) } \| f \|_{L^\infty (\D) }  \ln |z|.
\end{aligned}
\end{equation}
The $I_1$ estimate is slightly more difficult to obtain. Differentiating we find $I_{1 \, z }= \frac{1}{2 \pi} \left( L(z) + K(z) \right)$ with 
$$K(z) = \frac{ \zb }{|z|} \int_{\partial \D_{2|z|}(0)}       \frac{ \zb - \bar{x} }{|z-x|^2} \mu(x) f(x) dx$$ 
% \chi_{\D_{\frac{1}{2}}}(z)%
and 
$$L(z) =\left( \Omega * f\mu\chi_{\D \cap \D_{2|z|}} \right) (z)$$ where $\Omega (y) = -2 \frac{ \bar{y^2}}{|y|^4}$.
One clearly finds : 
 \begin{equation} \label{resultatintermediaire3} |K(z)| \le C \| f \|_{L^\infty( \D)} \int_{\partial \D_{2|z|} } | \mu(x)|  \le C |z|^a \| f \|_{L^\infty (\D)} \left\| \frac{ \mu}{|x|^a} \right\|_{L^\infty \left( \D \right) }, \end{equation}
and 
$$ L(z) - \mu (z) \left( \Omega * f \chi_{\D_{2|z|}}\right)(z) = \int_{\D_{2 |z|}} \Omega (z-x) f(x) \left( \mu(x) - \mu(z) \right) dx.$$
Given $z$ in $\D$ let $S_z$ be the cone with apex $ \frac{z}{2}$ such that it contains $D_{\frac{|z|}{2}} $. For $ x \in S_z$, we have $2|z-x| > |z|$.  Hence : 
$$ \begin{aligned} \int_{S_x \cap \D_{2 |z|}} \Omega (z-x) f(x) \left(\mu(x) - \mu(z) \right) dx &\le C  \left(\frac{|\mu(z)|}{|z|^2} \int_{D_{2|z|}} |f(x)|dx + \frac{1}{|z|^2}   \int_{D_{2|z|}} |f(x) | | \mu(x)| dx \right) \\ &\le C  \left\| \frac{ \mu}{|x|^a} \right\|_{L^\infty \left( \D \right) } \| f \|_{L^\infty (\D)} |z|^a. \end{aligned}$$
Since $\mu \in C^{1}\left(\D \backslash \{ 0 \} \right)$, $ \mu \in C^1 \left( S^c_z \right)$. Thus for all $x \in S_z^c$ one can write : 
$$ | \mu(z) - \mu(x) | \le C  \left\| \frac{\nabla \mu }{ |x|^{a-1} } \right\|_{L^\infty \left( \D \right) } |z|^{a-1} |x-z|.$$ %since we assume $| \mu_z | \le |z|^{a-1}$. 
Accordingly : 
\begin{equation}
\label{resultatintermediaire4} \begin{aligned} \left| \int_{S_x^c \cap D_{2|z|}} \Omega ( z-x) f(x) \left( \mu(z) - \mu(x) \right) dx \right| &\le C \left\| \frac{\nabla \mu }{ |x|^{a-1} } \right\|_{L^\infty \left( \D \right) }  |z|^{a-1} \int_{D_{2|z|}} \frac{|f(z)|}{|z-x|} dx \\
&\le C \left\| \frac{\nabla \mu }{ |x|^{a-1} } \right\|_{L^\infty \left( \D \right) } |z|^{a-1} \| f \|_{L^\infty (\D) } \int_{D_{2|z|}} \frac{1}{|z-x|} dx \\
&\le C  \left\| \frac{\nabla \mu }{ |x|^{a-1} } \right\|_{L^\infty \left( \D \right) }|z|^{a-1} \| f \|_{L^\infty (\D) } \int_{B_{3|z|}(z)} \frac{1}{|z-x|} dx \\
&\le C \left\| \frac{\nabla \mu }{ |x|^{a-1} } \right\|_{L^\infty \left( \D \right) } |z|^a \| f \|_{L^\infty (\D)}. \end{aligned} \end{equation}
Combining (\ref{resultatintermediaire0}), (\ref{resultatintermediaire1}), (\ref{resultatintermediaire2}), (\ref{resultatintermediaire3}) and (\ref{resultatintermediaire4}) yields the desired result and concludes the proof.
\end{proof}

\begin{theo}
\label{theoBernardRivierebis}
Let $u \in C^2 \left( \D \backslash \{ 0 \} \right)$ solve
$$ \Delta u(z) = \mu(z) f(z) \text{ in } \D ,$$
with $f \in L^p \left( \D \right)$ for $2 < p \le \infty$ and the weight $\mu$ satisfying for some $a \in \mathbb{R}_+$
$$|\mu ( z) | = O \left( |z|^{a} \right). $$
Then $$ u_z (z) = P(z) + |z|^a T(z) $$ with $P \in \C_{\lceil a \rceil} \left[ X \right] $  and $  T = O ( |z|^{1- \frac{2}{p}-\upsilon} )$ for all $\upsilon >0$. Here $\lceil a \rceil$ is the upper integral part of $a$. More precisely one has 
$$ \left\| \frac{T}{|z|^{1- \frac{2}{p}-\upsilon}} \right\|_{L^\infty( \D)} \le C_{\upsilon} \left( \left\| \frac{ \mu}{|z|^a} \right\|_{L^\infty( \D)} \| f\|_{L^p( \D)} + \| u \|_{C^1 \left( \partial \D \right) } \right).$$
\vspace{5mm}
Additionally if $\mu \in C^1 \left( \D \backslash \{ 0 \} \right)$, $a \neq 0$ and $$ \nabla \mu (z) = O \left( |z|^{a-1} \right)$$
Then : $$u_{zz}(z) = P_z + |z|^a Q$$ with $Q \in L^{p'} \left( \D \right)$ for all $p'<p$ and 
$$ \| Q \|_{L^{p'}( \D)} \le C_{p'} \left(  \left( \left\| \frac{ \mu}{|z|^a} \right\|_{L^\infty( \D)} + \left\| \frac{ \nabla \mu}{|z|^{a-1}} \right\|_{L^\infty( \D)} \right) \left\| f \right\|_{L^p( \D)} + \left\| u \right\|_{C^1 \left( \partial \D \right) }  \right).$$
In fact $Q = \frac{ \left( |z|^a T(z) \right)_z }{ |z|^a}$.
\end{theo}
\begin{proof}
The proof is the same as in theorem \ref{theoBernardRiviere}, if $a$ is not an integer we simply split the terms in the sums at $\lceil a \rceil$, and we do not have to treat the $a+1$ term separately in that case (as we did in \eqref{leI2aplus1}).
\end{proof}

\begin{theo}
\label{theoBernardRivierebischanged}
Let $(u^\varepsilon)_{\varepsilon>0} \in C^2 \left( \D \backslash \{ 0 \} \right)$ solve
$$ \Delta u^\varepsilon(z) = \chi^a f^\varepsilon (z) \text{ in } \D ,$$
with $f^\varepsilon \in L^p \left( \D \right)$ for $2 \le p \le \infty$, $a \in \mathbb{R}$
and $\chi := \sqrt{ \varepsilon^2 + r^2 }$.
Then $$ u^\varepsilon_z (z) = P^\varepsilon(z) + \chi^a T^\varepsilon(z) $$ with $P^\varepsilon \in \C_{\lceil a \rceil} \left[ X \right] $  and $  T^\varepsilon = O ( \chi^{1- \frac{2}{p}-\upsilon} )$ for all $\upsilon >0$. Here $\lceil a \rceil$ is the upper integral part of $a$. More precisely one has 
$$ \left\| \frac{T^\varepsilon}{\chi^{1- \frac{2}{p}-\upsilon}} \right\|_{L^\infty( \D)} \le C_{\upsilon} \left( \| f^\varepsilon\|_{L^p( \D)} + \| u^\varepsilon \|_{C^1 \left( \partial \D \right) } \right).$$
\vspace{5mm}
Additionally : $$u^\varepsilon_{zz}(z) = P^\varepsilon_z + \chi^a Q^\varepsilon$$ with $Q^\varepsilon \in L^{p'} \left( \D \right)$ for all $p'<p$ and 
$$ \| Q^\varepsilon \|_{L^{p'}( \D)} \le C_{p'} \left( \left\| f^\varepsilon \right\|_{L^p( \D)} + \left\| u^\varepsilon \right\|_{C^1 \left( \partial \D \right) }  \right).$$
In fact $Q^\varepsilon = \frac{ \left( \chi^a T^\varepsilon(z) \right)_z }{ \chi^a}$.
\end{theo}
\begin{proof}
We first state that for all $a \in \R_+$, there exists $C_a \in \R_+^*$ such that \begin{equation} \label{lestimeepourcomprendrelechi} \frac{1}{C_a} \le \frac{\varepsilon^a + r^a}{\chi^a} \le C_a. \end{equation}
Here $C_a$ depends solely on $a$, and not on $\varepsilon$ or $r$.

We then write $$ \Delta u^\varepsilon = \left( \varepsilon^a + r^a \right) \frac{\chi^a }{ \varepsilon^a + r^a } f^\varepsilon =\left( \varepsilon^a + r^a \right) \widetilde f^\varepsilon,$$
where $\widetilde f^\varepsilon =  \frac{\chi^a }{ \varepsilon^a + r^a } f^\varepsilon$ satisfies, thanks to (\ref{lestimeepourcomprendrelechi}),
\begin{equation}
\label{oncontroleletildefepsilon}
\left\| \widetilde f^\varepsilon \right\|_{L^p \left( \D \right) } \le  C_a \left\| f^\varepsilon \right\|_{L^p \left( \D \right) .}
\end{equation}
We can then use Green's formula to write 
\begin{equation}
\begin{aligned}
u^\varepsilon_z (z)&= \frac{1}{2 \pi }\int_{\partial \D}  \left( \frac{ \zb-\bar{x}}{|z-x|^2} \partial_\nu u^\varepsilon(x) - u^\varepsilon(x) \partial \frac{\zb-\bar{x}}{|z-x|^2} \right) d\sigma(x) \\- & \frac{1}{2 \pi }\int_{\D} \frac{\zb-\bar{x}}{|z-x|^2} \left( \varepsilon^a + r^a \right) \widetilde f^\varepsilon (x) dx  \\
&=  \frac{1}{2 \pi }\int_{\partial \D}  \left( \frac{ \zb-\bar{x}}{|z-x|^2} \partial_\nu u^\varepsilon(x) - u^\varepsilon(x) \partial \frac{\zb-\bar{x}}{|z-x|^2} \right) d\sigma(x)  \\ &- \frac{1}{2 \pi }\int_{\D} \frac{\zb-\bar{x}}{|z-x|^2} \varepsilon^a  \widetilde f^\varepsilon (x) dx -  \frac{1}{2 \pi }\int_{\D} \frac{\zb-\bar{x}}{|z-x|^2}  r^a  \widetilde f^\varepsilon (x) dx \\
&= I_0^\varepsilon (z) + I_1^\varepsilon (z) + I_2^\varepsilon (z).
\end{aligned}
\end{equation}
We can then successively estimate the three terms as in the proof of theorem \ref{theoBernardRivierebis} and  write 
\begin{equation} \label{estimeeI0} I_0^\varepsilon (z) = P^\varepsilon_0(z) + z^{\lceil a \rceil +1}T^\varepsilon_0,\end{equation}
where $P^\varepsilon_0$ is a polynomial of degree at most $\lceil a \rceil$ and whose coefficients are bounded by $\left\|u^\varepsilon \right\|_{C^1 \left( \partial \D \right)}$, and 
$$\left\| T^\varepsilon_0 \right\|_{L^\infty \left( \D \right) } \le C \left\| u^\varepsilon \right\|_{C^1 \left( \D \right)}.$$
Working as for (\ref{expressiontype2}) we write 
\begin{equation} \label{estimeI1I2} \begin{aligned} I_1^\varepsilon (z) &= C^\varepsilon + \varepsilon^a T^\varepsilon_1(z) \\
I_2^\varepsilon (z) &= P^\varepsilon_2 + r^a T^\varepsilon_2(z) \end{aligned}\end{equation}
where $C^\varepsilon$ is a constant and $P^\varepsilon_2$  a polynomial of degree at most $\lceil a \rceil$, both bounded by $\left\|u^\varepsilon \right\|_{C^1 \left( \partial \D \right)}$, while 
$$\begin{aligned} \left\| \frac{ T^\varepsilon_1}{ r^{1 - \frac{2}{p} - \upsilon} } \right\|_{L^\infty \left( \D \right) } &\le C_\upsilon \left( \left\| \widetilde f^\varepsilon \right\|_{L^p \left( \D \right)} + \left\| u^\varepsilon \right\|_{C^1 \left( \D \right)} \right) \\ 
\left\| \frac{ T^\varepsilon_2}{ r^{1 - \frac{2}{p} - \upsilon} } \right\|_{L^\infty \left( \D \right) } &\le C_\upsilon \left( \left\| \widetilde f^\varepsilon \right\|_{L^p \left( \D \right)} + \left\| u^\varepsilon \right\|_{C^1 \left( \D \right)} \right) . \end{aligned}$$
In the end, combining (\ref{estimeeI0}) and (\ref{estimeI1I2}) yields
\begin{equation} \label{decompositionintermediaire1} u^\varepsilon_z = P^\varepsilon + \varepsilon^a T^\varepsilon_1 + r^a T^\varepsilon_3, \end{equation}
where $P^\varepsilon$ is a polynomial of degree at most $\lceil a \rceil$, $T^\varepsilon_1$ is as previously stated and  $T^\varepsilon_3$ still satisfies
$$\begin{aligned} \left\| \frac{ T^\varepsilon_3}{ r^{1 - \frac{2}{p} - \upsilon} } \right\|_{L^\infty \left( \D \right) } &\le C_\upsilon \left( \left\| \widetilde f^\varepsilon \right\|_{L^p \left( \D \right)} + \left\| u^\varepsilon \right\|_{C^1 \left( \D \right)} \right). \end{aligned}$$

Proceeding similarly then ensures that \begin{equation} \label{decompositionintermediaire2} u^\varepsilon_{zz} = P^\varepsilon_{z} + \varepsilon^a Q^\varepsilon_1 + r^a Q^\varepsilon_2, \end{equation}
where $Q^\varepsilon_1 = T^\varepsilon_{1, \, z } $ and $Q^\varepsilon_2 = \frac{ \left( r^a T^\varepsilon_3 \right)_z }{r^a}$ satisfy  for all $p'< p$
$$\begin{aligned}
\left\| Q^\varepsilon_1 \right\|_{L^{p'} \left( \D \right)} &\le C_{p'} \left( \left\| \widetilde f^\varepsilon \right\|_{L^p \left( \D \right) } + \left\| u^\varepsilon \right\|_{C^1 \left( \partial \D \right)} \right), \\
\left\| Q^\varepsilon_3 \right\|_{L^{p'} \left( \D \right)} &\le C_{p'} \left( \left\| \widetilde f^\varepsilon \right\|_{L^p \left( \D \right) } + \left\| u^\varepsilon \right\|_{C^1 \left( \partial \D \right)} \right).
\end{aligned}$$
Let us notice that the estimate on $Q^\varepsilon_1$ is not stricto sensu derived from the proof of theorem  \ref{theoBernardRivierebis} but from similar classical Calderon-Zygmund estimates.

From (\ref{decompositionintermediaire1}) we write $ u^\varepsilon_z = P^\varepsilon + \chi^a T^\varepsilon$ with $$T^\varepsilon = \frac{\varepsilon^a}{\chi^a} T^\varepsilon_1 + \frac{r^a }{\chi^a} T^\varepsilon_3$$ which then satisfies 
\begin{equation}
\label{lestimeeaufinalsurleTavecleschi}
\begin{aligned}
\left\| \frac{T^\varepsilon}{\chi^{1-\frac{2}{p}-\upsilon} } \right\|_{L^\infty \left(\D \right) } &\le \left\|\frac{\varepsilon^a}{\chi^a} \right\|_{L^\infty \left(\D \right) } \left\| \frac{T^\varepsilon_1}{\chi^{1-\frac{2}{p}-\upsilon} } \right\|_{L^\infty \left(\D \right) } + \left\|\frac{r^a}{\chi^a} \right\|_{L^\infty \left(\D \right) } \left\| \frac{T^\varepsilon_3}{\chi^{1-\frac{2}{p}-\upsilon} } \right\|_{L^\infty \left(\D \right) } \\
&\le  \left\|\frac{\varepsilon^a}{\chi^a} \right\|_{L^\infty \left(\D \right) } \left\| \frac{T^\varepsilon_1}{r^{1-\frac{2}{p}-\upsilon} } \right\|_{L^\infty \left(\D \right) } \left\|\frac{r^{1- \frac{2}{p} - \upsilon}}{\chi^{1- \frac{2}{p} - \upsilon}} \right\|_{L^\infty \left(\D \right) }   \\&+ \left\|\frac{r^a}{\chi^a} \right\|_{L^\infty \left(\D \right) } \left\| \frac{T^\varepsilon_3}{r^{1-\frac{2}{p}-\upsilon} } \right\|_{L^\infty \left(\D \right) }\left\|\frac{r^{1- \frac{2}{p} - \upsilon}}{\chi^{1- \frac{2}{p} - \upsilon}} \right\|_{L^\infty \left(\D \right) } \\
&\le C_\upsilon \left( \left\| \widetilde f^\varepsilon \right\|_{L^p \left( \D \right)} + \left\| u^\varepsilon \right\|_{C^1 \left( \D \right)} \right),
\end{aligned}
\end{equation}
using lemma \ref{lepetitlemmequiaide}.

From (\ref{decompositionintermediaire2}) we write $ u^\varepsilon_{zz} = P^\varepsilon_z + \chi^a Q^\varepsilon$ with $$Q^\varepsilon = \frac{\varepsilon^a}{\chi^a} Q^\varepsilon_1 + \frac{r^a }{\chi^a} Q^\varepsilon_3 = \frac{ \left( \chi^a T^\varepsilon \right)_z}{\chi^a}$$ which then satisfies 
\begin{equation}
\begin{aligned}
\left\| {Q^\varepsilon}  \right\|_{L^{p'} \left(\D \right) } &\le \left\|\frac{\varepsilon^a}{\chi^a} \right\|_{L^\infty \left(\D \right) } \left\|Q^\varepsilon_1\right\|_{L^{p'} \left(\D \right) } + \left\|\frac{r^a}{\chi^a} \right\|_{L^\infty \left(\D \right) } \left\|Q^{\varepsilon}_3 \right\|_{L^{p'} \left(\D \right) } \\
&\le C_{p'} \left( \left\| \widetilde f^\varepsilon \right\|_{L^p \left( \D \right)} + \left\| u^\varepsilon \right\|_{C^1 \left( \D \right)} \right),
\end{aligned}
\end{equation}
using lemma \ref{lepetitlemmequiaide}.
\end{proof}
\begin{remark}
\label{laremarquequiaideunpeuquandmeme}
We must point out that  the expansion offered by theorem \ref{theoBernardRivierebischanged} is by no means unique. Indeed if for instance $u_z = P^\varepsilon + \chi^m T^\varepsilon$, then one could readily write 
$$u_z = P^\varepsilon + \varepsilon^{m+1} + \chi^m \left( T^\varepsilon + \frac{\varepsilon^{m+1}}{\chi^m} \right)$$ with 
$T^\varepsilon + \frac{\varepsilon^{m+1}}{\chi^m}$ still satisfying (\ref{lestimeeaufinalsurleTavecleschi}).
\end{remark}

We give here a small lemma which can help one to understand the concrete impact of $\chi$ : 
\begin{lem}
\label{lepetitlemmequiaide}
For all $a$,$b \in \R_+$, there exists a constant $C_{a,b}$ such that 
$$\frac{\varepsilon^a r^b }{\chi^{a+b} } \le C_{a,b}.$$
\end{lem}

Theorem \ref{theoBernardRivierebischanged} can be applied several times to prove an increased regularity on the higher order terms :
\begin{lem}
\label{lelemmequimesertdheredite}
Let $u^\varepsilon \in C^2 \left( \D \backslash \{ 0 \} \right)$ such that 
$$\Delta u^\varepsilon = \chi^a f^\varepsilon, $$ with $f^\varepsilon \in L^\infty$ and 
$$ \Delta \left( \nabla u^\varepsilon \right) = \chi^{a-1} g^\varepsilon, $$ with $g^\varepsilon \in L^p $.
Then $$u^\varepsilon_z = P^\varepsilon + \mu^\varepsilon,$$ where $P^\varepsilon$ is a complex polynomial of degree at most $\lceil a \rceil$, and $\mu^\varepsilon$ such that 
$$\frac{\left| \mu^\varepsilon \right|}{\chi^{a +1 - \upsilon}} + \frac{\left| \nabla \mu^\varepsilon \right|}{\chi^{a - \frac{2}{p}- \upsilon}} \le C_\upsilon \left(  \left\| f^\varepsilon \right\|_{L^\infty \left( \D \right) }  +  \left\| g^\varepsilon \right\|_{L^p \left( \D \right) } + \left\| u^\varepsilon \right\|_{C^2 \left( \partial \D \right) }  \right),$$
and 
$$ \left\| \frac{ \nabla^2 \mu^\varepsilon }{ \chi^{a -1} } \right\|_{L^{p'} \left( \D \right)} \le C_{p'}\left(  \left\| f^\varepsilon \right\|_{L^\infty \left( \D \right) }  +  \left\| g^\varepsilon \right\|_{L^p \left( \D \right) } + \left\| u^\varepsilon \right\|_{C^2 \left( \partial \D \right) }  \right).$$
\end{lem}
\begin{proof}
We apply theorem \ref{theoBernardRivierebischanged} twice  and decompose $u_z$ and $\left( u_z \right)_z$ :
\begin{equation}
\begin{aligned}
u^\varepsilon_z &= P^\varepsilon_1 + \mu^\varepsilon_1 \\
\left( u^\varepsilon_z \right)_z &= P^\varepsilon_2 + \mu^\varepsilon_2,
\end{aligned}
\end{equation}
where 
\begin{equation}
\label{lepremiercontrolemu1mu2}
\begin{aligned}
\frac{ \left| \mu^\varepsilon_1 \right| }{ \chi^{a +1 - \upsilon} } + \frac{ \left| \mu^\varepsilon_2 \right| }{\chi^{a -\frac{2}{p} - \upsilon }} &\le C_{ \upsilon}\left(  \left\| f^\varepsilon \right\|_{L^\infty \left( \D \right) }  +  \left\| g^\varepsilon \right\|_{L^p \left( \D \right) } + \left\| u^\varepsilon \right\|_{C^2 \left( \partial \D \right) }  \right) \\
\left\| \frac{ \nabla \mu^\varepsilon_1}{ \chi^a} \right\|_{L^{p_1'} \left(\D \right)} +\left\| \frac{ \nabla \mu^\varepsilon_2}{ \chi^{a-1}} \right\|_{L^{p_2'} \left(\D \right)}  &\le  C_{p_1', p_2'}\left(  \left\| f^\varepsilon \right\|_{L^\infty \left( \D \right) }  +  \left\| g^\varepsilon \right\|_{L^p \left( \D \right) } + \left\| u^\varepsilon \right\|_{C^2 \left( \partial \D \right) }  \right),
\end{aligned}
\end{equation}
for all $p'_1 < \infty$ and $p'_2 < p$. We then enjoy two expressions for $u_{zz}$ : 
$$u_{zz} = P^\varepsilon_{1, \, z} + \mu^\varepsilon_{1, \, z } = P^\varepsilon_2 + \mu^\varepsilon_2.$$ Consequently 
$$ P^\varepsilon_{1, \, z } -P^\varepsilon_2 =  \mu^\varepsilon_2 - \mu^\varepsilon_{1, \, z },$$ which in turn, combined with \eqref{lepremiercontrolemu1mu2}, implies that 
$$\int_\D  \left| \frac{ P^\varepsilon_{1, \, z } -P^\varepsilon_2}{ \chi^a } \right|^sdz \le C_s\left(  \left\| f^\varepsilon \right\|_{L^\infty \left( \D \right) }  +  \left\| g^\varepsilon \right\|_{L^p \left( \D \right) } + \left\| u^\varepsilon \right\|_{C^2 \left( \partial \D \right) }  \right),$$
for all $s < \infty$. We decompose  $$P^\varepsilon_{1, \, z } -P^\varepsilon_2 = \sum_{q=0}^{ \lfloor a \rfloor} p^\varepsilon_q z^q,$$  and can state for a given $R_0>0$
$$\int_{\D_{\epsilon R_0}} \left| \frac{  \sum_{q=0}^{ \lfloor a \rfloor} p^\varepsilon_q z^q}{ \chi^a } \right|^s dz  \le \int_{\D} \left| \frac{  \sum_{q=0}^{ \lfloor a \rfloor} p^\varepsilon_q z^q}{ \chi^a } \right|^p  \le C_s\left(  \left\| f^\varepsilon \right\|_{L^\infty \left( \D \right) }  +  \left\| g^\varepsilon \right\|_{L^p \left( \D \right) } + \left\| u^\varepsilon \right\|_{C^2 \left( \partial \D \right) }  \right).$$
Changing variables yields 
$$\int_{\D_{ R_0}} \left| \frac{  \sum_{q=0}^{ \lfloor a \rfloor} \frac{p^\varepsilon_q}{\varepsilon^{a- q - \frac{2}{p}}} z^q}{ \sqrt{ 1 +r^2}^a } \right|^s dz \le C_s\left(  \left\| f^\varepsilon \right\|_{L^\infty \left( \D \right) }  +  \left\| g^\varepsilon \right\|_{L^p \left( \D \right) } + \left\| u^\varepsilon \right\|_{C^2 \left( \partial \D \right) }  \right).$$
And since on $\D_{\R_0}$, $\frac{1}{1+r^2} \ge \frac{1}{1+ R_0^2}$, we deduce 
\begin{equation} \label{lecontrolenya} \int_{\D_{ R_0}} \left|  \sum_{q=0}^{ \lfloor a \rfloor} \frac{p^\varepsilon_q}{\varepsilon^{a- q - \frac{2}{p}}} z^q \right|^s dz \le C_{s, R_0}\left(  \left\| f^\varepsilon \right\|_{L^\infty \left( \D \right) }  +  \left\| g^\varepsilon \right\|_{L^p \left( \D \right) } + \left\| u^\varepsilon \right\|_{C^2 \left( \partial \D \right) }  \right).\end{equation}
It is now important to notice that the left-hand term in \eqref{lecontrolenya} is in fact a polynomial in $R_0$, which is uniformly bounded in $\varepsilon$ on compacts of $\C$. All its coefficients are thus uniformly bounded in $\varepsilon$, and straightforward computations then yield : 
$$ \forall s < \infty \quad \forall j \le \lfloor a \rfloor \quad  \forall \varepsilon>0 \quad   \left| \frac{p^\varepsilon_q}{\varepsilon^{a- q - \frac{2}{s}}} \right| \le C_s \left(  \left\| f^\varepsilon \right\|_{L^\infty \left( \D \right) }  +  \left\| g^\varepsilon \right\|_{L^p \left( \D \right) } + \left\| u^\varepsilon \right\|_{C^2 \left( \partial \D \right) }  \right)$$
which thanks to lemma \ref{lepetitlemmequiaide} translates on $P^\varepsilon_{1, \, z } -P^\varepsilon_2$ as 
\begin{equation}
\label{oncontroleladiffdespoly}
\forall s < \infty \quad \left| \frac{P^\varepsilon_{1, \, z } -P^\varepsilon_2}{ \chi^{a - \frac{2}{s}}} \right| \le C_s\left(  \left\| f^\varepsilon \right\|_{L^\infty \left( \D \right) }  +  \left\| g^\varepsilon \right\|_{L^p \left( \D \right) } + \left\| u^\varepsilon \right\|_{C^2 \left( \partial \D \right) }  \right), 
\end{equation}
and
\begin{equation}
\label{oncontroleladiffdespolydev}
\forall s < \infty \quad \left| \frac{\left( P^\varepsilon_{1, \, z } -P^\varepsilon_2\right)_z}{ \chi^{a - \frac{2}{s}-1}} \right| \le C_s\left(  \left\| f^\varepsilon \right\|_{L^\infty \left( \D \right) }  +  \left\| g^\varepsilon \right\|_{L^p \left( \D \right) } + \left\| u^\varepsilon \right\|_{C^2 \left( \partial \D \right) }  \right).
\end{equation}
Now since $\mu^\varepsilon_{1, \, z} =  \mu^\varepsilon_2 - \left( P^\varepsilon_{1, \, z } -P^\varepsilon_2 \right)$ we can combine \eqref{lepremiercontrolemu1mu2} and  \eqref{oncontroleladiffdespoly} to find for all $\upsilon >0$
\begin{equation}
\label{lepremiercontroleidylliquesurmu1}
\left| \frac{\mu^\varepsilon_{1 \, z}}{ \chi^{a - \frac{2}{p} - \upsilon}} \right| \le C_\upsilon \left(  \left\| f^\varepsilon \right\|_{L^\infty \left( \D \right) }  +  \left\| g^\varepsilon \right\|_{L^p \left( \D \right) } + \left\| u^\varepsilon \right\|_{C^2 \left( \partial \D \right) }  \right).
\end{equation}
Further since $\mu^\varepsilon_{ 1 \, zz} = \mu^\varepsilon_{2 \, z } - \left( P^\varepsilon_{1, \, z } -P^\varepsilon_2 \right)_z$, \eqref{lepremiercontrolemu1mu2} and  \eqref{oncontroleladiffdespolydev} yield for all $p'<p$ :
\begin{equation}
\label{ledeuxiemecontroleidylliquesurmu1}
\left\| \frac{\mu^\varepsilon_{1 \, zz}}{ \chi^{a - 1}} \right\|_{L^{p'} \left( \D \right)} \le C_{p'} \left(  \left\| f^\varepsilon \right\|_{L^\infty \left( \D \right) }  +  \left\| g^\varepsilon \right\|_{L^p \left( \D \right) } + \left\| u^\varepsilon \right\|_{C^2 \left( \partial \D \right) }  \right).
\end{equation}

Applying similarly theorem \ref{theoBernardRivierebischanged} to $ u^\varepsilon_{\zb}$ yields controls akin to \eqref{lepremiercontroleidylliquesurmu1} and \eqref{ledeuxiemecontroleidylliquesurmu1} on the missing terms in the gradient and the Hessian, which concludes the proof.
\end{proof}
A cautious reader might have noticed that we have in fact proved the following lemma : 
\begin{lem}
\label{lelemmeauxiliaireenpassantz}
Let $u  \in \mathbb{N}$, $v\ge u$ and $P^\varepsilon = \sum_{j=0}^u p^\varepsilon_j z^j \in \C_u[X]$ such that 
$$\forall p< \infty \quad \frac{ P^\varepsilon}{\chi^v} \in L^p.$$
Then $$ \forall \nu >0 \quad \forall j\le u \quad \left| \frac{ p^\varepsilon_j}{\varepsilon^{v-j- \nu } } \right| \le C_\nu .$$
\end{lem}
We will also use a corresponding result for polynomials in $z$ and $\zb$ : 
\begin{lem}
\label{lelemmeauxiliaireenpassantzzb}
Let $u  \in \mathbb{N}$, $v\ge u$ and $P^\varepsilon = \sum_{i+j=0}^u p^\varepsilon_{i,j} z^i \zb^j $ such that 
$$\forall p< \infty \quad \frac{ P^\varepsilon}{\chi^v} \le C.$$
Then $$ \forall \nu >0 \quad \forall i+j\le u \quad \left| \frac{ p^\varepsilon_{i,j}}{\varepsilon^{v-i-j } } \right| \le C_\nu .$$
\end{lem}
Applying lemma \ref{lelemmequimesertdheredite} several times yields : 
\begin{cor}
\label{lecordedegrearbitraire}
Let $u^\varepsilon \in C^2 \left( \D \backslash \{ 0 \} \right)$ such that, for $a \ge t$ 
$$\begin{aligned} \Delta u^\varepsilon &= \chi^a f^\varepsilon_0, \\
\Delta \nabla u^\varepsilon &= \chi^{a-1} f^\varepsilon_1 \\
\dots& \\
\Delta \nabla^t u^\varepsilon & = \chi^{a-t} f^\varepsilon_t
\end{aligned} 
$$ 
with $f^\varepsilon_j \in L^\infty \left( \D \right)$ for $j\le t-1$ and $f^\varepsilon_t \in L^p \left(\D \right)$.
Then $$u^\varepsilon_z = P^\varepsilon + \mu^\varepsilon,$$ where $P^\varepsilon$ is a complex polynomial of degree at most $\lceil a \rceil$, and $\mu^\varepsilon$ such that 
$$\frac{\left| \mu^\varepsilon \right|}{\chi^{a +1 - \upsilon}} + \frac{\left| \nabla \mu^\varepsilon \right|}{\chi^{a - \upsilon}} + \dots + \frac{\left| \nabla^t \mu^\varepsilon \right|}{\chi^{a+1 -t- \frac{2}{p}- \upsilon}}   \le C_\upsilon \left(  \sum_{q=0}^t \left\| f^\varepsilon_q \right\|_{L^\infty \left( \D \right) }  + \left\| u^\varepsilon \right\|_{C^{t+1} \left( \partial \D \right) }  \right),$$
and 
$$ \left\| \frac{ \nabla^{t+1} \mu^\varepsilon }{ \chi^{a-t} } \right\|_{L^{p'} \left( \D \right)} \le C_{p'} \left(  \sum_{q=0}^t \left\| f^\varepsilon_q \right\|_{L^\infty \left( \D \right) }  + \left\| u^\varepsilon \right\|_{C^{t+1} \left( \partial \D \right) }  \right).$$
\end{cor}
\begin{proof}
The proof is a recurrence whose initialization is theorem \ref{theoBernardRivierebischanged} and whose heredity is obtained by applying lemma is \ref{lelemmequimesertdheredite} to the $\nabla^s u^\varepsilon$.
\end{proof}
\tocless\subsection{Auxiliary formulas}
In the following, given $\Phi$ a Willmore conformal immersion, and $\Lr$, $S$, $\vec{R}$ defined in  \eqref{LRSintro}, we wish to prove  : 
\begin{equation}
\label{exprimerrzenfonctiondesz}
\vec{R}_z = 2 \left( H + i V \right) \Phi_z -i S_z \n,
\end{equation}
where $V =\frac{1}{2} \left\langle \vec{L}, \n \right\rangle.$
Indeed, since 
$$\begin{aligned}
S_z &= \left\langle \vec{L}, \Phi_z\right\rangle \\
\vec{R}_z &= \vec{L} \times \Phi_z + 2 H \Phi_z,
\end{aligned}$$
we successively compute :
$$\begin{aligned}
\left\langle \vec{R}_z, \Phi_z \right\rangle &= 0\\
\left\langle \vec{R}_z, \Phi_{\zb} \right\rangle &= \frac{e^{2\lambda}}{2}\left( 2H + i \left\langle \Lr, \n \right\rangle \right) \\
\left\langle \vec{R}_z, \n \right\rangle &= -i S_z,
\end{aligned}$$
which proves the desired equality.
\tocless\subsection{Curvature formulas for branched immersions }
We first give a version of Gauss-Bonnet formula taking branch points and branched ends into account. We refer the reader to theorem 2.6 in \cite{biblammnguyen}.
\begin{prop}
\label{gaussbobonnet}
Let $\Sigma$ be a compact Riemann surface  and  $\Phi \, : \, \Sigma \rightarrow \R^3 \cup \{ \infty \}$ be a branched immersion with a finite number of ends.
Let $p_1, \dots, p_n$ be its branch points of respective orders $a_1+1, \dots, a_n+1$ and $q_1, \dots, q_m$  its ends of respective orders $b_1 -1, \dots b_m-1$.
We denote $\chi( \Sigma)$ the Euler characteristic of $\Sigma$, $g$ the metric induced on $\Sigma$ by $\Phi$ and $K$ its Gauss curvature. Hence 
$$ \int_\Sigma K d\mathrm{vol}_g = 2 \pi \left( \chi \left( \Sigma \right) + \sum_{i=1}^n a_i - \sum_{j=1}^m  b_j \right).$$
\end{prop}
Following are a few useful formulas linking the different notions of curvature. The computations are done in appendix  A.1 of \cite{bibnmheps} .
\begin{prop}
\label{laregledesegalitesdvjno}
Let $\Sigma$ be a compact Riemann surface and $\Phi \, \Sigma \rightarrow \R^3 \cup \{ \infty \}$ be a branched immersion with a finite number of ends. Let $g$ be the induced metric, $ \n$ be its Gauss map, $H$ its mean curvature, $ \Ar$ its tracefree second fundamental form and $K$ its Gauss curvature. Then 
$$\begin{aligned}
\int_\Sigma \left| \nabla \n \right|^2 d \mathrm{vol}_g &= 4 \int_\Sigma H^2  d \mathrm{vol}_g - 2 \int_\Sigma K  d \mathrm{vol}_g \\
&= 2 \int_\Sigma \big| \Ar \big|^2  d \mathrm{vol}_g + 2 \int_\Sigma K  d \mathrm{vol}_g.
\end{aligned}$$
\end{prop}

\addcontentsline{toc}{section}{Bibliography}
\bibliographystyle{plain}
\bibliography{bibliography}

\end{document}